\documentclass[a4paper, reqno]{amsart}

\usepackage{amsmath,amsthm,amssymb,relsize}
\usepackage[nobysame,alphabetic]{amsrefs}
\usepackage[margin=30mm]{geometry} 
\usepackage{mathrsfs,mathtools}

\usepackage{booktabs, multirow, adjustbox, array}

\usepackage{enumerate}

\usepackage[all]{xy}
\usepackage{paralist}

\usepackage{comment}
\usepackage{multirow}

\usepackage{tikz}
\usetikzlibrary{
  cd,
  calc,
  positioning,
  fit,
  arrows,
  decorations.pathreplacing,
  decorations.markings,
  shapes.geometric,
  backgrounds,
  bending
}
\usepackage{tikzsymbols}
\usepackage{xcolor}

\usepackage{tikz-cd}

\definecolor{darkblue}{rgb}{0,0,0.6}

\usepackage[breaklinks, pdftex, ocgcolorlinks,colorlinks=true, citecolor=darkblue, filecolor=darkblue, linkcolor=darkblue, urlcolor=black]{hyperref}
\usepackage[capitalize,noabbrev]{cleveref}

\makeatletter
\newtheorem*{rep@theorem}{\rep@title}
\newcommand{\newreptheorem}[2]{%
\newenvironment{rep#1}[1]{%
 \def\rep@title{#2 \ref{##1}}%
 \begin{rep@theorem}}%
 {\end{rep@theorem}}}
\makeatother

\newtheorem{proposition}{Proposition}[section]
\newtheorem{theorem}[proposition]{Theorem}
\newtheorem*{theorem*}{Theorem}
\newtheorem{corollary}[proposition]{Corollary}
\newtheorem{lemma}[proposition]{Lemma}

\newtheorem{thmx}{Theorem}

\crefname{thmx}{Theorem}{Theorems}

\theoremstyle{definition}
\newtheorem{definition}[proposition]{Definition}

\theoremstyle{remark}
\newtheorem{remark}[proposition]{Remark}

\newtheorem*{remark*}{Remark}

\newtheorem{assumptions}[proposition]{Hypothesis}

\newreptheorem{theorem}{Theorem}
\newreptheorem{lemma}{Lemma}
\newreptheorem{proposition}{Proposition}
\newreptheorem{corollary}{Corollary}
\newreptheorem{question}{Question}
\numberwithin{equation}{section}

\newcommand{\C}{\mathbb{C}}

\newcommand{\Q}{\mathbb{Q}}

\newcommand{\Z}{\mathbb{Z}}

\newcommand{\F}{\mathbb{F}}
\renewcommand{\mod}{\,\,\text{\normalfont mod}\,}
\renewcommand{\l}{\Lambda}
\newcommand{\wh}{\widehat}
\newcommand{\wt}{\widetilde}
\newcommand{\ol}{\overline}

\newcommand{\M}{\mathcal{M}}
\newcommand{\CC}{\mathcal{C}}

\DeclareMathOperator{\Aut}{Aut}

\DeclareMathOperator{\Hom}{Hom}
\DeclareMathOperator{\id}{Id}
\DeclareMathOperator{\Id}{Id}
\DeclareMathOperator{\rk}{rk}

\DeclareMathOperator{\image}{Im}
\DeclareMathOperator{\rank}{rk}

\DeclareMathOperator{\im}{Im}

\DeclareMathOperator{\coker}{coker}
\DeclareMathOperator{\Wh}{Wh}

\DeclareMathOperator{\GL}{GL}

\DeclareMathOperator{\CAT}{CAT}
\DeclareMathOperator{\TOP}{TOP}
\DeclareMathOperator{\PL}{PL}
\DeclareMathOperator{\Diff}{Diff}

\DeclareMathOperator{\odd}{odd}
\DeclareMathOperator{\IM}{Im}
\DeclareMathOperator{\hAut}{hAut}

\DeclareMathOperator{\colim}{colim}

\DeclareMathOperator{\ev}{ev}
\DeclareMathOperator{\Map}{Map}

\DeclareMathOperator{\PD}{PD}
\DeclareMathOperator{\interior}{Int}
\DeclareMathOperator{\hCob}{hCob}
\DeclareMathOperator{\Frob}{Frob}
\DeclareMathOperator{\ord}{ord}
\DeclareMathOperator{\Gal}{Gal}
\DeclareMathOperator{\disc}{disc}

\DeclareMathOperator{\st}{st}
\DeclareMathOperator{\nr}{nr}
\DeclareMathOperator{\Cyl}{Cyl}
\DeclareMathOperator{\E}{E}

\newcommand{\xrightarrowdbl}[2][]{%
  \xrightarrow[#1]{#2}\mathrel{\mkern-14mu}\rightarrow
}

\newcommand{\smfrac}[2]{\mbox{\footnotesize$\displaystyle\frac{#1}{#2}$}} 

\newcommand{\bsm}{\left(\begin{smallmatrix}}
\newcommand{\esm}{\end{smallmatrix}\right)}
\def\benum{\begin{clist}{(a)}}
\def\eenum{\end{clist}}

\newenvironment{clist}[1]
{\begin{enumerate}[\normalfont #1]}
{\end{enumerate}}

\usepackage{letltxmacro}

\LetLtxMacro\Oldfootnote\footnote

\begin{document}

\title{Simple homotopy types of even dimensional manifolds}

\author{Csaba Nagy}
\address{Max-Planck-Institut f\"{u}r Mathematik, Bonn, Germany}
\email{nagy@mpim-bonn.mpg.de}

\author{John Nicholson}
\address{School of Mathematics and Statistics, University of Glasgow, U.K.}
\email{john.nicholson@glasgow.ac.uk}

\author{Mark Powell}
\address{School of Mathematics and Statistics, University of Glasgow, U.K.}
\email{mark.powell@glasgow.ac.uk}


\begin{abstract}
Given a closed $n$-manifold, we consider the set of simple homotopy types of $n$-manifolds within its homotopy type, called its simple homotopy manifold set. We characterise it in terms of algebraic K-theory, the surgery obstruction map, and the homotopy automorphisms of the manifold. 
We use this to construct the first examples, for all $n \ge 4$ even, of closed $n$-manifolds that are homotopy equivalent but not simple homotopy equivalent. 
In fact, we construct infinite families of manifolds that are all homotopy equivalent but pairwise not simple homotopy equivalent, and our examples can be taken to be smooth for $n \geq 6$.

Our examples are homotopy equivalent to the product of a circle and a lens space. 
We analyse the simple homotopy manifold sets of these manifolds, determining exactly when they are trivial, finite, or infinite, and investigating their asymptotic behaviour.
The proofs involve integral representation theory and class numbers of cyclotomic fields. 
We also compare with the relation of $h$-cobordism, and produce similar detailed  quantitative descriptions of the manifold sets that arise. 
\end{abstract}

\maketitle

\vspace{-5mm}

\section{Introduction}

\subsection{Background and main results}

One of the earliest triumphs in manifold topology was the classification of 3-dimensional lens spaces up to homotopy equivalence, and up to homeomorphism, due to Seifert-Threlfall, Reidemeister,  Whitehead, and Moise~\cites{Threlfall-Seifert,Reidemeister-lens-spaces,Whitehead-incidence,Moise}; see~\cite{Co73} for a self-contained treatment and \cite{Volkert-manifold-atlas} for a detailed history.   
The two classifications do not coincide e.g.\ $L(7,1)$ and $L(7,2)$ are famously homotopy equivalent but not homeomorphic. 

The homeomorphism classification made use of Reidemeister torsion to distinguish homotopy equivalent lens spaces. While trying to understand this more deeply, J.H.C.~Whitehead defined the notion of \emph{simple homotopy equivalence}~\cites{Whitehead-simplicial-spaces,Whitehead-incidence,Whitehead-combinatorial,Whitehead-she}. A homotopy equivalence $f \colon X \to Y$ between CW complexes is said to be \emph{simple} if it is homotopic to the composition of a sequence of elementary expansions and collapses.  
He also defined the \emph{Whitehead group}~\cite{Whitehead-she} $\Wh(G)$ of a group $G$, and  the \emph{Whitehead torsion} $\tau(f) \in \Wh(\pi_1(Y))$.  
 He proved the fundamental result that $\tau(f) =0$ if and only if $f$ is simple.
See \cref{section:prelims} for the precise definition.

Chapman~\cite{Ch74} showed that every homeomorphism $f \colon X \to Y$ between compact CW complexes is a simple homotopy equivalence, and so we have:
\[\mbox{homeomorphism} \Rightarrow \mbox{simple homotopy equivalence} \Rightarrow \mbox{homotopy equivalence}.\]
Whitehead showed that the homeomorphism classification of 3-dimensional lens spaces coincides with the classification up to simple homotopy equivalence, so there are many examples of homotopy but not simple homotopy equivalent lens spaces, e.g.\ $L(7,1)$ and $L(7,2)$. Higher dimensional lens spaces give rise to similar examples in all odd dimensions $\geq 5$~\cite{Co73}, and infinite such families of odd-dimensional manifolds were produced by Jahren-Kwasik~\cite{JK15}.  

This article constructs the first examples of closed manifolds, in all even dimensions $2k \geq 4$, that are homotopy but not simple homotopy equivalent. In fact, we produce infinite families. 
Writing~$\simeq$ for homotopy equivalence and~$\simeq_{\mathrm{s}}$ for simple homotopy equivalence, we note that if $M_1 \simeq N_1$ and $M_2 \simeq N_2$ are two pairs of homotopy equivalent (but not necessarily simple homotopy equivalent) odd dimensional manifolds, then $M_1 \times M_2 \simeq_{\mathrm{s}} N_1 \times N_2$ (\cref{cor:odd-product}), so even dimensional examples cannot be constructed in such a straightforward way from odd dimensional examples.

\begin{thmx}\label{main-theorem}
Let $n \geq 4$ be even. Then there exists an infinite collection of closed, connected, orientable $n$-manifolds that are all homotopy equivalent but are pairwise not simple homotopy equivalent. If $n > 4$, these manifolds can be taken to be smooth, and if $n=4$, they are topological.
\end{thmx}

Our examples are manifolds homotopy equivalent to $S^1 \times L$, for certain lens spaces $L$.

    It is currently open whether all topological 4-manifolds are homeomorphic to CW complexes. Thus, the definition of simple homotopy equivalence given previously does not apply.
    For a more general definition that applies in this case, see \cref{defn:KS-defn-simple}.

\subsection{Simple homotopy manifold sets}\label{ss:simple-hom-manifold-sets}

From now on, $n \ge 4$ will be an integer and all manifolds will be assumed closed and connected.
In order to quantify the difference between homotopy and simple homotopy equivalence, for a $\CAT$ $n$-manifold $M$ we introduce the \textit{simple homotopy manifold set} 
\[
\M^{\mathrm{h}}_{\mathrm{s}}(M)  :=  \{\CAT \text{ $n$-manifolds } N \mid N \simeq M\}/\simeq_{\mathrm{s}}
\]
where $\CAT \in \{\Diff, \PL, \TOP\}$ is the category of either smooth, piecewise linear, or topological manifolds.
 This is the set of manifolds homotopy equivalent to $M$ up to simple homotopy equivalence. 
The term ``manifold set" is used in contrast to structure sets: an element of $\M^{\mathrm{h}}_{\mathrm{s}}(M)$ is represented by a manifold without a chosen homotopy equivalence to $M$.

We will frequently restrict to the case of $\CAT$ $n$-manifolds where $\CAT$ satisfies the following hypothesis which depends on $n$.

\begin{assumptions}\label{assumptions}
If $n>4$, we take $\CAT \in \{\Diff, \PL, \TOP\}$.
If $n=4$, we take $\CAT$ to be the full subcategory of $\TOP$ with objects $4$-manifolds~$M$ such that $\pi_1(M)$ is good in the sense of Freedman.
\end{assumptions}

\begin{remark}
For good groups, the surgery sequence for 4-manifolds is defined and exact, and the $s$-cobordism theorem is available (see \cref{ss:simple-hom-manifold-sets}). The class of good groups contains finite groups and solvable groups, and is closed under subgroups, quotients, extensions and direct limits. For more details and references see \cites{FQ,Freedman-book-goodgroups}.
We have $\pi_1(S^1 \times L) \cong C_\infty \times C_m$ where $C_\infty$ is infinite cyclic and $C_m$ is cyclic of order $m \ge 2$. Since $C_\infty \times C_m$ is good, \cref{assumptions} applies in the case of \cref{main-theorem}.
\end{remark}

To study simple homotopy manifold sets, we will make use of a further equivalence relation, $h$-cobordism.
A cobordism $(W;M,N)$ of closed manifolds is an \emph{$h$-cobordism} if the inclusion maps $M \to W$ and $N \to W$ are homotopy equivalences. The notion of an $h$-cobordism is central to manifold topology.  
Smale's $h$-cobordism theorem~\cites{smale-pe,Smale-h-cob,Milnor-h-cob}, together with its extensions to other categories and dimension 4 in \cite{Stallings-Tata,Kirby-Siebenmann:1977-1,FQ}, states that under \cref{assumptions}, every simply-connected $h$-cobordism is $\CAT$-equivalent to the product $M \times I$. 
To generalise the $h$-cobordism theorem to non-simply connected manifolds, one also considers $s$-cobordisms. An \emph{$s$-cobordism} $(W;M,N)$ is a cobordism for which the inclusions $M \to W$ and $N\to W$ are simple homotopy equivalences. The $\CAT$ $s$-cobordism theorem~\cites{Barden,Mazur:1963-1,Stallings-Tata,Kirby-Siebenmann:1977-1,FQ} states that under \cref{assumptions},  $\CAT$ equivalence classes of $h$-cobordisms based on $M$ are in bijection with $\Wh(\pi_1(M))$, with the bijection given by taking the Whitehead torsion of the inclusion $M \to W$ (\cref{theorem:scob}). In particular every $s$-cobordism is a product.  This theorem underpins manifold classification in dimension at least $4$. 

The $s$-cobordism theorem is one of the tools that we can use to construct homotopy equivalent but not simple homotopy equivalent manifolds (see \cref{ss:realising-cobordisms,ss:class-man}), and the examples constructed this way are always $h$-cobordant. This leads us to consider two refined versions of the problem of finding homotopy but not simple homotopy equivalent manifolds, where the manifolds are also required to be $h$-cobordant, or required not to be $h$-cobordant. We introduce the corresponding variations of $\M^{\mathrm{h}}_{\mathrm{s}}(M)$: 
\[
\begin{aligned}
\M^{\hCob}_{\mathrm{s}}(M) & :=  \left\{\CAT \text{ $n$-manifolds } N \mid N \text{ is $h$-cobordant to } M \right\} / \simeq_{\mathrm{s}} \\
\M^{\mathrm{h}}_{\mathrm{s},\hCob}(M) & :=  \left\{ \CAT \text{ $n$-manifolds } N \mid N \simeq M \right\} / \langle \simeq_{\mathrm{s}}, \hCob \rangle
\end{aligned}
\]
where $\langle \simeq_{\mathrm{s}}, \hCob \rangle$ denotes the equivalence relation generated by simple homotopy equivalence and $h$-cobordism. These sets
arise naturally in the computation of $\M^{\mathrm{h}}_{\mathrm{s}}(M)$ in \cref{thmx:bijections-with-manifold-sets-intro} below. 

In \cref{thmx:bijections-with-manifold-sets-intro} below, we characterise the simple homotopy manifold sets $\M^{\mathrm{h}}_{\mathrm{s}}(M)$, $\M^{\hCob}_{\mathrm{s}}(M)$, and $\M^{\mathrm{h}}_{\mathrm{s},\hCob}(M)$, 
for a $\CAT$ $n$-manifold $M$, in terms of the Whitehead group, the homotopy automorphisms of $M$, and the surgery obstruction map. We then apply these characterisations to the manifolds $S^1 \times L$ and compute the objects that appear in \cref{thmx:bijections-with-manifold-sets-intro} for this case.

We first introduce notation which appears in the statement of the theorem.
An orientation character $w \colon G \rightarrow \left\{ \pm 1 \right\}$ determines an involution $x \mapsto \ol{x}$ on the Whitehead group $\Wh(G)$ (see \cref{ss:wh-group}), and we write $\Wh(G,w)$ to specify that $\Wh(G)$ is equipped with this involution. If $w \equiv 1$ is the trivial character, then we omit it from the notation. 
Define: 
\[
\begin{aligned}
\mathcal{J}_n(G,w) &= \left\{ y \in \Wh(G,w) \mid y = -(-1)^n\ol{y} \right\} \leq \Wh(G,w), \\
\mathcal{I}_n(G,w) &= \left\{ x - (-1)^n\ol{x} \mid x \in \Wh(G,w) \right\} \leq \mathcal{J}_n(G,w).
\end{aligned}
\] 
The Tate cohomology group $\wh{H}^{n+1}(C_2;\Wh(G,w))$ is canonically identified with $\mathcal{J}_n(G,w) / \mathcal{I}_n(G,w)$ (see \cref{prop:tate-C2}), and we denote the quotient map by 
\[
\pi \colon \mathcal{J}_n(G,w) \rightarrow \wh{H}^{n+1}(C_2;\Wh(G,w)) .
\]
Next let \[\varrho \colon \wh{H}^{n+1}(C_2;\Wh(G,w)) \rightarrow L_n^{\mathrm{s}}(\Z G, w)\] be the map from the Ranicki-Rothenberg exact sequence \eqref{eq:SRR-sequence} (see \cite{Shaneson-GxZ}, \cite[\S9]{Ranicki-80-I}). 

For $M$ a $\CAT$ $n$-manifold with fundamental group $G$ and orientation character $w$, let 
\[ \sigma_{\mathrm{s}} \colon \mathcal{N}(M) \rightarrow L_n^{\mathrm{s}}(\Z G, w)\] 
be the surgery obstruction map (see \cref{ss:wall}). 
The homotopy automorphisms $\hAut(M)$ of $M$ is the group of homotopy classes of self-homotopy equivalences of~$M$.
There is an action of $\hAut(M)$ on $\Wh(G,w)$ (as a set) such that if $f \colon N \rightarrow M$ is a homotopy equivalence and $g \in \hAut(M)$, then $\tau(f)^g=\tau(g \circ f)$. Denote the  quotient map by
\[
q \colon \Wh(G,w) \rightarrow \Wh(G,w) / \hAut(M).
\]
The subset $\mathcal{J}_n(G,w)$ is invariant under the action of $\hAut(M)$ on $\Wh(G,w)$, and there is an induced action on $\wh{H}^{n+1}(C_2;\Wh(G,w))$. The subsets
\[
(\varrho \circ \pi)^{-1}(\image \sigma_{\mathrm{s}}) \subseteq \mathcal{J}_n(G,w)  \quad\text{ and }\quad 
 \varrho^{-1}(\image \sigma_{\mathrm{s}}) \subseteq \wh{H}^{n+1}(C_2;\Wh(G,w))
\]
are both invariant under the induced actions of $\hAut(M)$. In general $\mathcal{I}_n(G,w)$ is not invariant under the action of $\hAut(M)$ on $\Wh(G,w)$, but when it is, then $q(\mathcal{I}_n(G,w)) = \mathcal{I}_n(G,w)/\hAut(M)$.

The following theorem is the basis of our main results (see \cref{theorem:man-class,theorem:man-bij}). 

\begin{thmx}\label{thmx:bijections-with-manifold-sets-intro} 
Let $M$ be a $\CAT$ $n$-manifold with fundamental group $G$ and orientation character $w \colon G \rightarrow \left\{ \pm 1 \right\}$, with $\CAT$ as in \cref{assumptions}. There is a commutative diagram
\[
\xymatrix{
\M^{\hCob}_{\mathrm{s}}(M) \ar[r] \ar[d]^{\cong} & \M^{\mathrm{h}}_{\mathrm{s}}(M) \ar[r] \ar[d]^{\cong} & \M^{\mathrm{h}}_{\mathrm{s},\hCob}(M) \ar[d]^{\cong} \\
q(\mathcal{I}_n(G,w)) \ar[r] & (\varrho \circ \pi)^{-1}(\image \sigma_{\mathrm{s}}) / \hAut(M) \ar[r] & \varrho^{-1}(\image \sigma_{\mathrm{s}}) / \hAut(M)
}
\] 
where each row is a short exact sequence of pointed sets, and each vertical arrow is a bijection. 
\end{thmx}

It follows from the exact sequence \eqref{eq:SRR-sequence} that the image of the homomorphism $\psi \colon L_{n+1}^{\mathrm{h}}(\Z G, w) \rightarrow \wh{H}^{n+1}(C_2;\Wh(G,w))$ is $\varrho^{-1}(\{0\})$, hence it is contained in $\varrho^{-1}(\image \sigma_{\mathrm{s}})$, and so we obtain the following corollary. 

\begin{corollary} \label{cor:psi-surj-bij-intro}
Let $M$ be a $\CAT$ $n$-manifold with fundamental group $G$ and orientation character $w \colon G \rightarrow \left\{ \pm 1 \right\}$, with $\CAT$ as in \cref{assumptions}. If $\psi \colon L_{n+1}^{\mathrm{h}}(\Z G, w) \rightarrow \wh{H}^{n+1}(C_2;\Wh(G,w))$ is surjective, then we have
\[
\M^{\mathrm{h}}_{\mathrm{s}}(M) \cong \mathcal{J}_n(G,w) / \hAut(M), \quad \quad   \M^{\mathrm{h}}_{\mathrm{s},\hCob}(M) \cong \wh{H}^{n+1}(C_2;\Wh(G,w)) / \hAut(M).
\]
\end{corollary}

The advantage of \cref{cor:psi-surj-bij-intro} is that $\mathcal{J}_n(G,w)$ and $\wh{H}^{n+1}(C_2;\Wh(G,w))$ (and, more generally, $\image \psi$) only depend on $G$ and $w$, while $\image \sigma_{\mathrm{s}}$ depends on $M$ a priori. This allows us to apply \cref{cor:psi-surj-bij-intro} by separately analysing the involution on $\Wh(G,w)$ and the homotopy automorphisms of $M$. In cases where $M$ can be regarded as a manifold in multiple categories, the sets $q(\mathcal{I}_n(G,w))$, $\mathcal{J}_n(G,w) / \hAut(M)$ and $\wh{H}^{n+1}(C_2;\Wh(G,w)) / \hAut(M)$ do not depend on the choice of $\CAT$.

We can also use \cref{thmx:bijections-with-manifold-sets-intro} to describe when one of the simple homotopy manifold sets is nontrivial, see \cref{prop:hcob-not-she,prop:he-not-she+hcob,prop:he-not-she-equiv}. We introduce the following notation: 
\[
\begin{aligned}
T(M) &= \left\{ \tau(g) \mid g \in \hAut(M) \right\} \subseteq \mathcal{J}_n(G,w) \\
U(M) &= \left\{ \pi(\tau(g)) \mid g \in \hAut(M) \right\} \subseteq \wh{H}^{n+1}(C_2;\Wh(G,w)) 
\end{aligned}
\]

\begin{corollary} 
Let $M$ be a $\CAT$ $n$-manifold with fundamental group $G$ and orientation character $w \colon G \rightarrow \left\{ \pm 1 \right\}$, with $\CAT$ as in \cref{assumptions}. Then
\begin{clist}{(a)}
\item $|\M^{\hCob}_{\mathrm{s}}(M)| > 1$ if and only if $\mathcal{I}_n(G,w) \setminus T(M)$ is nonempty.
\item $|\M^{\mathrm{h}}_{\mathrm{s}}(M)| > 1$ if and only if $(\varrho \circ \pi)^{-1}(\image \sigma_{\mathrm{s}}) \setminus T(M)$ is nonempty.
\item $|\M^{\mathrm{h}}_{\mathrm{s},\hCob}(M)| > 1$ if and only if $\varrho^{-1}(\image \sigma_{\mathrm{s}}) \setminus U(M)$ is nonempty.
\end{clist}
\end{corollary}

To prove \cref{thmx:bijections-with-manifold-sets-intro}, first we observe that if $N$ is a manifold homotopy equivalent to $M$, then the set of Whitehead torsions of all possible homotopy equivalences $f \colon N \to M$ forms an orbit of the action of $\hAut(M)$ on $\Wh(G,w)$. By considering the different restrictions and equivalence relations on these manifolds $N$, we obtain maps from the various simple homotopy manifold sets of $M$ to subsets/subquotients of $\Wh(G,w) / \hAut(M)$. We then verify that these maps are injective, and the remaining task is to determine their images. 

For $\M^{\hCob}_{\mathrm{s}}(M)$, we use the $s$-cobordism theorem. For every $x \in \Wh(G,w)$ there is an $h$-cobordism $(W;M,N)$ with Whitehead torsion $x$, and then $\tau(f) = -x + (-1)^n \ol{x}$ for the induced homotopy equivalence $f \colon N \rightarrow M$ (see \cref{prop:WT-hcob}). This shows that every element of $\mathcal{I}_n(G,w)$ can be realised as the torsion of a homotopy equivalence $N \rightarrow M$ for some $N$ that is $h$-cobordant to $M$. 
For $\M^{\mathrm{h}}_{\mathrm{s},\hCob}(M)$, we need to describe the set of values of $\pi(\tau(f))$ for all homotopy equivalences $f \colon N \to M$. This set is the image of a natural map $\widehat{\tau} \colon \mathcal{S}^{\mathrm{h}}(M) \to \wh{H}^{n+1}(C_2;\Wh(G, w))$ from the homotopy structure set of $M$ (see \cref{def:hsset}). We show that $\widehat{\tau}$ fits into a commutative diagram involving the surgery exact sequence and the Ranicki-Rothenberg exact sequence~\eqref{eq:SRR-sequence} (see \cref{prop:braid}), and use that diagram to show that its image is $\varrho^{-1}(\image \sigma_{\mathrm{s}})$.
Finally, the characterisation of $\M^{\mathrm{h}}_{\mathrm{s}}(M)$ is obtained by combining the results on $\M^{\hCob}_{\mathrm{s}}(M)$ and~$\M^{\mathrm{h}}_{\mathrm{s},\hCob}(M)$.

\subsection{Computations of simple homotopy manifold sets}\label{ss:computations-of-s-h-man-sets}

We prove that for manifolds of the form $M = S^1 \times L$, the size of $\M^{\mathrm{h}}_{\mathrm{s}}(M)$ can be trivial,  finite and arbitrarily large, or infinite, and we determine precisely when each eventuality occurs. In particular the existence of the latter case implies \cref{main-theorem}. For an integer $m \ge 2$, let $C_m$ denote the cyclic group of order $m$, and let $C_\infty$ denote the infinite cyclic group. 

\begin{thmx} \label{theorem:main-S^1xL}
Let $n \geq 4$ be even and let $\CAT$ be as in \cref{assumptions}.  
Let $L$ be an $(n-1)$-dimensional lens space with $\pi_1(L) \cong C_m$, for some $m \geq 2$, and  let $M^n_m = S^1 \times L$.  Then 
\begin{clist}{(a)}
\item\label{item:thm-1-i} $|\M^{\mathrm{h}}_{\mathrm{s}}(M^n_m)|$ only depends on $n$ and $m$, and is independent of the choice of $L$ or $\CAT$;
\item\label{item:thm-1-ii} $|\M^{\mathrm{h}}_{\mathrm{s}}(M^n_m)| = 1$ if and only if $m \in \{2,3,5,6,7,10,11,13,14,17,19\}$;
\item\label{item:thm-1-iii} $|\M^{\mathrm{h}}_{\mathrm{s}}(M^n_m)| = \infty$ if and only if $m$ is not square-free;
\item\label{item:thm-1-iv} $|\M^{\mathrm{h}}_{\mathrm{s}}(M^n_m)| \to \infty$ as $m \to \infty$, uniformly in $n$.
\end{clist}
\end{thmx}

The proof is an application of \cref{thmx:bijections-with-manifold-sets-intro}.
This theorem identifies $\M^{\mathrm{h}}_{\mathrm{s}}(M)$ with the orbit set of an action of $\hAut(M)$ on a subgroup of $\Wh(\pi_1(M))$, which depends on the surgery obstruction map in the surgery exact sequence of $M$. 
However as in \cref{cor:psi-surj-bij-intro}, for some fundamental groups, including $C_{\infty} \times C_m$, the relevant subgroup of $\Wh(\pi_1(M))$ depends only on $\pi_1(M)$ and the orientation character of $M$. To study this subgroup for $S^1 \times L$, we use representation theory and number theory, as explained in \cref{ss:intro-k0}. 

The action of an element $g \in \hAut(M)$ on $\Wh(\pi_1(M))$ is determined by its Whitehead torsion $\tau(g) \in \Wh(\pi_1(M))$ and the induced homomorphism $\pi_1(g) \in \Aut(\pi_1(M))$ (see \cref{def:haut-act}).
It turns out that computing this action is particularly tractable for $S^1 \times L$; see \cref{theorem:haut-S1xL-intro} for more details. This explains why we are able to so accurately compute simple homotopy manifold sets for this class of manifolds. 

The next theorem gives examples where $\M^{\hCob}_{\mathrm{s}}(M)$ is nontrivial. Thus, even if we restrict to $h$-cobordant manifolds, many of the simple homotopy manifold sets we consider remain large. 

\begin{thmx}\label{theorem:hcob-S^1xL}
Let $n \geq 4$ be even and let $\CAT$ be as in \cref{assumptions}. 
Let $L$ be an $(n-1)$-dimensional lens space with $\pi_1(L) \cong C_m$, for some $m \geq 2$, and  let $M^n_m = S^1 \times L$. Then 
\begin{clist}{(a)}
\item\label{item:theorem-hcob-S1xL-a} $|\M^{\hCob}_{\mathrm{s}}(M^n_m)|$ only depends on $n$ and $m$, but it is independent of the choice of $L$ or $\CAT$;
\item\label{item:theorem-hcob-S1xL-b} $|\M^{\hCob}_{\mathrm{s}}(M^n_m)| = 1$ if and only if  $m \in \{2,3,5,6,7,10,11,13,14,15,17,19,29\}$;
\item\label{item:theorem-hcob-S1xL-c} $|\M^{\hCob}_{\mathrm{s}}(M^n_m)| = \infty$ if and only if $m$ is not square-free;
\item\label{item:theorem-hcob-S1xL-d} $|\M^{\hCob}_{\mathrm{s}}(M^n_m)| \to \infty$ as $m \to \infty$ uniformly in $n$.
\end{clist}
\end{thmx}

Finally, we discover a wealth of interesting behaviour when we also factor out by $h$-cobordism.

\begin{thmx}\label{theorem:s-hcob-S^1xL}
Let $n \geq 4$ be even and let $\CAT$ be as in \cref{assumptions}. 
Let $L$ be an $(n-1)$-dimensional lens space with $\pi_1(L) \cong C_m$, for some $m \geq 2$, and  let $M^n_m = S^1 \times L$. 
Then 
\begin{clist}{(a)}
\item\label{item-theorem:s-hcob-S^1xL-a} $|\M^{\mathrm{h}}_{\mathrm{s},\hCob}(M^n_m)|$ only depends on $n$ and $m$, but it is independent of the choice of $L$ or $\CAT$;
\item\label{item-theorem:s-hcob-S^1xL-b} $\displaystyle \liminf_{m \to \infty} \Bigl( \sup_n |\M^{\mathrm{h}}_{\mathrm{s},\hCob}(M^n_m)| \Bigr) = 1$;
\item\label{item-theorem:s-hcob-S^1xL-c} $|\M^{\mathrm{h}}_{\mathrm{s},\hCob}(M^n_m)| < \infty$ for all $n$ and $m$;
\item\label{item-theorem:s-hcob-S^1xL-d} $\displaystyle \limsup_{m \to \infty} \Bigl( \inf_n |\M^{\mathrm{h}}_{\mathrm{s},\hCob}(M^n_m)| \Bigr) = \infty$.
\end{clist}
\end{thmx}

Part \eqref{item-theorem:s-hcob-S^1xL-b} says that there are infinitely many $m$ with $|\M^{\mathrm{h}}_{\mathrm{s},\hCob}(M^n_m)| = 1$
for every $n$, while part \eqref{item-theorem:s-hcob-S^1xL-d} says that $\inf_n |\M^{\mathrm{h}}_{\mathrm{s},\hCob}(M^n_m)|$ is unbounded in $m$.

\subsection{Outline of the proof of \cref{main-theorem}}
\label{ss:intro-S1L}

Let $n = 2k \ge 4$ be even and let $M^n_m = S^1 \times L$, where $L$ is an $(n-1)$-dimensional lens space with $\pi_1(L) \cong C_m$, $m \geq 2$. It suffices to show that $|\M^{\mathrm{h}}_{\mathrm{s}}(M^n_m)| = \infty$ for some $m \ge 2$.
The proof is based on applying \cref{cor:psi-surj-bij-intro} to $M^n_m$. 
We show that $\psi \colon L_{n+1}^{\mathrm{h}}(\Z[C_\infty \times C_m]) \rightarrow \wh{H}^{n+1}(C_2;\Wh(C_\infty \times C_m))$ is surjective for $n$ even in \cref{prop:psi-surjective}, and so \cref{cor:psi-surj-bij-intro} implies that $\M^{\mathrm{h}}_{\mathrm{s}}(M^n_m) \cong \mathcal{J}_n(C_\infty \times C_m) / \hAut(M^n_m)$.  
For the homotopy automorphisms of $M^n_m$, we use the following (see \cref{theorem:wh-torsion-vanishes-of-he-s,theorem:aut-image}).

\begin{theorem}\label{theorem:haut-S1xL-intro}
\leavevmode
\begin{clist}{(a)}
\item\label{item:theorem-haut-S1xL-intro-a} Every homotopy automorphism $f \colon M^n_m \to M^n_m$ is simple.
\item\label{item:theorem-haut-S1xL-intro-b}  
If $\pi_1 \colon \hAut(M^n_m) \rightarrow \Aut(C_\infty \times C_m)$ is the map given by taking the induced automorphism on the fundamental group, then $\IM(\pi_1) = \left\{ \left(\begin{smallmatrix} a & b \\ 0 & c \end{smallmatrix}\right) \in \Aut(C_\infty \times C_m) \mid c^k \equiv \pm 1 \text{ \normalfont mod } m \right\}$.
\end{clist}
\end{theorem}

Part \eqref{item:theorem-haut-S1xL-intro-a} implies that the action of $\hAut(M^n_m)$ on $\Wh(C_\infty \times C_m)$ factors through the action of $\Aut(C_\infty \times C_m)$, which acts on $\Wh(C_\infty \times C_m)$ via automorphisms, using the functoriality of $\Wh$ (see \cref{def:haut-act,rem:haut-act-spec}). Therefore every orbit has cardinality bounded above by $|\hspace{-0.5mm}\Aut(C_\infty \times C_m)| < 2m^2$, and $|\mathcal{J}_n(C_\infty \times C_m) / \hAut(M^n_m)|=1$ if and only if $\mathcal{J}_n(C_\infty \times C_m) = 0$. 

So to prove \cref{main-theorem}, and more generally \cref{theorem:main-S^1xL}, it remains to study the involution on $\Wh(C_\infty \times C_m)$ and prove the corresponding statements about $|\mathcal{J}_n(C_\infty \times C_m)|$. First, the fundamental theorem for $K_1(\Z C_m[t,t^{-1}])$ gives rise to a direct sum decomposition (see \cref{theorem:BHS}):
\[
\Wh(C_{\infty} \times C_m) \cong \Wh(C_m) \oplus \wt{K}_0(\Z C_m) \oplus NK_1(\Z C_m)^2
\]
where $\wt{K}_0$ is the reduced projective class group, and $NK_1$ is the so-called Nil group. All summands have natural involutions, which are compatible with this isomorphism. Using that $\mathcal{J}_n(C_m) = 0$ (see \cref{prop:j0cm}), we obtain the following decomposition for $\mathcal{J}_n(C_\infty \times C_m)$ (see \cref{prop:I(ZxG)}):
\[\mathcal{J}_n(C_\infty \times C_m) \cong \{ y \in \wt{K}_0(\Z C_m) \mid y=-\ol{y} \} \oplus NK_1(\Z C_m).\]

\begin{theorem}[Bass-Murthy, Martin, Weibel, Farrell] \label{theorem:NK_1=0}
If $m$ is square-free then $NK_1(\Z C_m)=0$. Otherwise $NK_1(\Z C_m)$ is infinite.  
\end{theorem}

In particular $\mathcal{J}_n(C_\infty \times C_m)$ is infinite in the case where $m$ is not square-free. Since each orbit of $\hAut(M^n_m)$ is finite, there are infinitely many orbits. Since $\M^{\mathrm{h}}_{\mathrm{s}}(M^n_m) \cong \mathcal{J}_n(C_\infty \times C_m) / \hAut(M^n_m)$, we obtain \cref{main-theorem} by taking $m$ not square-free, e.g.\ $m=4$.

The proofs of Theorems \ref{theorem:hcob-S^1xL} and \ref{theorem:s-hcob-S^1xL} are similarly approached using \cref{cor:psi-surj-bij-intro} and \cref{theorem:haut-S1xL-intro}, as well as the following decompositions (see \cref{prop:I(ZxG),cor:tate-decomp}):
\begin{align*} &\mathcal{I}_n(C_\infty \times C_m) \cong \{x-\ol{x} \mid x \in \wt K_0(\Z C_m)\} \oplus NK_1(\Z C_m) \\
&\wh{H}^{n+1}(C_2;\Wh(C_\infty \times C_m)) \cong \wh{H}^{n+1}(C_2;\wt{K}_0(\Z C_m)) \cong \frac{\{y \in \wt K_0(\Z C_m) \mid \overline{y} = -(-1)^ny \}}{\{x-(-1)^n\ol{x} \mid x \in \wt K_0(\Z C_m)\}}.\end{align*}
The finiteness results follow from \cref{theorem:NK_1=0} combined with the fact that $\wt K_0(\Z C_m)$ is finite (see \cref{prop:proj=>LF}). However, the results on triviality and asymptotic behaviour of the manifold sets requires a deeper analysis of the involution on $\wt K_0(\Z C_m)$ (see \cref{theorem:main12-red-intro,theorem:main3-red-intro} below).

\subsection{The involution on $\wt{K}_0(\Z C_m)$} \label{ss:intro-k0}

The purpose of \cref{p:algebra} is to explore this involution on $\wt K_0(\Z C_m)$ in detail and to prove the following three results which are the key algebraic ingredients behind the proofs of \cref{theorem:main-S^1xL,theorem:hcob-S^1xL,,theorem:s-hcob-S^1xL}.

\begin{theorem} \label{theorem:main12-red-intro}
Let $m \ge 2$ be a square-free integer. Then 
\begin{clist}{(i)}
\item\label{item:theorem:main12-red-intro-i} 
$|\{ y \in \wt{K}_0(\Z C_m) \mid \overline{y}=-y\}|=1$ if and only if  $m \in \{2, 3, 5, 6, 7, 10, 11, 13, 14, 17, 19\}$;
\item\label{item:theorem:main12-red-intro-ii} 
$|\{ x - \overline{x} \mid x \in \wt{K}_0(\Z C_m)\}|=1$ if and only if  $m \in \{2, 3, 5, 6, 7, 10, 11, 13, 14, 15, 17, 19, 29\}$; and
\item\label{item:theorem:main12-red-intro-iii} 
$|\{ x - \overline{x} \mid x \in \wt{K}_0(\Z C_m)\}| \to \infty$ super-exponentially in $m$, and hence we also have that $|\{ y \in \wt{K}_0(\Z C_m) \mid \overline{y}=-y\}| \to \infty$ super-exponentially in $m$.
\end{clist}
\end{theorem}

\begin{theorem} \label{theorem:main3-red-intro}
Let $m \ge 2$ be an integer. Then
\begin{clist}{(i)}
\item\label{item:theorem:main3-red-intro-i} $|\{ y \in \wt{K}_0(\Z C_m) \mid \overline{y}=-y\}/\{ x - \overline{x} \mid x \in \wt{K}_0(\Z C_m)\}|=1$  for infinitely many $m$; and
\item\label{item:theorem:main3-red-intro-ii} $\displaystyle \sup_{k \le m} |\{ y \in \wt{K}_0(\Z C_k) \mid \overline{y}=-y\}/\{ x - \overline{x} \mid x \in \wt{K}_0(\Z C_k)\}| \to \infty$  exponentially in $m$. 
\end{clist}
\end{theorem}

We now explain the strategy of proof of these three results, as well as some of the key ingredients. 
Firstly, whilst $\wt K_0(R)$ is difficult to compute for an arbitrary ring, for a finite group $G$ we have that $\wt K_0(\Z G) \cong C(\Z G)$ where $C(\cdot)$ denotes the locally free class group (see \cref{ss:class-groups-basics}).
In the case $G = C_m$, this allows us to obtain a short exact sequence:
\begin{equation} 0 \to D(\Z C_m) \to \wt K_0(\Z C_m) \to \bigoplus_{d \mid m} C(\Z[\zeta_d])\to 0 \label{eq:SES-K_0} \end{equation}
where $D(\Z C_m)$ denotes the kernel group of $\Z C_m$ (see \cref{ss:kernel-groups}) and $C(\Z[\zeta_d])$ denotes the ideal class group of $\Z[\zeta_d]$, where $\zeta_d := e^{2\pi i/d} \in \C$. The standard involution on $\wt K_0(\Z C_m)$ preserves $D(\Z C_m)$ and induces the involution given by conjugation on each $C(\Z[\zeta_d])$ (see \cref{ss:induced-inv}).

To prove \cref{theorem:main12-red-intro}, we make use of \cref{lemma:useful,lemma:very-useful} which allow us to obtain information about the orders $|\{ x \in \wt{K}_0(\Z C_m) \mid \overline{x}=-x\}|$ and $|\{ x - \overline{x} \mid x \in \wt{K}_0(\Z C_m)\}|$ from the orders of $|\{ x \in A \mid \overline{x}=-x\}|$ and $|\{ x - \overline{x} \mid x \in A\}|$ for $A = D(\Z C_m)$ or $C(\Z[\zeta_d]))$ for some $d \mid m$. 
The key result making this approach possible is the following, which is \cref{prop:inv-ideal-class-1}. Here $h_m^{-}$ is a factor of the class number $h_m = |C(\Z[\zeta_m])|$ (see \cref{def:class-number}).
For an integer $m$, we let $\odd(m)$ denote the unique odd integer $r$ such that $m = 2^k r$ for some $k$.

\begin{proposition} \label{lemma:intro1}
The integer $\odd(h_m^{-})$ divides $|\{x-\ol{x} \mid x \in C(\Z[\zeta_m])\}|$.
\end{proposition}

It was shown by Horie \cite{Ho89}, using results from Iwasawa theory \cite{Fr82}, that there exists finitely many $m$ for which $\odd(h_m^{-})=1$ and $\odd(h_m^{-}) \to \infty$ (see \cref{prop:odd(h_m)=1}).
Since \cref{lemma:intro1} also gives a lower bound on $ |\{x \in C(\Z[\zeta_m]) \mid \ol{x} = -x \}|$, this is enough to prove part \eqref{item:theorem:main12-red-intro-iii} of \cref{theorem:main12-red-intro} and to reduce the proof of parts \eqref{item:theorem:main12-red-intro-i} and \eqref{item:theorem:main12-red-intro-ii} to checking finitely many cases. These cases are dealt with via a variety of methods such as analysing the group structure on $C(\Z[\zeta_m])$ (see the proof of \cref{prop:part1}) and the studying the involution on $D(\Z C_m)$ by relating it to maps between units groups (see \cref{ss:kernel-groups-ZC_m,ss:proofs-algebra-1/2}).

The proof of \cref{theorem:main3-red-intro} can similarly be approached by using \eqref{eq:SES-K_0} to relate the Tate cohomology group $\wh H^1(C_2 ; \wt K_0(\Z C_m))$ to the parity of class numbers. This can be found in \cref{ss:proofs-algebra-3}.

\subsection*{Organisation of the paper}

The paper is structured into three parts. 
In \cref{p:general} we develop the necessary background on simple homotopy equivalence, Whitehead torsion and $h$-cobordisms. We then prove \cref{thmx:bijections-with-manifold-sets-intro}, which is our main general result. 
In \cref{p:lens} we study the manifolds $L \times S^1$, leading to the proofs of \cref{theorem:main-S^1xL,theorem:hcob-S^1xL,,theorem:s-hcob-S^1xL} subject to results about $\wt K_0(\Z C_m)$. 
In \cref{p:algebra} we develop the necessary background on integral representation theory and algebraic number theory. We then study the involution of $\wt K_0(\Z C_m)$ and prove Theorems \ref{theorem:main12-red-intro} and \ref{theorem:main3-red-intro}.

\subsection*{Conventions} 

The following conventions apply throughout this article, unless otherwise specified.
As above, $n \ge 4$ will be an integer and an $n$-manifold will be a compact connected $\CAT$ $n$-manifold where $\CAT \in \{\Diff, \PL, \TOP\}$.
They will be assumed closed except where it is clear from the context that they are not, e.g.\ thickenings and cobordisms.
We frequently assume \cref{assumptions} but will state this as needed.
CW complexes will be assumed to be connected.
Rings will be assumed to be associative and have a multiplicative identity, but not necessarily commutative.

\subsection*{Acknowledgements}
We are grateful to Daniel Kasprowski and Arunima Ray for discussions on the $s$-cobordism theorem, to Henri Johnston, Rachel Newton, and Jack Shotton for advice on ideal class groups, and to Scott Schmieding for discussions on Nil groups. We are extremely grateful to the anonymous referees for their careful reading and insightful comments.

CsN was supported by EPSRC New Investigator grant EP/T028335/2.
JN was supported by the Heilbronn Institute for Mathematical Research. 
MP was partially supported by EPSRC New Investigator grant EP/T028335/2 and EPSRC New Horizons grant EP/V04821X/2.

\part{General results}
\label{p:general}

In this part we establish general results regarding simple homotopy equivalence. This is the basis for our applications in Part \ref{p:lens}. In \cref{section:prelims}, we recall the basic theory of simple homotopy equivalence, culminating in constraints on the Whitehead torsion of homotopy equivalences between manifolds. \cref{section:realisation} concerns the two methods which we use for constructing manifolds: via $h$-cobordisms (\cref{ss:realising-cobordisms}) and via the surgery exact sequence (\cref{ss:wall}).
In \cref{s:manifold-sets} we study the simple homotopy manifold sets and prove \cref{thmx:bijections-with-manifold-sets-intro}. 

\section{Preliminaries}
\label{section:prelims}

In this section we recall the definition of simple homotopy equivalence, the Whitehead group and the Whitehead torsion, as well as some of their basic properties. Our main sources are Milnor \cite{Mi66}, Cohen \cite{Co73}, and Davis-Kirk \cite{DK01}.

\subsection{Simple homotopy equivalence}

Let $X$ be a connected CW complex and let $\phi \colon D^n \to X$ be a cellular map. Divide the boundary of the closed $(n+1)$-cell $D^{n+1}$ into two $n$-discs, $\partial D^{n+1} \cong D^n \cup_{S^{n-1}} D^n$, and let $D^n \to \partial D^{n+1}$ be the inclusion of the first copy of $D^n$. Using this inclusion and $\phi$ we form the union $X \cup_{\phi} D^{n+1}$.
The inclusion $X \to X \cup_{\phi} D^{n+1}$ is called an \emph{elementary expansion}. There is a deformation retract $X \cup_{\phi} D^{n+1} \to X$ in the other direction, and this is called an \emph{elementary collapse}.

\begin{definition}\label{defn:simple-h-e}
A homotopy equivalence $f \colon X \to Y$ between finite CW complexes is \emph{simple} if $f$ is homotopic to a map that is a composition of finitely many elementary expansions and collapses
\[X = X_0 \to X_1 \to X_2 \to \cdots \to X_k = Y.\]
\end{definition}

For a homotopy equivalence $f \colon X \to Y$, Whitehead introduced an invariant $\tau(f)$, the \emph{Whitehead torsion} of $f$, which lies in the \emph{Whitehead group} $\Wh(\pi_1(Y))$ of $\pi_1(Y)$. We shall define both the Whitehead group and the Whitehead torsion shortly. The motivation for the Whitehead torsion is the following beautiful result, which completely characterises whether or not a homotopy equivalence is simple.

\begin{theorem}[Whitehead \cite{Whitehead-she}]\label{theorem:Whitehead}
A homotopy equivalence $f \colon X \to Y$ between CW complexes $X$, $Y$ is simple if and only if the element $\tau(f)$ of $\Wh(\pi_1(Y))$ is zero.
\end{theorem}

In dimensions $n \neq 4$, every closed $n$-manifold admits an $n$-dimensional CW structure; see Kirby-Siebenmann~\cite[III.2.2]{Kirby-Siebenmann:1977-1} for $n \geq 5$, \cite{Ra26} for $n=2$ and \cite{Moise} for $n=3$. 
In dimension 4, it is an open question whether this holds.  Every smooth or PL $n$-manifold admits an $n$-dimensional triangulation, and hence a CW structure. 

We explain how the notion of simple homotopy equivalence makes sense for 4-manifolds, even those for which we do not know whether they admit a CW structure. The procedure, which is due to Kirby-Siebenmann~\cite[III,~\S4]{Kirby-Siebenmann:1977-1} works for any dimension, so we work in this generality.

Let $M$ be a topological $n$-manifold.
Embed $M$ in high-dimensional Euclidean space. By \cite{Kirby-Siebenmann:1977-1}*{III,~\S4}, there is a normal disc bundle $D(M) \to M$ that admits a triangulation and hence a CW structure. The inclusion map $z_M \colon M \to D(M)$ of the $0$-section is a homotopy equivalence.
Let $z_M^{-1}$ denote the homotopy inverse of $z_M$.

\begin{definition}\label{defn:KS-defn-simple}
 We say that a homotopy equivalence $f \colon M \to N$ between topological manifolds is \emph{simple} if the composition $z_N \circ f \circ z_M^{-1} \colon D(M) \to D(N)$ is simple.
\end{definition}

Kirby-Siebenmann~\cite{Kirby-Siebenmann:1977-1}*{III,~\S4} showed that whether or not the composition $z_N \circ f \circ z_M^{-1}$ is simple does not depend on the choice of normal disc bundle nor on the choice of triangulation. In particular, this leads to a well-defined equivalence relation on topological manifolds which extends the notion of simple homotopy equivalence on smooth manifolds.

\begin{remark}
Simple homotopy theory in fact extends beyond topological manifolds. 
By West's resolution of the Borsuk conjecture~\cite{We77}, a compact ANR has a canonical simple type.  
Hanner~\cite{Han51} showed that compact topological manifolds are compact ANRs. 
\end{remark}

If $M$, $N$ are smooth or PL, we can ask whether $f$ is simple using the canonical class of triangulations of $M$ and $N$, or by forgetting the smooth/PL structures and using the  Kirby-Siebenmann method from \cref{defn:KS-defn-simple}.

\begin{proposition}[Kirby-Siebenmann~{\cite[III.5.1]{Kirby-Siebenmann:1977-1}}]\label{prop:kirby-siebenmann-well-defined-simple}
    For $\CAT \in \{\Diff,\PL\}$, a homotopy equivalence $f \colon M \to N$ between $\CAT$ manifolds $M$ and $N$ is simple with respect to their canonical class of triangulations if and only if it is simple with respect to \cref{defn:KS-defn-simple}. 
\end{proposition}

Hence we have a coherent notion of simple homotopy equivalence in all three manifold categories. 
We remark that \cref{prop:kirby-siebenmann-well-defined-simple} can also be proven using the following theorem of Chapman. 

\begin{theorem}[Chapman~\cite{Ch74}]\label{thm:chapman}
    Let $f \colon X \to Y$ be a homeomorphism between compact, connected CW complexes. Then $f$ is a simple homotopy equivalence. 
\end{theorem}

Thus for closed manifolds $M$ and $N$ that admit a CW structure, one can deduce from Chapman's \cref{thm:chapman} that the question of whether $f \colon M \xrightarrow{\simeq} N$ is simple does not depend on the CW structures. Thus the extra work using disc bundles in \cref{defn:KS-defn-simple} is only required for non-smoothable topological 4-manifolds.  

\subsection{The Whitehead group} \label{ss:wh-group}

\begin{definition}
For a ring $R$, let $\GL(R) = \colim_{n} \GL_n(R)$, where we take the colimit with respect to the inclusions $\GL_n(R) \hookrightarrow \GL_{n+1}(R)$, $A \mapsto \left(\begin{smallmatrix} A & 0 \\ 0 & 1 \end{smallmatrix}\right)$. Define \[K_1(R)  :=  \GL(R)^{\text{ab}}.\]
\end{definition}

The Whitehead lemma \cite[Lemma 1.1]{Mi66} states that the commutator subgroup of $\GL(R)$ is equal to the subgroup $\E(R)$ generated by elementary matrices (i.e.\ the matrices $E$ such that $A \mapsto EA$ is an elementary row operation). It follows that we can also write \[K_1(R) = \GL(R)/\E(R).\]

\begin{definition}
Define the \emph{Whitehead group} of a group $G$ to be $\Wh(G) = K_1(\Z G)/{\pm G}$, where the map $\pm G \to K_1(\Z G)$ is the composition $\pm G \subseteq \GL_1(\Z G) \subseteq \GL(\Z G) \to K_1(\Z G)$.
\end{definition}

Note that $\Wh$ is a functor from the category of groups to the category of abelian groups. The map induced by a homomorphism $\theta \colon G \rightarrow H$ is denoted by $\theta_* \colon \Wh(G) \rightarrow \Wh(H)$. Similarly, for a continuous map $f \colon X \rightarrow Y$ between topological spaces, the map induced by $\pi_1(f) \colon \pi_1(X) \rightarrow \pi_1(Y)$ is denoted by $f_* \colon \Wh(\pi_1(X)) \rightarrow \Wh(\pi_1(Y))$.
The isomorphism type of $\Wh(G)$ is known in some cases. For us the following examples are relevant.

\begin{proposition}[{\cite[(11.5)]{Co73}}] \label{prop:cyclic-wh}
If $C_m$ is the finite cyclic group of order $m$, then 
\[\Wh(C_m) \cong \Z^{\left\lfloor m / 2 \right\rfloor + 1 - \delta(m)}\] 
where $\delta(m)$ is the number of positive integers dividing $m$.
\end{proposition}

The Whitehead group $\Wh(G)$ of a group $G$ is equipped with a natural involution, i.e.\ an automorphism $x \mapsto \ol{x}$ such that $\ol{\ol{x}} = x$. Equivalently, it has a $\Z C_2$-module structure, where the generator of $C_2$ acts by the involution. We describe this involution next. 

Let $R$ be a ring with involution, i.e.\ a ring equipped with a map $x \mapsto \ol{x}$ that is an involution on $R$ as an abelian group, and satisfies $\ol{xy} = \ol{y} \cdot \ol{x}$. This induces an involution on $\GL(R)$, sending $A = (A_{ij}) \in \GL_n(R)$ to $\ol{A} = (\ol{A_{ji}}) \in \GL_n(R)$, the conjugate transpose of $A$. Note that this is perhaps non-standard notation for the conjugate transpose. We use this convention so that the notation for involutions is consistent.
This involution preserves the subgroup $\E(R) \subseteq \GL(R)$ and so induces an involution $K_1(R)$.

If $G$ is a group and $w \colon G \to \{\pm 1\}$ is an orientation character, i.e.\ a homomorphism from $G$ to $\{\pm 1\}$, then the integral group ring $\Z G$ has an involution given by $\sum_{i=1}^k n_i g_i \mapsto \sum_{i=1}^k w(g_i) n_i g_i^{-1}$ for $n_i \in \Z$ and $g_i \in G$. The resulting involution on $K_1(\Z G)$ preserves $\pm G$ and so induces an involution on $\Wh(G)$. 

We write $\Wh(G,w)$ for the abelian group $\Wh(G)$ equipped with the involution determined by the orientation character $w \colon G \to \{\pm 1\}$. This can equivalently be regarded as a $\Z C_2$-module. When $w \equiv 1$ is the trivial homomorphism, we omit it from the notation and write $\Wh(G)$ for $\Wh(G,1)$.
Note that the choice of $w$ only affects the involution, and $\Wh(G,w)$ is equal to $\Wh(G)$ as an abelian group. 

Next we recall, from the introduction, the definition of two subgroups of $\Wh(G,w)$ which will play an important r\^{o}le in the rest of this article.

\begin{definition} \label{def:IJ}
For a group $G$, a homomorphism $w \colon G \to \{\pm 1\}$ and $n \in \Z$, define:
\[
\begin{aligned}
\mathcal{J}_n(G,w) &:= \{ y \in \Wh(G,w) \mid y = -(-1)^n\ol{y} \} \leq \Wh(G,w)\\
\mathcal{I}_n(G,w) & :=  \{ x - (-1)^n\ol{x} \mid x \in \Wh(G,w) \} \leq \Wh(G,w). \\
\end{aligned}
\]
\end{definition}

We have $\mathcal{I}_n(G,w) \leq \mathcal{J}_n(G,w)$, and it follows from \cref{prop:tate-C2} that there is a canonical isomorphism
\[
\mathcal{J}_n(G,w) / \mathcal{I}_n(G,w) \cong \wh{H}^{n+1}(C_2;\Wh(G,w)).
\]
We denote the quotient map by 
\[
\pi \colon \mathcal{J}_n(G,w) \rightarrow \wh{H}^{n+1}(C_2;\Wh(G,w)).
\]

\subsection{The Whitehead torsion of a chain homotopy equivalence} \label{ss:wh-che}

We now define the Whitehead torsion in the algebraic setting. 
A \textit{free, based $R$-module} is a free $R$-module $P$ together with a choice of isomorphism $P \cong R^k$ for some $k$. The canonical basis for $R^k$ determines a basis for $P$.

Let $G$ be a group and let $f \colon C_* \to D_*$ be a chain homotopy equivalence between finitely generated, free, based (left) $\Z G$-module chain complexes. Consider the algebraic mapping cone $(\mathscr{C}(f),\partial_*)$, which is also a finitely generated, free, based $\Z G$-module chain complex. Since $f$ is a chain homotopy equivalence, $\mathscr{C}(f)$ is chain contractible; see e.g.~\cite[Proposition~3.14]{Ranicki-AGS}. Choose a chain contraction $s \colon \mathscr{C}(f)_* \to \mathscr{C}(f)_{*+1}$. Now consider the modules
\[
\mathscr{C}(f)_{\operatorname{odd}}  :=  \bigoplus_{i=0}^{\infty} \mathscr{C}(f)_{2i+1} \quad \text{ and } \quad \mathscr{C}(f)_{\operatorname{even}}  :=  \bigoplus_{i=0}^{\infty} \mathscr{C}(f)_{2i}.
\]
These are finitely generated, free, based $\Z G$-modules. Since $\mathscr{C}(f)$ is contractible, its Euler characteristic vanishes, and so these modules are of equal, finite rank. The collection of boundary maps $\oplus_{i=0}^{\infty} \partial_{2i+1}$ with odd degree domain, and the collection of the maps in the chain contraction $\oplus_{i=0}^{\infty} s_{2i+1}$ with odd degree domain, both give rise to homomorphisms from $\mathscr{C}(f)_{\operatorname{odd}}$ to $\mathscr{C}(f)_{\operatorname{even}}$. Using the given bases, their sum is an element of $\GL(\Z G)$ ~\cite[(15.1)]{Co73}, and hence represents an element of the Whitehead group.

\begin{definition}\label{defn:whitehead-torsion}
The \emph{Whitehead torsion} of $f$ is the equivalence class
\[
\tau(f)  :=  \Bigl[ \bigoplus_{i=0}^{\infty} \bigl( \partial_{2i+1} + s_{2i+1} \bigr) \colon \mathscr{C}(f)_{\operatorname{odd}} \to \mathscr{C}(f)_{\operatorname{even}} \Bigr] \in \Wh(G).
\]
This equivalence class is independent of the choice of the chain contraction $s$~\cite[(15.3)]{Co73}.
If~$\tau(f)=0$ then we say that $f$ is a \emph{simple chain homotopy equivalence}.
\end{definition}

\begin{remark} \label{rem:tau-inv}
Since $\mathscr{C}(f)_{\operatorname{odd}}$ and $\mathscr{C}(f)_{\operatorname{even}}$ are finitely generated, the sum in the definition of $\tau(f)$ is a finite sum. 
The equivalence class $\tau(f)$ remains invariant under permutations of the bases of $C_*$ and $D_*$, or if a basis element is multiplied by $(-1)$ or an element of $G$~\cite[(15.2) and (10.3)]{Co73}.
\end{remark}

We now list some useful facts about $\tau$. We fix the group $G$, and assume all chain complexes are finitely generated, free, based, left $\Z G$-module chain complexes.

\begin{proposition}[{\cite[Theorem 11.27]{DK01}}] \label{prop:chain-hom-chain-equivs-same-torsion}
Let $f,g \colon C_* \to D_*$ be homotopic chain homotopy equivalences. Then $\tau(f) = \tau(g)$.
\end{proposition}

\begin{lemma}[{\cite[Theorem 11.28]{DK01}}] \label{lem:WT-ch-comp}
Let $f \colon C_* \rightarrow D_*$ and $g \colon D_* \rightarrow E_*$ be chain homotopy equivalences. Then $\tau(g \circ f) = \tau(f) + \tau(g)$. In particular $\tau(\Id)=0$. 
\end{lemma}

We say that a short exact sequence $0 \rightarrow C'_* \rightarrow C_* \rightarrow C''_* \rightarrow 0$ of chain complexes is based if the basis of $C_*$ consists of the image of the basis of $C'_*$ and an element from the preimage of each basis element of $C''_*$.

\begin{lemma} \label{lem:WT-ch-add}
Let $0 \rightarrow C'_* \rightarrow C_* \rightarrow C''_* \rightarrow 0$ and $0 \rightarrow D'_* \rightarrow D_* \rightarrow D''_* \rightarrow 0$ be based short exact sequences of chain complexes, and let $(f',f,f'')$ be a morphism between them, where $f' \colon C'_* \rightarrow D'_*$, $f \colon C_* \rightarrow D_*$, and $f'' \colon C''_* \rightarrow D''_*$ are chain homotopy equivalences. Then $\tau(f) = \tau(f') + \tau(f'')$. 
\end{lemma}

\begin{proof}
The given data determines a based short exact sequence $0 \rightarrow \mathscr{C}(f') \rightarrow \mathscr{C}(f) \rightarrow \mathscr{C''}(f) \rightarrow 0$ of mapping cones, so the statement follows from \cite[Theorem 11.23]{DK01}.
\end{proof}

For a chain complex $C_*$ and an integer $k$ we denote by $C_{k+*}$ the same chain complex with shifted grading $i \mapsto C_{k+i}$. Similarly, the cochain complex $C_{k-*}$ is defined by changing the grading to $i \mapsto C_{k-i}$. The notation $C^{k+i}$ and $C^{k-*}$ is defined analogously for a cochain complex $C^*$.

\begin{lemma} \label{lem:WT-ch-shift}
Let $f \colon C_* \rightarrow D_*$ be a chain homotopy equivalence. For every $k \in \Z$, it can also be regarded as a chain homotopy equivalence $f \colon C_{k+*} \rightarrow D_{k+*}$, and we have $\tau(f \colon C_{k+*} \rightarrow D_{k+*}) = (-1)^k \tau(f \colon C_* \rightarrow D_*)$.
\end{lemma}

\begin{proof}
If we shift the grading by an even number, then $\mathscr{C}(f)_{\operatorname{odd}}$ and $\mathscr{C}(f)_{\operatorname{even}}$ remain unchanged. If we shift the grading by an odd number, then $\mathscr{C}(f)_{\operatorname{odd}}$ and $\mathscr{C}(f)_{\operatorname{even}}$ are swapped, and by \cite[(15.1)]{Co73}, $\tau(f)$ changes its sign.
\end{proof}

\begin{definition} \label{def:WT-ch-cochain}
Let $f \colon C^* \rightarrow D^*$ be a homotopy equivalence of cochain complexes of finitely generated, free, based, left $\Z G$-modules. It can be regarded as a homotopy equivalence of chain complexes $f \colon C^{-*} \rightarrow D^{-*}$, and we define $\tau(f \colon C^* \rightarrow D^*)  :=  \tau(f \colon C^{-*} \rightarrow D^{-*})$.
\end{definition}

Next we describe how we can define the dual of a left $R$-module as another left $R$-module, if $R$ is a ring with involution.

\begin{definition} \label{def:tensor-hom}
Suppose that $R$ is a ring with involution. Let $X$ and $Y$ be a left and a right $R$-module respectively. Then we define the abelian groups
\[
\begin{aligned}
Y \otimes_R X &= {Y \otimes_{\Z} X} / {(yr \otimes x = y \otimes rx)} \quad \forall r \in R, x \in X, y \in Y; \\
\Hom_R^{\mathrm{lr}}(X,Y) &= \left\{ \varphi \in \Hom_{\Z}(X, Y) \mid \forall r \in R, x \in X 
\colon \varphi(rx)=\varphi(x)\ol{r} \right\}.
\end{aligned}
\]
If $Y$ is moreover an $(S,R)$-bimodule for some ring $S$, then $Y \otimes_R X$ is a left $S$-module with $s(y \otimes x) = (sy) \otimes x$, and $\Hom_R^{\mathrm{lr}}(X,Y)$ is a left $S$-module with $(s\varphi)(x)=s\varphi(x)$.
\end{definition}

\begin{remark}
$\Hom_R^{\mathrm{lr}}(X,Y)$ is equal to $\Hom_R^r(\ol{X},Y)$, the group of right $R$-module homomorphisms $\ol{X} \rightarrow Y$, as a subgroup of $\Hom_{\Z}(X, Y)$ (or as a left $S$-module). Here $\ol{X}$ denotes the right $R$-module that is equal to $X$ as an abelian group, with its multiplication given by $xr= \ol{r} \cdot x$, where $\cdot$ denotes multiplication in $X$.
\end{remark}

\begin{definition} \label{def:dual}
Suppose that $R$ is a ring with involution. Let $X$ be a left $R$-module. Its \emph{dual} is the left $R$-module $X^* := \Hom_R^{\mathrm{lr}}(X,R)$.
\end{definition}

Now suppose that $G$ is equipped with a group homomorphism $w \colon G \to \{\pm 1\}$, i.e.\ an orientation character. This determines an involution on the group ring $\Z G$, and also on the Whitehead group $\Wh(G,w)$ (see \cref{ss:wh-group}). Since $\Hom_R^{\mathrm{lr}}$ depends on the involution of $R$, we will write $\Hom_{\Z G,w}^{\mathrm{lr}}$ when $\Z G$ is equipped with the involution determined by $w$, and $\Hom_{\Z G}^{\mathrm{lr}}$ will mean $\Hom_{\Z G,1}^{\mathrm{lr}}$. 

\begin{definition}
We introduce the following notation.
\begin{itemize}[$\bullet$]
\item Let $\Z^w$ be the right $\Z G$-module with underlying abelian group $\Z$ and multiplication $ng = w(g)n$, where $n \in \Z$, $g \in G$ and $w(g) \in  \{\pm 1\} \subseteq \Z$. 
\item For a right $\Z G$-module $Y$ we define the right $\Z G$-module $Y^w := \Z^w \otimes_{\Z} Y$, with underlying abelian group $\Z \otimes_{\Z} Y$ and multiplication $(n \otimes y)g = (ng) \otimes (yg)$, where $n \in \Z$, $y \in Y$ and $g \in G$. Equivalently, $Y^w$ has underlying abelian group $Y$ and multiplication $yg = w(g)y \cdot g$, where $\cdot$ is the multiplication in $Y$. If $Y$ is an $(S,\Z G)$-bimodule for some ring $S$, then so is $Y^w$.
\end{itemize}
\end{definition}

For example, $\Z G^w$ is a $\Z G$-bimodule, whose left multiplication is the same as that of $\Z G$, and whose right multiplication is twisted by $w$. 

If $X$ is a left $\Z G$-module, then by the definitions we have $\Hom_{\Z G,w}^{\mathrm{lr}}(X,Y) = \Hom_{\Z G}^{\mathrm{lr}}(X,Y^w)$ as abelian groups (or as left $S$-modules, when $Y$ is a left $S$-module). 

\begin{definition} 
For a group $G$, homomorphism $w \colon G \to \{\pm 1\}$ and left $\Z G$-module $X$, the left $\Z G$-module $\Hom_{\Z G,w}^{\mathrm{lr}}(X,\Z G) = \Hom_{\Z G}^{\mathrm{lr}}(X,\Z G^w)$ will be denoted by $X^{w*}$. 
\end{definition}

If $X$ is a finitely generated, free, based, left $\Z G$-module, then $X^{w*}$ is also a finitely generated, free, left $\Z G$-module, naturally equipped with a dual basis. If $f \colon X_1 \rightarrow X_2$ is a homomorphism between two such modules, then the matrix of its dual, $f^{w*} \colon X_2^{w*} \rightarrow X_1^{w*}$ is the conjugate transpose of the matrix of $f$, where conjugation means applying the involution determined by $w$.

If $C_*$ is a finitely generated, free, based, left $\Z G$-module chain complex, then $C^{w*}$ will denote the dual cochain complex, which also consists of finitely generated, free, based, left $\Z G$-modules, i.e.\ we have $C^{wi} = \Hom_{\Z G,w}^{\mathrm{lr}}(C_i,\Z G)$ for all $i \in \Z$, with the induced coboundary maps.

\begin{lemma} \label{lem:WT-ch-dual}
Let $f \colon C_* \rightarrow D_*$ be a chain homotopy equivalence and let $f^{w*} \colon D^{w*} \rightarrow C^{w*}$ be its dual. Then $\tau(f^{w*}) = \ol{\tau(f)} \in \Wh(G,w)$. 
\end{lemma}

\begin{proof}
Let $(\mathscr{C}(f),\partial^f_*)$ denote the mapping cone of $f$, so that $\mathscr{C}(f)_i = C_{i-1} \oplus D_i$. We regard $f^{w*} \colon D^{w*} \rightarrow C^{w*}$ as a homotopy equivalence $f^{w*} \colon D^{w(-*)} \rightarrow C^{w(-*)}$ of chain complexes (where $C^{w(-i)} = \Hom_{\Z G,w}^{\mathrm{lr}}(C_{-i},\Z G)$, etc.), and define its mapping cone $(\mathscr{C}(f^{w*}),\partial^{f^{w*}}_*)$. Then 
\[
\mathscr{C}(f^{w*})_i = (D_{-i+1})^{w*} \oplus (C_{-i})^{w*} \cong (\mathscr{C}(f)_{-i+1})^{w*},
\] 
and this implies that 
$\mathscr{C}(f^{w*})_{\operatorname{odd}} \cong (\mathscr{C}(f)_{\operatorname{even}})^{w*}$ and $\mathscr{C}(f^{w*})_{\operatorname{even}} \cong (\mathscr{C}(f)_{\operatorname{odd}})^{w*}$. 
Moreover, $\partial^{f^{w*}}_i \colon \mathscr{C}(f^{w*})_i \rightarrow \mathscr{C}(f^{w*})_{i-1}$ is the dual of $\partial^f_{-i+2} \colon \mathscr{C}(f)_{-i+2} \rightarrow \mathscr{C}(f)_{-i+1}$.

Let $s_* \colon \mathscr{C}(f)_* \rightarrow \mathscr{C}(f)_{*+1}$ be a chain contraction. We define 
$s'_* \colon \mathscr{C}(f^{w*})_* \rightarrow \mathscr{C}(f^{w*})_{*+1}$ by 
\[s'_i = (s_{-i})^{w*} \colon \mathscr{C}(f^{w*})_i \cong (\mathscr{C}(f)_{-i+1})^{w*} \rightarrow \mathscr{C}(f^{w*})_{i+1} \cong (\mathscr{C}(f)_{-i})^{w*}.\] 
Then it follows by dualising the formula $s_* \partial^f + \partial^f s_* = \Id$ that $s'_*$ is a chain contraction of $\mathscr{C}(f^{w*})$. 
Thus $\bigoplus_i \bigl( \partial^{f^{w*}}_{2i+1} + s'_{2i+1} \bigr) \colon \mathscr{C}(f^{w*})_{\operatorname{odd}} \to \mathscr{C}(f^{w*})_{\operatorname{even}}$ is the dual of $\bigoplus_i \bigl( \partial^f_{2i+1} + s_{2i+1} \bigr) \colon \mathscr{C}(f)_{\operatorname{odd}} \to \mathscr{C}(f)_{\operatorname{even}}$.  
Hence its matrix is the conjugate transpose of the matrix of the latter map, which implies that $\tau(f^{w*}) = \ol{\tau(f)}$.
\end{proof}

Finally we consider the effect of changing the underlying (group) ring.

\begin{definition} \label{def:twist-module}
Let $A$ and $B$ be groups, $X$ a left $\Z B$-module and $\theta \in \Hom(A,B)$. The left $\Z A$-module $X_{\theta}$ is defined as follows. The underlying abelian group of $X_{\theta}$ is the same as that of $X$. For every $a \in A$ and $x \in X_{\theta}$ let $ax = \theta(a) \cdot x$, where $\cdot$ denotes multiplication in $X$.

Similarly, if $Y$ is a right $\Z B$-module, then the right $\Z A$-module $Y^{\theta}$ is equal to $Y$ as an abelian group, and $ya = y \cdot \theta(a)$ for every $a \in A$ and $y \in Y^{\theta}$, where $\cdot$ denotes multiplication in $Y$.
\end{definition}

\begin{lemma} \label{lem:theta}
Let $A$ and $B$ be groups equipped with orientation characters $w_A \colon A \rightarrow \left\{ \pm 1 \right\}$ and $w_B \colon B \rightarrow \left\{ \pm 1 \right\}$ respectively. Let $X$ be a left $\Z A$-module and let $\theta \colon A \rightarrow B$ be an isomorphism such that $w_B \circ \theta = w_A$. Then we have the following isomorphisms of left $\Z B$-modules. 
\begin{clist}{(a)}
\item\label{item-lem-theta-a} 
$\Z B^{\theta} \otimes_{\Z A} X \cong X_{\theta^{-1}}$. 
\item\label{item-lem-theta-b}
$\Hom_{\Z B,w_B}^{\mathrm{lr}}(X_{\theta^{-1}},\Z B) = \Hom_{\Z A,w_A}^{\mathrm{lr}}(X,\Z B^{\theta}) \cong \Hom_{\Z A,w_A}^{\mathrm{lr}}(X,\Z A)_{\theta^{-1}}$, in particular we have $(X_{\theta^{-1}})^{w_B*} \cong (X^{w_A*})_{\theta^{-1}}$. 
\end{clist}
\end{lemma}

\begin{proof}
\eqref{item-lem-theta-a} If $Y$ is a right $\Z A$-module and a left $R$-module for some ring $R$, then it follows from the definition of the tensor product (see \cref{def:tensor-hom}) that $Y \otimes_{\Z A} X = Y^{\theta^{-1}} \otimes_{\Z B} X_{\theta^{-1}}$ (as left $R$-modules). By applying this for $Y = \Z B^{\theta}$ and $R = \Z B$, we get that $\Z B^{\theta} \otimes_{\Z A} X = \Z B \otimes_{\Z B} X_{\theta^{-1}}$. Of course $X_{\theta^{-1}} \cong \Z B \otimes_{\Z B} X_{\theta^{-1}}$ via the map $x \mapsto 1 \otimes x$. 

\eqref{item-lem-theta-b} 
This follows from \cite[Proposition 3.9]{Ni23}.
\end{proof}

\begin{lemma} \label{lem:tau-change}
Let $A$ and $B$ be groups and let $\theta \colon A \rightarrow B$ be an isomorphism. Let $f \colon C_* \rightarrow D_*$ be a chain homotopy equivalence of finitely generated, free, based, left $\Z B$-module chain complexes, which can also be regarded as a chain homotopy equivalence $f \colon (C_*)_{\theta} \rightarrow (D_*)_{\theta}$ of $\Z A$-module chain complexes. We have 
\[
\tau(f \colon (C_*)_{\theta} \rightarrow (D_*)_{\theta})  =  \theta^{-1}_*(\tau(f \colon C_* \rightarrow D_*)) \in \Wh(A).
\]
\end{lemma}

\begin{proof}
First note that if $g \colon X \rightarrow Y$ is an isomorphism between finitely generated, free, based, left $\Z B$-modules, and we regard it as an isomorphism $g \colon X_{\theta} \rightarrow Y_{\theta}$ of $\Z A$-modules, then its matrix changes by applying $\Z \theta^{-1}$ to each entry, hence $\tau(g \colon X_{\theta} \rightarrow Y_{\theta}) = \theta^{-1}_*(\tau(g \colon X \rightarrow Y)) \in \Wh(A)$. 

Now let $(\mathscr{C}(f),\partial_*)$ denote the mapping cone of $f$ and let $s_* \colon \mathscr{C}(f)_* \rightarrow \mathscr{C}(f)_{*+1}$ be the chain contraction used to define $\tau(f)$ over $\Z B$. Then $s_*$ is also a chain contraction for $(\mathscr{C}(f)_{\theta},\partial_*)$, the mapping cone over $\Z A$.  
So we can compute $\tau(f \colon (C_*)_{\theta} \rightarrow (D_*)_{\theta})$ from 
\[ 
\bigoplus_i \bigl( \partial_{2i+1} + s_{2i+1} \bigr) \colon (\mathscr{C}(f)_{\theta})_{\operatorname{odd}} = (\mathscr{C}(f)_{\operatorname{odd}})_{\theta} \to (\mathscr{C}(f)_{\theta})_{\operatorname{even}} = (\mathscr{C}(f)_{\operatorname{even}})_{\theta},
\] 
which is $\bigoplus_i \bigl( \partial_{2i+1} + s_{2i+1} \bigr) \colon \mathscr{C}(f)_{\operatorname{odd}} \to \mathscr{C}(f)_{\operatorname{even}}$ regarded as an isomorphism of $\Z A$-modules, hence $\tau(f \colon (C_*)_{\theta} \rightarrow (D_*)_{\theta})  =  \theta^{-1}_*(\tau(f \colon C_* \rightarrow D_*))$.
\end{proof}

\subsection{The Whitehead torsion of a homotopy equivalence}

Now let $X$ and $Y$ be finite CW complexes with universal covers $\widetilde{X}$ and $\widetilde{Y}$, and let $F := \pi_1(X)$ and $G := \pi_1(Y)$.  
For an $(R,\Z G)$-bimodule $P$, the cellular chain complex of $Y$ with $P$ coefficients is defined to be $C_*(Y;P) := P \otimes_{\Z G} C_*(\wt{Y})$, which is a finitely generated, free, left $R$-module chain complex. Similarly, the cellular cochain complex of $Y$ with $P$ coefficients is $C^*(Y;P) := \Hom_{\Z G}^{\mathrm{lr}}(C_*(\wt{Y}),P)$.
Thus the cellular chain complex of $Y$ with $\Z G$ coefficients is $C_*(Y; \Z G) = \Z G \otimes_{\Z G} C_*(\widetilde{Y}) \cong C_*(\widetilde{Y})$, which is a left $\Z G$-module chain complex. Choose a lift of each cell of $Y$ in $\widetilde{Y}$ to obtain a basis of $C_*(\widetilde{Y})$, which is well-defined up to ordering, signs and multiplication by elements of $G$. Similarly, the cellular chain complex of $X$ with $\Z F$ coefficients, $C_*(X; \Z F) \cong C_*(\widetilde{X})$, is a finitely generated, free, left $\Z F$-module chain complex with a basis well-defined up to ordering and multiplication by elements of $\pm F$.

Let $f \colon X \to Y$ be a cellular homotopy equivalence, and let $\theta = \pi_1(f) \colon F \rightarrow G$. The right $\Z G$-module $\Z G$ corresponds to a local coefficient system on $Y$, which is pulled back to the local coefficient system on $X$ corresponding to the right $\Z F$-module $\Z G^{\theta}$. Therefore $f$ induces a chain homotopy equivalence $f_* \colon C_*(X; \Z G^{\theta}) \rightarrow C_*(Y; \Z G)$ of left $\Z G$-module chain complexes. Note that by \cref{lem:theta}~\eqref{item-lem-theta-a} we have \[C_*(X; \Z G^{\theta}) = \Z G^{\theta} \otimes_{\Z F} C_*(\widetilde{X}) \cong C_*(\widetilde{X})_{\theta^{-1}} \cong C_*(X; \Z F)_{\theta^{-1}},\] so this is also a finitely generated, free chain complex with a basis that is well-defined up to ordering and multiplication by elements of $\pm G$.

\begin{definition}\label{defn:Whitehead-torsion-spaces}
The \emph{Whitehead torsion} $\tau(f) \in \Wh(\pi_1 Y)$ of the cellular homotopy equivalence $f \colon X \to Y$ is $\tau(f) :=\tau(f_*)$, where $f_* \colon C_*(X; \Z G^{\theta}) \rightarrow C_*(Y; \Z G)$ is the induced chain homotopy equivalence.
\end{definition}

By \cref{rem:tau-inv}, it follows that $\tau(f_*)$ is well-defined, even though the bases of $C_*(X; \Z G^{\theta})$ and $C_*(Y; \Z G)$ are well-defined only up to ordering and multiplication by elements of $\pm G$.

\begin{proposition}[{\cite[Statement~22.1]{Co73}}] \label{prop:WT-homotopic-maps}
Let $f, g \colon X \to Y$ be homotopic cellular homotopy equivalences between finite CW complexes. Then $\tau(f) = \tau(g) \in \Wh(G)$.
\end{proposition}

\begin{proof}
This follows immediately from  \cref{prop:chain-hom-chain-equivs-same-torsion}, because homotopic homotopy equivalences induce homotopic chain homotopy equivalences.
\end{proof}

We can now extend the definition of Whitehead torsion to arbitrary homotopy equivalences. If $f \colon X \to Y$ is a homotopy equivalence between finite CW complexes, then it is homotopic to a cellular homotopy equivalence $f' \colon X \to Y$, and we define $\tau(f)  :=  \tau(f')$. By \cref{prop:WT-homotopic-maps} this is independent of the choice of $f'$. Moreover, it follows that if $f, g \colon X \to Y$ are arbitrary homotopy equivalences and $f \simeq g$, then $\tau(f) = \tau(g)$.
Because of this, and as we are only interested in the Whitehead torsion of a homotopy equivalence, we will often not distinguish between homotopic maps.

We collect a few key properties of simple homotopy equivalences and Whitehead torsion.

\begin{proposition}
Let $f \colon X \to Y$ be a homeomorphism between finite CW complexes. Then $\tau (f)=0$.
\end{proposition}

\begin{proof}
By \cref{thm:chapman}, $f$ is a simple homotopy equivalence. Hence $\tau(f)=0$ by \cref{theorem:Whitehead}. 
\end{proof}

\begin{proposition}[{\cite[Corollary 5.1A]{Co73}}] \label{prop:map-cyl}
Let $X$ and $Y$ be finite CW complexes and let $f \colon X \to Y$ be a cellular map. Let $\Cyl_f$ denote the mapping cylinder of $f$. Then the inclusion $Y \rightarrow \Cyl_f$ is a simple homotopy equivalence.
\end{proposition}

The following is a consequence of \cref{lem:WT-ch-comp,lem:tau-change}.

\begin{proposition}[{\cite[Statement~22.4]{Co73}}] \label{prop:WT-composition}
Let $X$, $Y$ and $Z$ be finite CW complexes and let $f \colon X \to Y$ and $g \colon Y \to Z$ be homotopy equivalences. Then
\[ \tau(g \circ f) = \tau(g) + g_*(\tau(f)) .\]
\end{proposition}

It will be important to have formulae for how Whitehead torsion behaves under products and in fibre bundles.

\begin{proposition}[{\cite[Statement~23.2]{Co73}}] \label{prop:WT-product}
For $i=1,2$, let $X_i$, $Y_i$ be finite CW complexes and let $f_i \colon X_i \to Y_i$ be a homotopy equivalence.
Let $i_1 \colon Y_1 \hookrightarrow Y_1 \times Y_2$ and $i_2 \colon Y_2 \hookrightarrow Y_1 \times Y_2$ be natural inclusion maps defined by fixing a point in $Y_2$ and $Y_1$ respectively. 
For the product map $f_1 \times f_2 \colon X_1 \times X_2 \to Y_1 \times Y_2$, we have
\[ \tau(f_1 \times f_2) = \chi(Y_2) \cdot {i_1}_*(\tau(f_1)) + \chi(Y_1) \cdot {i_2}_*(\tau(f_2)) \]
where $\chi$ denotes the Euler characteristic.
\end{proposition}

\begin{corollary} \label{cor:odd-product}
Let $f_1\colon M_1 \rightarrow N_1$ and $f_2\colon M_2 \rightarrow N_2$ be homotopy equivalences between odd dimensional manifolds. Then $f_1 \times f_2\colon M_1 \times M_2 \rightarrow N_1 \times N_2$ is a simple homotopy equivalence.
\end{corollary}

\begin{proof}
We have $\tau(f_1 \times f_2) = 0$, by \cref{prop:WT-product} and the fact that $\chi(N_1) = \chi(N_2) = 0$.
\end{proof}

\begin{proposition}[{\cite[Theorem A]{Anderson}}] \label{prop:WT-fiber}
For $i=1,2$, let $F_i \xrightarrow[]{j_i} E_i \xrightarrow[]{p_i} B_i$ be fibre bundles of finite CW complexes, and let
\[
\xymatrix{
F_1 \ar[r]^-{j_1} \ar[d]_-{h} & E_1 \ar[r]^-{p_1} \ar[d]_-{g} & B_1 \ar[d]^-{f} \\
F_2 \ar[r]^-{j_2} & E_2 \ar[r]^-{p_2} & B_2 
}
\]
be a fibre homotopy equivalence. Then there is a homomorphism $p_2^* \colon \Wh(\pi_1(B_2)) \rightarrow \Wh(\pi_1(E_2))$ such that 
\[
\tau(g) = p_2^*(\tau(f)) + \chi(B_2) \cdot {j_2}_*(\tau(h)).
\] 
\end{proposition}

Anderson~\cite[Corollary B]{Anderson} deduced \cref{prop:WT-product} from \cref{prop:WT-fiber}, by setting $F_1 = X_2$, $B_1 = X_1$, $F_2=Y_2$, and $B_2 = Y_1$, so that $h=f_2$, $f=f_1$, $g = f_2 \times f_1$, and $j_2 = i_2$. One needs more information than we have given, or will need: Anderson appeals to \cite{Kwun-Szczarba} for a computation that in this case $p_2^*(\tau(f)) = \chi(Y_2) \cdot {i_1}_*(\tau(f_1))$.

\subsection{The Whitehead torsion of a homotopy equivalence between manifolds}

In this section we specialise further to the case of manifolds, and see that $\mathcal{J}_n(G,w)$ and $\mathcal{I}_n(G,w)$ arise naturally.

The following generalises an observation of Wall, who considered the case where $M$, $N$ are simple Poincar\'{e} complexes \cite[p.~612]{Wa74}, i.e.\ finite Poincar\'{e} complexes $X$ for which the chain duality isomorphism $C^*(\wt{X}) \to C_*(\wt{X})$ is a simple chain homotopy equivalence (\cref{defn:whitehead-torsion}). At the time, it was known that smooth manifolds are simple Poincar\'{e} complexes \cite[Theorem 2.1]{Wall-SCM} but the case of topological manifolds was not established until the work of Kirby-Siebenmann \cite[III.5.13]{Kirby-Siebenmann:1977-1}.

\begin{proposition} \label{prop:WT-man}
Let $M$ and $N$ be closed $n$-dimensional topological manifolds. Let $G = \pi_1(M)$ with orientation character $w \colon G \rightarrow \left\{ \pm 1 \right\}$, and let $f \colon N \to M$ be a homotopy equivalence. Then $\tau(f) \in \mathcal{J}_n(G,w)$. 
\end{proposition}

\begin{proof}
Let $F = \pi_1(N)$ with orientation character $v \colon F \rightarrow \left\{ \pm 1 \right\}$, and let $\theta = \pi_1(f) \colon F \to G$. Then $(\Z^w)^{\theta} = \Z^{w \circ \theta}$, and since $f$ is a homotopy equivalence, it induces an isomorphism $H_n(N;\Z^{w \circ \theta}) \xrightarrow{\cong} H_n(M;\Z^w) \cong \Z$. It follows from Poincar\'{e} duality that $v$ is the only homomorphism $F \rightarrow \left\{ \pm 1 \right\}$ such that $H_n(N;\Z^v) \cong \Z$; indeed, for any such homomorphism $u$, $H_n(N;\Z^u) \cong H^0(N;\Z^{u\cdot v})$, which is $\Z$ if and only if $u =v$. Thus  $v = w \circ \theta$. 

Consider the following diagram of left $\Z G$-module chain complexes: 
\begin{equation} \label{eq:pd-diag}
\xymatrix{
C^{n-*}(N; \Z G^{w \theta}) \ar[d]_{\PD} & & C^{n-*}(M; \Z G^w) \ar[ll]_-{C^{w(n-*)}(f)} \ar[d]^{\PD} \\
C_*(N; \Z G^{\theta}) \ar[rr]^-{C_*(f)} & & C_*(M; \Z G)
}
\end{equation}
The horizontal maps are the chain and cochain maps induced by $f$. The Poincar\'{e} duality map $\PD \colon C^{n-*}(M; \Z G^w) \to C_*(M; \Z G)$ is given by taking cap product with a cycle representing the twisted fundamental class of $M$, a generator $[M] \in H_n(M;\Z^w) \cong \Z$ (using $\Z^w \otimes_{\Z} \Z G^w = (\Z G^w)^w = \Z G$). 
Similarly, taking cap product with a representative of $[N] \in H_n(N;\Z^v)$ defines $\PD \colon C^{n-*}(N; \Z G^{w \theta}) \to C_*(N; \Z G^{\theta})$. 
Here we write $\Z G^{w \theta} = (\Z G^w)^{\theta}$ for brevity, and we have $\Z G^{w \theta} = \Z G^{\theta v}$ (by the definitions, using that $v = w \circ \theta$) and $\Z^v \otimes_{\Z} \Z G^{\theta v} = \Z G^{\theta}$.
Since $f$ is a homotopy equivalence, $f_*([N]) = \pm [M]$, so by the naturality of the cap product, diagram \eqref{eq:pd-diag} commutes up to chain homotopy and sign. 

Note that by \cref{lem:theta} there are isomorphisms
\begin{equation}\label{eqn:star-star}
C_i(N; \Z G^{\theta}) = \Z G^{\theta} \otimes_{\Z F} C_i(\widetilde{N}) \cong C_i(\widetilde{N})_{\theta^{-1}} = C_i(N; \Z F)_{\theta^{-1}}
\end{equation}
and
\begin{align}
C^i(N; \Z G^{w \theta}) &= C^i(N; \Z G^{\theta v}) = \Hom_{\Z F}^{\mathrm{lr}}(C_i(\widetilde{N}),\Z G^{\theta v}) = \Hom_{\Z F, v}^{\mathrm{lr}}(C_i(\widetilde{N}),\Z G^{\theta}) \nonumber \\
& \begin{cases}
 \cong \Hom_{\Z F, v}^{\mathrm{lr}}(C_i(\widetilde{N}),\Z F)_{\theta^{-1}} = \Hom_{\Z F}^{\mathrm{lr}}(C_i(\widetilde{N}),\Z F^v)_{\theta^{-1}} = C^i(N; \Z F^v)_{\theta^{-1}} \label{cochains} \\
= \Hom_{\Z G, w}^{\mathrm{lr}}(C_i(\widetilde{N})_{\theta^{-1}},\Z G) = (C_i(\widetilde{N})_{\theta^{-1}})^{w*} \cong C_i(N; \Z G^{\theta})^{w*},  
\end{cases}
\end{align}
i.e.\ we have $C^i(N; \Z G^{w \theta}) \cong C^i(N; \Z F^v)_{\theta^{-1}} \cong C_i(N; \Z G^{\theta})^{w*}$.

We also have, by the definitions, that
\[
C^i(M; \Z G^w) = \Hom_{\Z G}^{\mathrm{lr}}(C_i(\widetilde{M}),\Z G^w) = C_i(\widetilde{M})^{w*} = C_i(M; \Z G)^{w*}.
\]
Using this and the second line of~\eqref{cochains}, $C^{w*}(f)$ can be identified with the $\vphantom{C}^{w*}$-dual of $C_*(f)$, i.e.\ $C^{wi}(f) =  C_i(f)^{w*}$ for all $i$.

Poincar\'{e} duality is a simple chain homotopy equivalence by \cite[III.5.13]{Kirby-Siebenmann:1977-1}, so we have that $\tau(\PD \colon C^{n-*}(M; \Z G^w) \to C_*(M; \Z G)) = 0$. By the first line of \eqref{cochains}, we have $C^{n-*}(N; \Z G^{w \theta}) \cong C^{n-*}(N; \Z F^v)_{\theta^{-1}}$,  and by \eqref{eqn:star-star} $C_*(N; \Z G^{\theta}) \cong C_*(N; \Z F)_{\theta^{-1}}$. Using this, we can identify the Poincar\'{e} duality map  
\[
\PD \colon C^{n-*}(N; \Z G^{w \theta}) \to C_*(N; \Z G^{\theta})
\] 
with 
\[
\PD \colon C^{n-*}(N; \Z F^v)_{\theta^{-1}} \to C_*(N; \Z F)_{\theta^{-1}}.
\] 
By \cref{lem:tau-change}, the Whitehead torsion of the latter map is given by $\theta_*$ applied to the Whitehead torsion of 
$\PD \colon C^{n-*}(N; \Z F^v) \to C_*(N; \Z F)$. But this torsion vanishes, again by \cite[III.5.13]{Kirby-Siebenmann:1977-1}. Thus all of the above Poincar\'{e} duality maps are simple. It then follows from \eqref{eq:pd-diag}, \cref{prop:chain-hom-chain-equivs-same-torsion}, and \cref{lem:WT-ch-comp} that $\tau(C^{w(n-*)}(f)) + \tau(C_*(f)) = 0$. Therefore, by the definition of $\mathcal{J}_n$, it is enough to prove that $\tau(C^{w(n-*)}(f)) = (-1)^n\overline{\tau(C_*(f))}$.

By \cref{lem:WT-ch-shift}, we know that $\tau(C^{w(n-*)}(f)) = (-1)^n \tau(C^{w(-*)}(f))$, and by \cref{def:WT-ch-cochain} $\tau(C^{w(-*)}(f)) = \tau(C^{w*}(f))$. Finally \cref{lem:WT-ch-dual} shows (the second equality of) $\tau(C^{w*}(f)) = \tau(C_*(f)^{w*}) = \overline{\tau(C_*(f))}$. We deduce that $\tau(C^{w(n-*)}(f)) = (-1)^n\overline{\tau(C_*(f))}$, completing the proof. 
\end{proof}

We recall the definitions of $h$- and $s$-cobordisms.

\begin{definition}
A cobordism $(W;M,N)$ of closed manifolds is an \emph{$h$-cobordism} if the inclusion maps $i_M \colon M \to W$ and $i_N \colon N \to W$ are homotopy equivalences. If in addition $i_M$ and $i_N$ are simple homotopy equivalences then $W$ is an \emph{$s$-cobordism}.
\end{definition}

\begin{definition}
Suppose that $W$ is an $h$-cobordism between closed $n$-dimensional manifolds $M$ and $N$. Let $G = \pi_1(W)$ and let $w \colon G \rightarrow \left\{ \pm 1 \right\}$ be the orientation character of $W$. We write $\tau(W,M)$ and $\tau(W,N)$ to denote the Whitehead torsion of the inclusions $M \rightarrow W$ and $N \rightarrow W$ respectively. We refer to the composition $N \to M$ of the inclusion $N \rightarrow W$ and the homotopy inverse of the inclusion $M \rightarrow W$ as \emph{the homotopy equivalence induced by $W$}.
\end{definition}

Note that the homotopy inverse of the inclusion $M \to W$, and hence the homotopy equivalence $N \to M$, are well-defined only up to homotopy. This suffices for our purposes since we will only consider the Whitehead torsion of these maps (see the discussion after \cref{prop:WT-homotopic-maps}).

\begin{proposition} \label{prop:WT-hcob}
Let $W$ be an $h$-cobordism between closed $n$-dimensional manifolds $M$ and $N$. Let $G = \pi_1(W)$ and let $w \colon G \rightarrow \left\{ \pm 1 \right\}$ be the orientation character of $W$.
\benum
\item\label{item-prop-WT-hcob-a} We have that $\tau(W,N) = (-1)^n \overline{\tau(W,M)} \in \Wh(G,w)$. 
\item\label{item-prop-WT-hcob-b} If $f \colon N \rightarrow M$ denotes the homotopy equivalence induced by $W$, then we have 
\[
\tau(f) = -\tau(W,M) + (-1)^n \overline{\tau(W,M)} \in \Wh(G,w)
\] 
where $\pi_1(M)$ is identified with $G$ via inclusion. In particular $\tau(f) \in \mathcal{I}_n(G,w)$. 
\eenum
\end{proposition}

\begin{proof}
For \eqref{item-prop-WT-hcob-a} this is the ``duality theorem'' \cite[p.~394]{Mi66}, translated into our conventions; see also the remark on \cite[p.~398]{Mi66}. 
Then \eqref{item-prop-WT-hcob-b} follows from part \eqref{item-prop-WT-hcob-a} and \cref{prop:WT-composition}.
\end{proof}

\subsection{Milnor's definition of Whitehead torsion} \label{ss:milnor}

So far we have used the Whitehead torsion of a (chain) homotopy equivalence. Now we recall Milnor's definition \cite[Sections 3--4]{Mi66}, which applies in broader settings. We will use this more general notion (only) in the surgery theoretic arguments of \cref{ss:wall}. 
We slightly alter the presentation compared to \cite{Mi66}, so that we do not need to manipulate (s-)bases directly, and we will see that these alterations do not change the result. 

We fix a group $G$, and assume that all modules are finitely generated left $\Z G$-modules. 

\begin{definition}
A module $X$ is called \emph{stably free} if $X \oplus (\Z G)^k$ is free for some $k \geq 0$, and an \emph{s-basis} (or stable basis) of $X$ is the data of an integer $k$ and a basis of $X \oplus (\Z G)^k$ for some $k$. An \emph{s-based module} will mean a stably free module equipped with an s-basis.
\end{definition}

If $X$ is free, then a basis of $X$ is also an s-basis. A trivial module has a unique basis, which determines the \emph{canonical s-basis} for it. The \emph{trivial s-based module} is the trivial module $0$ equipped with the canonical s-basis.

\begin{definition}
Suppose that $X$ and $Y$ are s-based modules, with $X \oplus (\Z G)^k$ and $Y \oplus (\Z G)^l$ free and based, and $f \colon X \rightarrow Y$ is an isomorphism. If $r \geq \max(k,l)$, then the s-bases determine bases for $X \oplus (\Z G)^r$ and $Y \oplus (\Z G)^r$ (by adding the standard bases of $(\Z G)^{r-k}$ and $(\Z G)^{r-l}$), and $f \oplus \id_{(\Z G)^r}$ is an isomorphism between them. The matrix of this isomorphism in the given bases represents an element of $\GL(\Z G)$, and hence $\Wh(G)$. This element is independent of the choice of~$r$, and it is called the \emph{Whitehead torsion of the isomorphism $f$}, denoted $\tau(f)$. 
\end{definition}

\begin{remark}
Two s-bases of a stably free module $X$ are said to be \emph{equivalent} if the isomorphism $\id_X$ between them has vanishing Whitehead torsion. As the Whitehead torsion is additive under composition, the Whitehead torsion of an isomorphism $f \colon X \rightarrow Y$ only depends on the equivalence classes of the s-bases of $X$ and $Y$. Hence in all statements about s-based modules, it is actually sufficient to equip the modules with an equivalence class of s-bases.
\end{remark}

\begin{definition} \label{def:milnor-ses}
Suppose that 
\[
0 \rightarrow X \rightarrow Y \rightarrow Z \rightarrow 0
\]
is a short exact sequence of s-based $\Z G$-modules. Since a stably free module (in particular, $Z$) is projective, such a sequence splits, and a splitting determines an isomorphism $Y \cong X \oplus Z$ of s-based modules (where the s-basis of $X \oplus Z$ is determined by those of $X$ and $Z$). Since changing the splitting changes the isomorphism by an element of $\E(\Z G)$, the Whitehead torsion of this isomorphism is independent of the choice of splitting (cf.\ \cite[Section 2]{Mi66}), and is defined to be the \emph{Whitehead torsion of the short exact sequence}. (In \cite{Mi66} this is denoted by $[xz/y]$, where $x,y,z$ are the (s-)bases of the modules.)
\end{definition}

Suppose that $C_*$ is a chain complex of s-based modules, and assume that the modules $H_i := H_i(C_*)$ are also s-based. In such a situation we will always assume that all but finitely many of the $C_i$ and $H_i$ are trivial s-based modules. By \cite[Section 4]{Mi66} the submodules $B_i \leq Z_i \leq C_i$ of boundaries and cycles are also stably free (and all but finitely many of them are trivial). 

\begin{definition} \label{def:milnor-cc}
We define the \emph{Whitehead torsion of an s-based chain complex $C_*$ with s-based homology} as follows. Choose s-bases for the modules $B_i$ and $Z_i$ such that all but finitely many of them are trivial s-based modules, and the s-bases of the remaining ones are chosen arbitrarily. Denote the Whitehead torsions of the short exact sequences
\[
0 \rightarrow Z_i \rightarrow C_i \rightarrow B_{i-1} \rightarrow 0 
\quad \quad \quad \text{and} \quad \quad \quad 
0 \rightarrow B_i \rightarrow Z_i \rightarrow H_i \rightarrow 0
\]
by $\gamma_i$ and $\zeta_i$ in $\Wh(G)$ respectively. Then only finitely many of the $\gamma_i$ and $\zeta_i$ are non-zero, and we define
\[
\tau(C_*) := \sum_i (-1)^i (\gamma_i + \zeta_i) \in \Wh(G).
\]
\end{definition}

One checks that changing the s-basis of $B_i$ or $Z_i$ has no effect on the sum, so the Whitehead torsion of $C_*$ does not depend on the choice of s-bases of $B_i$ and $Z_i$. By choosing an s-basis for $Z_i$ such that $\zeta_i$ vanishes (denoted by $b_ih_i$ in \cite{Mi66}) we see that this definition coincides with the one given in \cite[Sections 3--4]{Mi66}. 

\begin{definition} \label{def:milnor-cc-acyc}
The \emph{Whitehead torsion of an acyclic s-based chain complex} is defined by taking the canonical s-basis for all of the (trivial) homology groups and using \cref{def:milnor-cc}.
\end{definition}

Now suppose that $f \colon C_* \to D_*$ is a chain homotopy equivalence between finitely generated, free, based chain complexes, so that its mapping cone $\mathscr{C}(f)$ is acyclic, free and based. 
The Whitehead torsion of $f$ can be recovered using Milnor's definition of torsion, as follows.

\begin{proposition}[{\cite[\S 16]{Co73}}] \label{prop:milnor-gen}
We have $\tau(f) = \tau(\mathscr{C}(f)) \in \Wh(G)$, where \cref{defn:whitehead-torsion} and \cref{def:milnor-cc-acyc} are used to define the left and right hand side, respectively.
\end{proposition}

In fact, one could obtain a notion of Whitehead torsion $\tau(C_*)$ for acyclic, free, based chain complexes $C_*$ by replacing $\mathscr{C}(f)$ with $C_*$ in the formula of \cref{defn:whitehead-torsion}, and this would coincide with $\tau(C_*)$ as given by \cref{def:milnor-cc-acyc} (see \cite[\S 15--16]{Co73}). 

The next three lemmas follow directly from the definitions. 

\begin{lemma} \label{lem:milnor-les-isom}
Suppose that $f \colon X \rightarrow Y$ is a module isomorphism between s-based modules $X$ and $Y$, and $C_*$ is an acyclic chain complex such that $C_i = X$, $C_{i-1} = Y$ $($with the given s-bases$)$, $d_i = f$, and for $j \neq i,i-1$, $C_j$ is a trivial s-based module. Then $\tau(C_*) = (-1)^i \tau(f)$. 
\end{lemma}

\begin{lemma} \label{lem:milnor-les-ses}
Suppose that $0 \rightarrow X \rightarrow Y \rightarrow Z \rightarrow 0$ is a short exact sequence of s-based modules, and it is regarded as an acyclic chain complex $C_*$, with $C_{i+1} = X$, $C_i = Y$, and $C_{i-1} = Z$. Then its Whitehead torsion is $(-1)^i \tau(C_*)$. 
\end{lemma}

\begin{lemma} \label{lem:milnor-same}
Suppose that $C_*$ is an s-based chain complex with trivial differentials, so that $H_i = C_i$ for all $i$. Equip $H_i$ with the same s-basis as $C_i$ for each $i$. Then $\tau(C_*)=0$. 
\end{lemma}

The behaviour of torsion in short exact sequences is very important. 

\begin{theorem}[{\cite[Theorem 3.2]{Mi66}}] \label{thm:milnor-es}
Suppose that $C'_*$, $C_*$ and $C''_*$ are s-based chain complexes with s-based homology. Suppose also that $0 \rightarrow C'_* \rightarrow C_* \rightarrow C''_* \rightarrow 0$ is a short exact sequence of chain complexes such that the short exact sequence $0 \rightarrow C'_i \rightarrow C_i \rightarrow C''_i \rightarrow 0$ has vanishing Whitehead torsion for all $i$. Let $D_*$ be the associated homological long exact sequence, regarded as an acyclic s-based chain complex, with $D_{3i+2} = H'_i$, $D_{3i+1} = H_i$ and $D_{3i} = H''_i$. Then $\tau(C_*) = \tau(C'_*) + \tau(C''_*) + \tau(D_*)$.
\end{theorem}

The next lemma shows how we can use the Whitehead torsion to specify an s-basis of a homology group. 

\begin{lemma} \label{lem:milnor-sb}
Suppose that $C_*$ is an s-based chain complex and $i$ is an integer such that $H_j$ is s-based for $j \neq i$ and $H_i$ is stably free. Then for every $x \in \Wh(G)$ there is an s-basis of $H_i$ $($unique up to equivalence$)$ such that $\tau(C_*)=x$.
\end{lemma}

\begin{proof}
Choose an arbitrary s-basis of $H_i$ and let $y \in \Wh(G)$ be the Whitehead torsion of $C_*$ with respect to that choice. There is an $r \geq 0$ such that the chosen s-basis of $H_i$ determines a basis of $H_i \oplus (\Z G)^r$ and the element $(-1)^i(x-y) \in \Wh(G)$ can be represented by an $f \in \GL_r(\Z G)$. Then a suitable s-basis of $H_i$ can be obtained by applying $\id_{H_i} \oplus f$ to the original basis of $H_i \oplus (\Z G)^r$.
\end{proof}

Note that, even when $H_i$ is free, this construction only determines an s-basis of $H_i$ (as opposed to a basis), because it cannot be guaranteed that a given element of the Whitehead group is represented by an automorphism of a free module of a fixed rank. The following special case is used in the definition of the surgery obstruction (see \cite[Lemma 2.3]{Wall-SCM}). 

\begin{definition} \label{def:milnor-sb}
Suppose that $C_*$ is an s-based chain complex and $i$ is an integer such that $H_j$ is trivial for $j \neq i$ and $H_i$ is stably free. The \emph{distinguished s-basis of $H_i$} is defined (up to equivalence) by equipping $H_j$ with the canonical s-basis for $j \neq i$ and applying \cref{lem:milnor-sb} with $x=0$.
\end{definition}

If $H_i$ is trivial too, then it also has the canonical s-basis. However, in general the distinguished s-basis differs from the canonical one. 

\begin{lemma} \label{lem:milnor-sb-diff}
Suppose that $C_*$ is an acyclic s-based chain complex and $\tau(C_*)$ is its Whitehead torsion as in \cref{def:milnor-cc-acyc}. Then the transition matrix between the canonical and the distinguished s-basis of $H_i$ represents $(-1)^{i+1}\tau(C_*) \in \Wh(G)$. 
\end{lemma}

\begin{proof}
When constructing the distinguished s-basis via the proof of \cref{lem:milnor-sb}, start with the canonical s-basis of $H_i$ (so that $y = \tau(C_*)$).
\end{proof}

Finally, we can use Milnor's definition of Whitehead torsion in the geometric setting. 

\begin{definition}
Suppose that $Y$ is a finite CW complex with $\pi_1(Y) = G$ and $X$ is a subcomplex of $Y$ (possibly empty). Let $C_i(Y,X; \Z G)$ denote the relative cellular $i$-chains (which we equip with the basis given by the $i$-cells of $Y$ not in $X$) and let $H_i(Y,X; \Z G)$ denote the $i$th relative homology. Suppose that $H_i(Y,X; \Z G)$ is an s-based $\Z G$-module for all $i$. Then the \emph{Whitehead torsion of the pair $(Y,X)$}, denoted $\tau(Y,X)$, is defined to be $\tau(C_*(Y,X; \Z G))$. 
\end{definition}

\begin{lemma} \label{lem:milnor-sub}
Suppose that $Y$ is a finite CW complex with $\pi_1(Y) = G$ and $X$ is a subcomplex of $Y$ such that the inclusion $f \colon X \rightarrow Y$ is a homotopy equivalence. If $H_i(Y,X; \Z G) = 0$ is equipped with the canonical s-basis for all $i$, then $\tau(Y,X)=\tau(f)$.
\end{lemma}

\begin{proof}
Identify $\pi_1(X)$ with $G$ via $\pi_1(f)$. For every $i$, since $C_i(Y,X; \Z G) = C_i(Y; \Z G) / C_i(X; \Z G) = \mathscr{C}(f_*)_i / \mathscr{C}((\id_X)_*)_i$, there is a short exact sequence 
\[
0 \rightarrow \mathscr{C}((\id_X)_*)_i \rightarrow \mathscr{C}(f_*)_i \rightarrow C_i(Y,X; \Z G) \rightarrow 0
\]
The basis of $\mathscr{C}((\id_X)_*)_i$ is given by $(i-1)$-cells and $i$-cells of $X$, the basis of $\mathscr{C}(f_*)_i$ by the $(i-1)$-cells of $X$ and the $i$-cells of $Y$, and the basis of $C_i(Y,X; \Z G)$ by the $i$-cells of $Y$ not in $X$. Hence the short exact sequence has vanishing Whitehead torsion. If $H_i(\mathscr{C}((\id_X)_*)) = H_i(\mathscr{C}(f_*)) = 0$ are also equipped with the canonical s-basis for all $i$, then by \cref{thm:milnor-es} $\tau(\mathscr{C}(f_*)) = \tau(\mathscr{C}((\id_X)_*)) + \tau(C_*(Y,X; \Z G))$. By \cref{prop:milnor-gen} and the definitions we have $\tau(\mathscr{C}(f_*)) = \tau(f_*) = \tau(f)$, $\tau(\mathscr{C}((\id_X)_*)) = \tau((\id_X)_*) = \tau(\id_X) = 0$ and $\tau(C_*(Y,X; \Z G)) = \tau(Y,X)$.
\end{proof}

\begin{lemma} \label{lem:milnor-cyl}
Suppose that $X$ and $Y$ are finite CW complexes, $\pi_1(Y) = G$, and $f \colon X \rightarrow Y$ is a cellular homotopy equivalence. Identify $\pi_1(\Cyl_f)$ with $\pi_1(Y) = G$ via the retraction $\Cyl_f \rightarrow Y$. If $H_i(\Cyl_f,X; \Z G) = 0$ is equipped with the canonical s-basis for all $i$, then $\tau(\Cyl_f,X)=\tau(f)$.
\end{lemma}

\begin{proof}
Since $C_*(\Cyl_f,X;\Z G) \cong \mathscr{C}(f_*)$, we have $\tau(\Cyl_f,X) = \tau(C_*(\Cyl_f,X;\Z G)) = \tau(\mathscr{C}(f_*)) = \tau(f_*) = \tau(f)$ by the definitions and \cref{prop:milnor-gen}. (Alternatively, use \cref{lem:milnor-sub} and \cref{prop:map-cyl}.)
\end{proof}

We can now check that $\tau(Y,X)$ is invariant under simple homotopy equivalence, in a suitable sense. 

\begin{lemma} \label{lem:milnor-she-pairs}
Let $(Y,X)$ and $(Y',X')$ be finite CW pairs with identifications $\pi_1(Y) = \pi_1(Y') = G$. Suppose that $H_i(Y,X; \Z G)$ and $H_i(Y',X'; \Z G)$ are s-based for all $i$. Let $(F,f) \colon (Y,X) \rightarrow (Y',X')$ be a map of pairs such that $F \colon Y \rightarrow Y'$ and $f \colon X \rightarrow X'$ are cellular simple homotopy equivalences. Assume that $\pi_1(F) = \id_G$, so that $(F,f)$ induces $\Z G$-module isomorphisms $H_i(F,f) \colon H_i(Y,X; \Z G) \rightarrow H_i(Y',X'; \Z G)$, and further assume that $H_i(F,f)$ has vanishing Whitehead torsion for all $i$. Then $\tau(Y,X)=\tau(Y',X')$.
\end{lemma}

\begin{proof}
All chain complexes and homology appearing in this proof will be taken with $\Z G$ coefficients, with the local coefficient systems pulled back from $Y'$ using the standard inclusions and retractions (except in the case of the pair $(Y,X)$, where we use the given identification $\pi_1(Y) = G$). Using the induced maps on $\pi_1$, and hence on $\Wh(\pi_1({-}))$, all Whitehead torsions will be considered as elements of $\Wh(G)$.

Consider the short exact sequences
\begin{align*}
0 \rightarrow C_*(\Cyl_f, X) &\rightarrow C_*(\Cyl_F, Y) \rightarrow C_*(\Cyl_F, Y \cup \Cyl_f) \rightarrow 0 \\
0 \rightarrow C_*(\Cyl_f, X') &\rightarrow C_*(\Cyl_F, Y') \rightarrow C_*(\Cyl_F, Y' \cup \Cyl_f) \rightarrow 0.
\end{align*}
All of these chain complexes are acyclic (for the last term in each sequence, this follows from the exactness of the associated homological sequence), and we choose the canonical s-bases on their homologies. 
Then the associated homological long exact sequences have vanishing Whitehead torsion. 
Since $f$ and $F$ are simple, we deduce from \cref{lem:milnor-cyl} that $\tau(\Cyl_f, X) = \tau(\Cyl_F, Y) = 0$. By \cref{prop:map-cyl} and \cref{lem:milnor-sub} we have that $\tau(\Cyl_f, X') = \tau(\Cyl_F, Y') = 0$. Hence it follows from \cref{thm:milnor-es} that
\[
\tau(\Cyl_F, Y \cup \Cyl_f) = \tau(\Cyl_F, Y' \cup \Cyl_f) = 0.
\] 

Next, note that since $\pi_1(F) = \id_G$, the inclusion induces a morphism $C_*(Y, X) \rightarrow C_*(\Cyl_F,\Cyl_f)$ of $\Z G$-module chain complexes. Consider the short exact sequences
\begin{align*}
0 \rightarrow C_*(Y, X) &\rightarrow C_*(\Cyl_F,\Cyl_f) \rightarrow C_*(\Cyl_F, Y \cup \Cyl_f) \rightarrow 0 \\
0 \rightarrow C_*(Y', X') &\rightarrow C_*(\Cyl_F,\Cyl_f) \rightarrow C_*(\Cyl_F, Y' \cup \Cyl_f) \rightarrow 0.
\end{align*}
Recall that $H_i(\Cyl_F, Y \cup \Cyl_f) = H_i(\Cyl_F, Y' \cup \Cyl_f)=0$ have both been equipped with the canonical s-basis, for all $i$. We have $H_i(Y', X') \cong H_i(\Cyl_F,\Cyl_f)$, so we can equip the former with the s-basis given on the latter. 
Then, since $H_i(F,f)$ has vanishing Whitehead torsion, the same is true for the isomorphism $H_i(Y,X) \rightarrow H_i(\Cyl_F,\Cyl_f)$ induced by the inclusion. It follows that both of the homological long exact sequences associated to the above short exact sequences have vanishing Whitehead torsion.
So by \cref{thm:milnor-es}, $\tau(Y, X) = \tau(\Cyl_F,\Cyl_f) = \tau(Y', X')$. 
\end{proof}

\section{Realising elements of $\Wh(G)$ by maps between manifolds}
\label{section:realisation}

Throughout this section, fix $n \geq 4$, a finitely presented group $G$, and $\CAT \in \{\Diff, \PL, \TOP\}$, satisfying \cref{assumptions}. 

\subsection{Realising via $h$-cobordisms}
\label{ss:realising-cobordisms}

If $(W;M,N)$ is an $h$-cobordism, then we use the isomorphism $\pi_1(i_M)$ to identify $\pi_1(M)$ with $\pi_1(W)$, and regard $\tau(W,M) \in \Wh(\pi_1(W))$ as an element of $\Wh(\pi_1(M))$.
Here is the complete statement of the $s$-cobordism theorem, for closed manifolds. It is due to Smale, Barden, Mazur, Stallings, Kirby-Siebenmann, and Freedman-Quinn~\cite{Smale-h-cob,Barden,Mazur:1963-1,Stallings-Tata,Kirby-Siebenmann:1977-1,FQ}.

\begin{theorem}[$s$-cobordism theorem]	\label{theorem:scob}
    Let $M$ be a closed $\CAT$ $n$-manifold with $\pi_1(M) \cong G$, with $\CAT$ as in \cref{assumptions}.
	\benum
		\item\label{it:s-cob1} Let $(W;M,M')$ be an $h$-cobordism over $M$. Then $W$ is trivial over $M$, i.e.\ $W\cong M\times[0,1]$, via a $\CAT$-isomorphism restricting to the identity on $M$, if and only if its Whitehead torsion $\tau(W,M)\in\Wh(G)$ vanishes.
		\item\label{it:s-cob2} For every $x \in\Wh(G)$ there exists an $h$-cobordism $(W;M,M')$ with $\tau(W,M)=x$.
		\item\label{it:s-cob3} The function assigning to an $h$-cobordism $(W;M,M')$ its Whitehead torsion $\tau(W,M)$ yields a bijection from the $\CAT$-isomorphism classes relative to $M$ of $h$-cobordisms over $M$ to the Whitehead group $\Wh(G)$.
	\eenum
\end{theorem}

For $n \geq 5$, see \cite[p.~90]{RS82} for a discussion of part \eqref{it:s-cob2}. 
In the case $n=4$, \cite[Theorem~3.5]{KPR-survey} gives details of the proof of part \eqref{it:s-cob2}. In this case there is an extra subtlety that does not occur for $n \geq 5$, because one has to check that $M'$ has the same fundamental group as $M$. 

Recall that 
$\mathcal{I}_n(G,w)  :=  \{ y + (-1)^{n+1}\ol{y} \mid y \in \Wh(G,w)\} \leq \Wh(G,w)$.

\begin{corollary}\label{corollary:realisation-of-I-by-hom-equivs}
Let $M$ be a closed $\CAT$ $n$-manifold with $\pi_1(M) \cong G$ and orientation character $w \colon G \rightarrow \left\{ \pm 1 \right\}$, with $\CAT$ as in \cref{assumptions}. For every $x \in \mathcal{I}_n(G,w)$ there exists a closed $\CAT$ $n$-manifold $N$ and a homotopy equivalence $f \colon N \to M$ induced by an $h$-cobordism between $M$ and $N$ such that $\tau(f) = x$. 
\end{corollary}

\begin{proof}
Let $x \in \mathcal{I}_n(G,w)$ and write $x = -y + (-1)^n\ol{y}$ for some $y \in \Wh(G,w)$.  Apply \cref{theorem:scob} to obtain an $h$-cobordism $(W;M,N)$ from $M$ to some $n$-manifold $N$ with $\tau(W,M)=y$. If $f$ is the homotopy equivalence induced by $W$, then $\tau(f) = -y+ (-1)^n\ol{y} = x$ by \cref{prop:WT-hcob}~\eqref{item-prop-WT-hcob-b}.
\end{proof}

\begin{corollary} \label{cor:she+hcob-realise}
Let $M$ and $N$ be closed $\CAT$ $n$-manifolds, with $\CAT$ as in \cref{assumptions}. 
Suppose $M$ has $\pi_1(M) \cong G$ and orientation character $w \colon G \rightarrow \left\{ \pm 1 \right\}$. 
If there is a homotopy equivalence $f \colon N \rightarrow M$ such that $\tau(f) \in \mathcal{I}_n(G,w)$, then there exists a closed $\CAT$ $n$-manifold $P$ that is simple homotopy equivalent to $N$ and $h$-cobordant to $M$.
\end{corollary}

\begin{proof}
By \cref{corollary:realisation-of-I-by-hom-equivs} there is an $h$-cobordism between $M$ and some $n$-manifold $P$ such that $\tau(g) = \tau(f)$ for the induced homotopy equivalence $g \colon P \rightarrow M$. Moreover, if $g^{-1}$ denotes the homotopy inverse of $g$, then $g^{-1} \circ f \colon N \rightarrow P$ is a simple homotopy equivalence, because $\tau(g^{-1} \circ f) = \tau(g^{-1}) + g^{-1}_*(\tau(f)) = \tau(g^{-1}) + g^{-1}_*(\tau(g)) = \tau(g^{-1} \circ g) = \tau(\Id)= 0$ by \cref{prop:WT-composition}. 
\end{proof}

\subsection{Realising via the surgery exact sequence}
\label{ss:wall}

We recall the surgery exact sequence. Let $M$ be a closed $\CAT$ $n$-dimensional manifold with $\pi_1(M)=G$ and orientation character $w \colon G \to \{\pm 1\}$, for some $n \geq 4$, with $\CAT$ as in \cref{assumptions}.

\begin{definition} \label{def:hsset}
The \emph{homotopy structure set} of $M$, denoted $\mathcal{S}^{\mathrm{h}}(M)$, is by definition the set of pairs $(N,f)$, where $N$ is a closed $\CAT$ $n$-manifold and $f \colon N \rightarrow M$ is a homotopy equivalence, considered up to $h$-cobordism over $M$. That is, $[N,f] =[N',f'] \in \mathcal{S}^{\mathrm{h}}(M)$ if and only if there is an $h$-cobordism $(W;N,N')$, with inclusion maps $i \colon N \to W$ and $i' \colon N' \to W$, together with a map $F \colon W \to M$ such that $F \circ i = f$ and $F \circ i' =f'$.

We can similarly define the \textit{simple homotopy structure set} $\mathcal{S}^{\mathrm{s}}(M)$ to be the set of pairs $(N,f)$, where $N$ is a closed $\CAT$ $n$-manifold and $f \colon N \rightarrow M$ is a simple homotopy equivalence, considered up to $s$-cobordism over $M$. Such a pair $(N,f)$ also represents an element of $\mathcal{S}^{\mathrm{h}}(M)$, inducing a forgetful map $\mathcal{S}^{\mathrm{s}}(M) \rightarrow \mathcal{S}^{\mathrm{h}}(M)$.
\end{definition}

The main tool for studying the structure sets $\mathcal{S}^{\dagger}(M)$, for ${\dagger} \in \{{\mathrm{h}},{\mathrm{s}}\}$, is the Browder-Novikov-Sullivan-Wall surgery exact sequence \cite{Wall-SCM,FQ}; see also \cite[Theorem~5.12]{Luck-basic-intro-surgery}, \cite{Freedman-book-surgery-chapter}. The two versions of the sequence, and the commutative diagram formed by taking the forgetful maps between them, are as follows.

\begin{equation} \label{eq:surgery-diag}
\xymatrix{
\mathcal{N}(M \times [0,1],M \times \{0,1\}) \ar[r]^-{\sigma_{\mathrm{s}}} \ar@{=}[d] & L_{n+1}^{\mathrm{s}}(\Z G, w) \ar[r]^-{W_{\mathrm{s}}} \ar[d]_-{F} & \mathcal{S}^{\mathrm{s}}(M) \ar[r]^-{\eta_{\mathrm{s}}} \ar[d] & \mathcal{N}(M) \ar[r]^-{\sigma_{\mathrm{s}}} \ar@{=}[d] & L_n^{\mathrm{s}}(\Z G, w) \ar[d]^-{F} \\
\mathcal{N}(M \times [0,1],M \times \{0,1\}) \ar[r]^-{\sigma_{\mathrm{h}}} & L_{n+1}^{\mathrm{h}}(\Z G, w) \ar[r]^-{W_{\mathrm{h}}} & \mathcal{S}^{\mathrm{h}}(M) \ar[r]^-{\eta_{\mathrm{h}}} & \mathcal{N}(M) \ar[r]^-{\sigma_{\mathrm{h}}} & L_n^{\mathrm{h}}(\Z G, w)
}
\end{equation}

Next we explain the remaining terms and maps in these sequences.

\begin{itemize}[$\bullet$]
\item 
The sets of the normal bordism classes of degree one normal maps over $M \times [0,1]$ and $M$ are denoted
$\mathcal{N}(M \times [0,1],M \times \{0,1\})$ and $\mathcal{N}(M)$ respectively.  In the former case, the normal map is required to restrict to a $\CAT$ isomorphism over $M \times \{0,1\}$. 
These sets do not depend on the decoration~${\dagger}$. 
\item 
The groups $L_{n}^{\dagger}(\Z G, w)$ are the surgery obstruction groups. 
Elements of $L_{n}^{\mathrm{h}}(\Z G, w)$ are represented by nonsingular Hermitian forms over finitely generated free $\Z G$-modules for $n$ even, and by nonsingular formations over finitely generated free $\Z G$-modules for $n$ odd, with the involution on $\Z G$ determined by $w$. 
See e.g.~\cite{Ranicki-80-I} for the definitions. 
In the case of $L_{n}^{\mathrm{s}}(\Z G, w)$ the forms/formations are also required to be based and simple. The map $F$ is given by forgetting the bases.
\item 
The maps labelled $\sigma_{\dagger}$ are the surgery obstruction maps.  
For the definition of $\sigma_{\mathrm{s}}$, we take a degree one normal map $(f,b) \colon N \to M$, perform surgery below the middle dimension to make the map $[n/2]$-connected, and then produce the based, simple form or formation (for $n$ even or odd respectively) of the surgery kernel in the middle dimension(s), to obtain an element of $L_n^{\mathrm{s}}(\Z G,w)$. To define the map $\sigma_{\mathrm{h}}$ we perform the same procedure, and then forget the data of the bases, to obtain an element of $L_n^{\mathrm{h}}(\Z G,w)$.  One of the main theorems of surgery \cite{Wall-SCM}, \cite{Luck-basic-intro-surgery} is that the maps $\sigma_{\dagger}$ are well-defined. 
\item 
We write $\eta_{\dagger}$ for the normal invariant map. It sends $[N,f] \in \mathcal{S}^{\dagger}(M)$ to the normal bordism class of $f \colon N \rightarrow M$ covered by a suitable bundle map.
\item 
The map $W_{\dagger}$ is the Wall realisation map. Given $z \in L_{n+1}^{\dagger}(\Z G, w)$ and $[M_0,f_0] \in \mathcal{S}^{\dagger}(M)$, Wall realisation produces a new element $[M_0,f_0]^z = [M_1,f_1] \in \mathcal{S}^{\dagger}(M)$ together with a degree one normal bordism between $(M_0,f_0)$ and $(M_1,f_1)$ whose surgery obstruction equals $z$.
This determines an action of $L_{n+1}^{\dagger}(\Z G, w)$ on $\mathcal{S}^{\dagger}(M)$, and the map $W_{\dagger}$ is defined by acting on the equivalence class of the identity map, $[M,\Id_M] \in \mathcal{S}^{\dagger}(M)$.   
\end{itemize}

Each term has a natural basepoint, and the sequences are exact sequences of pointed sets. Furthermore, each nonempty fibre of $\eta_{\dagger}$ coincides with an orbit of the action of $L_{n+1}^{\dagger}(\Z G, w)$ on $\mathcal{S}^{\dagger}(M)$.

The forgetful maps $F$ fit into the following exact sequence of abelian groups:
\begin{align}\label{eq:SRR-sequence}
    &\cdots \to L_{n+1}^{\mathrm{s}}(\Z G, w) \xrightarrow{F} L_{n+1}^{\mathrm{h}}(\Z G, w) \xrightarrow{\psi}  \wh{H}^{n+1}(C_2;\Wh(G, w)) \xrightarrow{\varrho} L_n^{\mathrm{s}}(\Z G, w) \xrightarrow{F} L_n^{\mathrm{h}}(\Z G, w) \to \cdots
\end{align}

\begin{remark}
The only proof we could find in the literature for this sequence is due to Shaneson~\cite[Section~4]{Shaneson-GxZ}. There, Shaneson attributed its derivation to Rothenberg, by a different (unpublished) proof.   
\end{remark}

The homomorphism $\psi$ from~\eqref{eq:SRR-sequence} determines an action of $L^{\mathrm{h}}_{n+1}(\Z G, w)$ on $\wh{H}^{n+1}(C_2;\Wh(G, w))$, given by $x^z = x + \psi(z)$ for $x \in \wh{H}^{n+1}(C_2;\Wh(G, w))$ and $z \in L^{\mathrm{h}}_{n+1}(\Z G, w)$.
Here for the definition of the Tate group, $C_2$ acts on $\Wh(G, w)$ via the involution. Recall that $\wh{H}^{n+1}(C_2;\Wh(G, w)) \cong \mathcal{J}_n(G, w) / \mathcal{I}_n(G, w)$ and $\pi \colon \mathcal{J}_n(G, w) \to \mathcal{J}_n(G, w) / \mathcal{I}_n(G, w)$ is the quotient map. 
Define a map 
\[
\begin{aligned}
\widehat{\tau} \colon \mathcal{S}^{\mathrm{h}}(M) &\to \wh{H}^{n+1}(C_2;\Wh(G, w)) \\
[N,f] &\mapsto \pi(\tau(f)). 
\end{aligned}
\]
We check that this map is well-defined. By \cref{prop:WT-man}, the Whitehead torsion of a homotopy equivalence between manifolds lies in $\mathcal{J}_n(G, w)$. If we change the representative of $[N,f]$ we obtain an $h$-cobordism between $N$ and some $N'$ over $M$, and by \cref{prop:WT-hcob} this changes the torsion by an element of $\mathcal{I}_n(G, w)$. 
We now establish two properties of $\widehat{\tau}$.

\begin{proposition} \label{prop:psi-comm-diag}
The map $\widehat{\tau}$ is $L^{\mathrm{h}}_{n+1}(\Z G, w)$-equivariant. That is, for all $z \in L^{\mathrm{h}}_{n+1}(\Z G, w)$ and $[M_0,f_0] \in \mathcal{S}^{\mathrm{h}}(M)$, we have $\widehat{\tau}([M_0,f_0]^z) = \widehat{\tau}([M_0,f_0])^z$.
\end{proposition}

\begin{proposition} \label{prop:braid}
There is a commutative diagram 
\[
\xymatrix@R-0.5cm{
 & \mathcal{S}^{\mathrm{h}}(M) \ar[dd]_-{\widehat{\tau}} \ar[r]^-{\eta_{\mathrm{h}}} & \mathcal{N}(M) \ar[dd]^-{\sigma_{\mathrm{s}}} \ar[dr]^-{\sigma_{\mathrm{h}}} & \\
L^{\mathrm{h}}_{n+1}(\Z G, w) \ar[ur]^-{W_{\mathrm{h}}} \ar[dr]^-{\psi} & & & L^{\mathrm{h}}_n(\Z G, w) \\
 & \wh{H}^{n+1}(C_2;\Wh(G, w)) \ar[r]^-{\varrho} & L^{\mathrm{s}}_n(\Z G, w) \ar[ur]^-{F} & 
}
\]
between the surgery exact sequence and the exact sequence~\eqref{eq:SRR-sequence}. 
\end{proposition}

\begin{remark}
The commutative diagram in \cref{prop:braid} is part of a braid that appeared as Diagram (6) in the preprint version (arXiv:1711.04546v1) of \cite{JK18}, but was removed in the published version. In the case $\CAT = \TOP$, see also Ranicki \cite[p.359--360]{Ra86}. Neither reference contains a proof that the braid commutes. 
\end{remark}

In the proofs of \cref{prop:psi-comm-diag,prop:braid} we will use Milnor's definition of Whitehead torsion, introduced in \cref{ss:milnor}. The main reason for this is that Milnor's definition applies to chain complexes with nontrivial, but stably based, homology groups, and we will need to work in this setting in the upcoming proofs.
We will assume that trivial modules are equipped with the canonical s-basis, unless we specify a different one. 

Further, in these proofs every space $X$ is equipped with a fixed map to~$M$. This allows us to take chain complexes and homology groups with $\Z G$ coefficients. That is, if $f \colon X \rightarrow M$ is the fixed map, we write $C_*(X)$ for $C_*(X;\Z G^{\theta})$, where $\theta = \pi_1(f) \colon \pi_1(X) \rightarrow \pi_1(M) = G$. Moreover, via the isomorphism $\theta_* \colon \Wh(\pi_1(X)) \rightarrow \Wh(G)$ we will regard elements of $\Wh(\pi_1(X))$ as elements of $\Wh(G)$.

\begin{proof}[Proof of \cref{prop:psi-comm-diag}]
The case of odd $n$ was proved in Shaneson \cite[Lemma 4.2]{Shaneson-GxZ}, so we assume that $n=2q$ is even. 

An element $z \in L^{\mathrm{h}}_{n+1}(\Z G, w)$ is represented by a formation $(H;L_0,L_1)$, where $H$ is a $(-1)^q$-symmetric hyperbolic form over $\Z G$ (of rank $2k$) and $L_0$ and $L_1$ are lagrangians in $H$. Let $[M_0,f_0] \in \mathcal{S}^{\mathrm{h}}(M)$. First we recall the definition of the action of $z$ on $[M_0,f_0]$. 

Consider $M_0 \times I$ and add $k$ trivial $q$-handles to $M_0 \times \{1\}$. Then we obtain a cobordism $W_0$ between $M_0$ and $N := M_0 \# k(S^q \times S^q)$. There is an isometry $H \cong H_q(\#^k S^q \times S^q)$ such that $L_0$ corresponds to the subgroup generated by the $\ast \times S^q$ 
(all homology is with $\Z G$ coefficients by default in this proof). 
Then $W_0$ can also be constructed from $N \times I$ by adding $k$ $(q+1)$-handles to $N \times \{0\}$ along a basis of $L_0$. 

Next add $k$ $(q+1)$-handles to $N \times I$ along a basis of $L_1$ in $N \times \{1\}$. This yields a cobordism~$W_1$ between~$N$ and some $M_1$, and we define $W = W_0 \cup_N W_1$. It is a cobordism between $M_0$ and $M_1$ over $M$ with the map $F \colon W \to M$, which restricts to the composition of the projection $M_0 \times I \rightarrow M_0$ and $f_0$ on $M_0 \times I$, and sends the extra handles to a point. Let 
\[
f_1 := F \big| _{M_1} \colon M_1 \to M, \quad\quad F_i := F \big| _{W_i} \colon W_i \rightarrow M\;\; (i=0,1), \quad\quad g := F \big| _N \colon N \to M.
\] 
Then $[M_0,f_0]^z = [M_1,f_1]$. 

Next we recall the definition of $\psi$. Fix a basis of $H$. On $L_i$ we consider the basis which corresponds to the gluing maps of the handles that are added to $N \times I$ to construct $W_i$. This determines a dual basis on $L_i^*$. The adjoint of the intersection form on $H$ is an isomorphism $H \rightarrow H^*$. By composing it with the dual $H^* \rightarrow L_i^*$ of the inclusion, we get a map $H \rightarrow L_i^*$ with kernel $L_i$, and hence a short exact sequence of based modules
\begin{equation}\label{eqn:short-exact-sequence-with-L-Lstar}
0 \rightarrow L_i \rightarrow H \rightarrow L_i^* \rightarrow 0.
\end{equation}
Let $x_i \in \Wh(G, w)$ be the Whitehead torsion of the short exact sequence \eqref{eqn:short-exact-sequence-with-L-Lstar}; see \cref{def:milnor-ses}.

If the basis of $H$ is given by the homology classes corresponding to the spheres in $H_q(\#^k S^q \times S^q)$, then $x_0=0$ and $\psi(z)=\pi(x_1)$ by definition; see \cite[p.\ 312]{Shaneson-GxZ} and the correspondence between the different descriptions of the odd-dimensional L-groups given at the end of \cite[Chapter~6]{Wall-SCM}. Note that \cite[p.\ 312]{Shaneson-GxZ} claims that there is a map $L^{\mathrm{h}}_{n+1}(\Z G, w) \rightarrow \Wh(\Z G, w)$, but actually the map described there is only well-defined in the quotient $\Wh(\Z G, w) / \mathcal{I}_n(G, w)$ (with image in $\mathcal{J}_n(G, w) / \mathcal{I}_n(G, w)$), because the same element of $L^{\mathrm{h}}_{n+1}(\Z G, w)$ can be represented by different automorphisms $\alpha$. This has no effect on any other arguments, as the map $\psi$ is meant to take values in $\wh{H}^{n+1}(C_2;\Wh(G, w))$ anyway. Correspondingly note that, while we defined $x_i$ using a specific basis of $L_i$, changing the basis of $L_i$ would change the value of $x_i$ by adding an element of~$\mathcal{I}_n(G, w)$. 

Now observe that changing the basis of $H$ has the same effect on $x_0$ and $x_1$, so $x_1-x_0$ is independent of the choice of basis of $H$. Therefore we have, independently of the choice of basis of $H$, that $\psi(z)=\pi(x_1-x_0) \in \wh{H}^{n+1}(C_2;\Wh(G, w))$.

Since 
\[
\widehat{\tau}([M_0,f_0]^z) = \widehat{\tau}([M_1,f_1]) = \pi(\tau(f_1)) = \pi(\tau(f_0)) + \pi(\tau(f_1)- \tau(f_0))
\]
and 
\[
\widehat{\tau}([M_0,f_0])^z = \widehat{\tau}([M_0,f_0]) + \psi(z) = \pi(\tau(f_0)) + \pi(x_1-x_0)
\]
our goal is to show that $\pi(\tau(f_1)- \tau(f_0)) = \pi(x_1-x_0)$. We will in fact prove a stronger statement, that $\tau(f_1)- \tau(f_0) = \pm(x_1-x_0)$, by relating both $x_i$ and $\tau(f_i)$ to the bordism $F_i \colon W_i \rightarrow M$.

First we consider the triple $(\Cyl_{F_i},W_i,N)$, where $\Cyl_{F_i}$ denotes the mapping cylinder of $F_i$. We compute the relative homology groups as follows. 
\begin{compactitem}[$\bullet$]
\item We have \[H_*(W_i,N) \cong H_{q+1}(W_i,N) \cong L_i,\] because $W_i$ is constructed from $N \times I$ by adding $(q+1)$-handles along $L_i$. 
\item We have \[H_*(\Cyl_{F_i},N) \cong H_{q+1}(\Cyl_{F_i},N) \cong \ker(H_q(g)) \cong H,\] because $\Cyl_{F_i} \simeq M \simeq M_0$, and we have $N = M_0 \# k(S^q \times S^q)$ by construction; 
\item
Since $W_i$ is constructed from $M_i \times I \simeq M \simeq \Cyl_{F_i}$ by adding $q$-handles,  
\[H_*(\Cyl_{F_i},W_i) \cong H_{q+1}(\Cyl_{F_i},W_i) \cong \ker(H_q(F_i)).\]
If $j_i \colon N \to W_i$ denotes the inclusion, then since $g = F_i \circ j_i$ we have that $\ker(H_q(g)) = H_q(j_i)^{-1}(\ker(H_q(F_i)))$. Since $H_q(j_i)$ is surjective, we have
\[
\ker(H_q(F_i)) = H_q(j_i)(\ker(H_q(g))) \cong \ker(H_q(g)) / \ker(H_q(j_i)) \cong H/L_i \cong L_i^*
\]
using the previous two items, and \eqref{eqn:short-exact-sequence-with-L-Lstar}. 
Hence $H_{q+1}(\Cyl_{F_i},N) \cong L_i^*$.
\end{compactitem}

So all of these homology groups are free (over $\Z G$), and we equip them with the previously chosen bases. The long exact sequence of the triple 
\begin{equation}\label{eqn:les-of-triple}
\cdots \rightarrow H_{q+2}(\Cyl_{F_i},W_i) \rightarrow H_{q+1}(W_i,N) \rightarrow H_{q+1}(\Cyl_{F_i},N) \rightarrow H_{q+1}(\Cyl_{F_i},W_i) \rightarrow H_q(W_i,N) \rightarrow \cdots 
\end{equation} 
is therefore identified with the short exact sequence \eqref{eqn:short-exact-sequence-with-L-Lstar}. If we regard \eqref{eqn:les-of-triple} as an acyclic based free chain complex, then by \cref{lem:milnor-les-ses} its Whitehead torsion is $(-1)^{3q+4} x_i = (-1)^q x_i$. Here note that the term $H \cong H_{q+1}(\Cyl_{F_i},N)$ lies in degree $3q+4$ of the long exact sequence of the triple thought of as a chain complex. 

So by \cref{thm:milnor-es} we have
\[
\tau(\Cyl_{F_i},N) = \tau(W_i,N) + \tau(\Cyl_{F_i},W_i) + (-1)^q x_i.
\]
The cobordism $W_i$ is simple homotopy equivalent to $N$ with a $(q+1)$-cell attached for each basis element of $L_i$, so \[C_*(W_i,N) \cong C_{q+1}(W_i,N) \cong L_i \cong H_{q+1}(W_i,N) \cong H_*(W_i,N).\] Moreover, the canonical isomorphism $C_{q+1}(W_i,N) \cong H_{q+1}(W_i,N)$ is compatible with the chosen bases, so by \cref{lem:milnor-same}, $\tau(W_i,N) = 0$. We also have $\Cyl_{F_i} \simeq_{\mathrm{s}} M \simeq_{\mathrm{s}} \Cyl_g$ by \cref{prop:map-cyl} (and this simple homotopy equivalence commutes with the inclusions of $N$).  So if we equip $H_*(\Cyl_g,N) \cong H_{q+1}(\Cyl_g,N) \cong \ker(H_q(g)) \cong H$ with the previously fixed basis of $H$, then by \cref{lem:milnor-she-pairs} $\tau(\Cyl_{F_i},N) = \tau(\Cyl_g,N)$. Hence
\begin{equation}\label{eq:xi}
x_i = (-1)^{q+1}(\tau(\Cyl_{F_i},W_i) - \tau(\Cyl_g,N)).
\end{equation}

Next consider the triple $(\Cyl_{F_i},W_i,M_i)$. Since $W_i$ can be obtained from $M_i \times I$ by adding $k$ trivial $q$-handles, we have $W_i \simeq_{\mathrm{s}} M_i \vee (\bigvee^k S^q)$. The basis of $C_*(W_i,M_i) \cong C_q(W_i,M_i) \cong C_q(\bigvee^k S^q)$ is given by the spheres, and we choose the corresponding basis also in $H_*(W_i,M_i) \cong H_q(W_i,M_i) \cong H_q(\bigvee^k S^q)$, so that $\tau(W_i,M_i) = 0$ by \cref{lem:milnor-same}. Since $f_i \colon M_i \rightarrow M$ is a homotopy equivalence and $\Cyl_{F_i} \simeq_{\mathrm{s}} M \simeq_{\mathrm{s}} \Cyl_{f_i}$, we have $H_*(\Cyl_{F_i},M_i) \cong H_*(\Cyl_{f_i},M_i) = 0$ and $\tau(\Cyl_{F_i},M_i) = \tau(\Cyl_{f_i},M_i) = \tau(f_i)$ (see \cref{lem:milnor-cyl,lem:milnor-she-pairs}). In the homological long exact sequence of the triple $(\Cyl_{F_i},W_i,M_i)$ all terms are trivial except for the isomorphism
\begin{equation}\label{equation:an-isomorphism}
0 \rightarrow H_{q+1}(\Cyl_{F_i},W_i) \xrightarrow{\cong} H_q(W_i,M_i) \rightarrow 0.
\end{equation}
We equipped $H_{q+1}(\Cyl_{F_i},W_i) \cong L_i^*$ with the dual of the basis of $L_i$. Since the handles added to $M_i \times I$ to obtain $W_i$ (determining the basis of $H_q(W_i,M_i)$) are the duals of the handles added to $N \times I$ to obtain $W_i$ (corresponding to the basis if $L_i$), the isomorphism~\eqref{equation:an-isomorphism} preserves the basis. Therefore its Whitehead torsion vanishes, and by \cref{lem:milnor-les-isom} the same is true when \eqref{equation:an-isomorphism} is regarded as a long exact sequence. So by \cref{thm:milnor-es} we have  
$\tau(f_i) = \tau(\Cyl_{F_i},W_i)$.

By combining this with \eqref{eq:xi} we get that 
\[
\tau(f_1)- \tau(f_0) = \tau(\Cyl_{F_1},W_1) - \tau(\Cyl_{F_0},W_0) = (-1)^{q+1}(x_1-x_0)
\] 
as required.
\end{proof}

\begin{proof}[Proof of \cref{prop:braid}]
The first triangle commutes by \cref{prop:psi-comm-diag} combined with the fact that $\widehat{\tau}([M,\Id])=0$. 
The commutativity of the last triangle follows from the commutativity of the right-hand square of diagram \eqref{eq:surgery-diag}.
 
So we need to prove that the square commutes. Let $[N,f] \in \mathcal{S}^{\mathrm{h}}(M)$, then $\eta_{\mathrm{h}}([N,f]) \in \mathcal{N}(M)$ is the normal bordism class of $f \colon N \rightarrow M$ (with an appropriate bundle map), and $\widehat{\tau}([N,f]) = \pi(\tau(f))$. We need to determine the image of these elements in $L^{\mathrm{s}}_n(\Z G, w)$.

First assume that $n=2q$ is even. Since $f$ is $q$-connected, $H_i(\Cyl_f,N) = 0$ for $i \neq q+1$, and $\sigma_{\mathrm{s}}(\eta_{\mathrm{h}}([N,f]))$ is defined as a form on $\ker(H_q(f)) \cong H_{q+1}(\Cyl_f,N)$, which is equipped with the distinguished s-basis from \cref{def:milnor-sb}, see \cite[Lemmas 2.3 and 5.1, Theorem 5.6]{Wall-SCM}. Since $f$ is a homotopy equivalence, $H_{q+1}(\Cyl_f,N)$ is trivial and $\tau(\Cyl_f,N)=\tau(f)$ by \cref{lem:milnor-cyl}. So by \cref{lem:milnor-sb-diff} the transition matrix between the canonical and the distinguished s-basis of $H_{q+1}(\Cyl_f,N)$ has Whitehead torsion $(-1)^q\tau(f)$. In $L^{\mathrm{s}}_n(\Z G, w)$ the same element $\sigma_{\mathrm{s}}(\eta_{\mathrm{h}}([N,f]))$ is also represented by a standard hyperbolic form with a basis such that the transition matrix between the standard and the chosen basis has Whitehead torsion $(-1)^q\tau(f)$. 

The image of $\pi(\tau(f))$ in $L^{\mathrm{s}}_n(\Z G, w)$ is represented by a standard hyperbolic form with any basis that has the following property: if $x$ denotes the Whitehead torsion of the transition matrix between the standard and the chosen basis, then $x \in \mathcal{J}_n(G, w)$ and $\pi(x) = \pi(\tau(f))$ (see \cite[p.\ 312]{Shaneson-GxZ}). In particular, the representative of $\sigma_{\mathrm{s}}(\eta_{\mathrm{h}}([N,f]))$ constructed above also represents $\varrho(\widehat{\tau}([N,f]))$, noting that $\pi((-1)^q\tau(f)) = \pi(\tau(f))$, because $\wh{H}^{n+1}(C_2;\Wh(G, w))$ is $2$-torsion. Hence the square commutes if $n$ is even.

Next consider the case when $n=2q+1$ is odd. Let $U$ denote a tubular neighbourhood of a disjoint union of embeddings $S^q \rightarrow N$ representing a generating set of $\ker(\pi_q(f))$ and let $N_0 = N \setminus \interior U$. We identify $H_q(\partial U)$ with the standard hyperbolic form, then $H_{q+1}(U, \partial U)$ (more precisely, its image under the boundary map) corresponds to the standard lagrangian. We can assume that $f \big| _U$ is constant, so the restrictions of $f$ determine a map of pairs $(f \big| _{N_0},f \big| _{\partial U}) \colon (N_0, \partial U) \rightarrow (M,\ast)$. Then $\ker(H_{q+1}(N_0, \partial U) \rightarrow H_{q+1}(M,\ast)) \cong H_{q+2}(\Cyl_{f|_{N_0}}, N_0 \cup \Cyl_{f|_{\partial U}})$ determines another lagrangian in $H_q(\partial U)$. We equip $H_q(\partial U)$ and $H_{q+1}(U, \partial U)$ with their standard bases, and $H_{q+2}(\Cyl_{f|_{N_0}}, N_0 \cup \Cyl_{f|_{\partial U}})$ with the distinguished s-basis from \cref{def:milnor-sb}.
Then $\sigma_{\mathrm{s}}(\eta_{\mathrm{h}}([N,f]))$ is represented by the formation $(H_q(\partial U);H_{q+1}(U, \partial U),H_{q+2}(\Cyl_{f|_{N_0}}, N_0 \cup \Cyl_{f|_{\partial U}}))$, see \cite[Section 6]{Wall-SCM}. Since $f$ is a homotopy equivalence, we can take the empty generating set for $\ker(\pi_q(f))$. Then $H_q(\partial U) = 0$, so we have the trivial formation, with the standard basis on the ambient form and on the first lagrangian, and the distinguished s-basis on the second lagrangian, $H_{q+2}(\Cyl_f, N)$. By \cref{lem:milnor-sb-diff} applied with $C_* = C_*(\Cyl_f,N)$, the transition matrix between the canonical and the distinguished s-basis has Whitehead torsion 
$(-1)^{q+1}\tau(\Cyl_f,N)$, which by \cref{lem:milnor-cyl} equals $(-1)^{q+1}\tau(f)$.

In $L^{\mathrm{s}}_n(\Z G, w)$ the same element $\sigma_{\mathrm{s}}(\eta_{\mathrm{h}}([N,f]))$ is also represented by a formation on the standard hyperbolic form given by the standard lagrangians, such that the ambient form and the first lagrangian are equipped with their standard bases, and for the second lagrangian the transition matrix between the standard and the chosen basis has Whitehead torsion $(-1)^{q+1}\tau(f)$. 

The image of $\pi(\tau(f))$ in $L^{\mathrm{s}}_n(\Z G, w)$ is represented by a formation on the standard hyperbolic form given by the standard lagrangians, such that the ambient form and the first lagrangian are equipped with their standard bases, and if $x$ denotes the Whitehead torsion of the transition matrix between the standard and the chosen basis of the second lagrangian, then $x \in \mathcal{J}_n(G, w)$ and $\pi(x) = \pi(\tau(f))$ (see \cite[p.\ 310]{Shaneson-GxZ}). Again we use that we can ignore signs because $\wh{H}^{n+1}(C_2;\Wh(G, w))$ is $2$-torsion. So the representative of $\sigma_{\mathrm{s}}(\eta_{\mathrm{h}}([N,f]))$ constructed above also represents $\varrho(\widehat{\tau}([N,f]))$, showing that the square also commutes if $n$ is odd.
\end{proof}

\begin{corollary}
Let $x \in \wh{H}^{n+1}(C_2;\Wh(G, w))$. If $x \in \image \psi$, then there is a $\CAT$ $n$-manifold $N$ and a homotopy equivalence $f \colon N \to M$ such that $\pi(\tau(f)) = x$.
\end{corollary}

\begin{proof}
If $x = \psi(z)$ for some $z \in L^{\mathrm{h}}_{n+1}(\Z G, w)$, then let $[N,f] := W_{\mathrm{h}}(z)$. By \cref{prop:psi-comm-diag} we have $\pi(\tau(f)) = \widehat{\tau}([N,f]) = \widehat{\tau}(W_{\mathrm{h}}(z)) = \psi(z) = x$.
\end{proof}

\begin{corollary} \label{cor:im-psi}
$\image \psi \subseteq \varrho^{-1}(\image \sigma_{\mathrm{s}})$.
\end{corollary}

\begin{proof}
By the exactness of the sequence \eqref{eq:SRR-sequence}, and since $\Id_M$ has vanishing surgery obstruction, $\image \psi = \ker \varrho \subseteq \varrho^{-1}(\image \sigma_{\mathrm{s}})$.
\end{proof}

\begin{corollary} \label{cor:braid} 
$\image \widehat{\tau} = \varrho^{-1}(\image \sigma_{\mathrm{s}})$.
\end{corollary}

\begin{proof}
The commutativity of the square in \cref{prop:braid} implies that $\image \widehat{\tau} \subseteq \varrho^{-1}(\image \sigma_{\mathrm{s}})$. For the other direction assume that $x \in \wh{H}^{n+1}(C_2;\Wh(G, w))$ and $\varrho(x)=\sigma_{\mathrm{s}}(y)$ for some $y \in \mathcal{N}(M)$. Then $\sigma_{\mathrm{h}}(y) = F \circ \sigma_{\mathrm{s}}(y) = F \circ \varrho(x) = 0$, so $y = \eta_{\mathrm{h}}(v)$ for some $v \in \mathcal{S}^{\mathrm{h}}(M)$. Then we have $\varrho(x-\widehat{\tau}(v)) = \varrho(x)-\sigma_{\mathrm{s}}(\eta_{\mathrm{h}}(v)) = \varrho(x)-\sigma_{\mathrm{s}}(y) = 0$, so $x - \widehat{\tau}(v) = \psi(z)$ for some $z \in L^{\mathrm{h}}_{n+1}(\Z G, w)$. So by \cref{prop:psi-comm-diag} $x = \widehat{\tau}(v)^z = \widehat{\tau}(v^z) \in \image \widehat{\tau}$.
\end{proof}

\begin{remark} \label{rem:real}
There is a common generalisation of the two realisation techniques used in \cref{section:realisation}.
Let $L_{n+1}^{\mathrm{s},\tau}(\Z G, w)$ denote the surgery obstruction groups defined in \cite[Section 4]{Kreck-preprint}. Similarly to $L_{n+1}^{\mathrm{s}}(\Z G, w)$, the elements of $L_{n+1}^{\mathrm{s},\tau}(\Z G, w)$ are represented by based nonsingular forms/formations, but, unlike the elements of $L_{n+1}^{\mathrm{s}}(\Z G, w)$, they are not assumed to be simple. Thus there is a natural forgetful map $L_{n+1}^{\mathrm{s},\tau}(\Z G, w) \to L_{n+1}^{\mathrm{h}}(\Z G, w)$ and $L_{n+1}^{\mathrm{s}}(\Z G, w)$ is the kernel of a map $L_{n+1}^{\mathrm{s},\tau}(\Z G, w) \to \mathcal{J}_n(G, w) \leq \Wh(G, w)$ (see \cite{Kreck-preprint}). There is also an action of $\Wh(G, w)$ on $L_{n+1}^{\mathrm{s},\tau}(\Z G, w)$ (which determines a map $\Wh(G, w) \to L_{n+1}^{\mathrm{s},\tau}(\Z G, w)$ by acting on $0$) given by changing the basis; this action is transitive on the fibres of the map $L_{n+1}^{\mathrm{s},\tau}(\Z G, w) \to L_{n+1}^{\mathrm{h}}(\Z G, w)$. 

Wall's construction for realising elements of $L_{n+1}^{\mathrm{s}}(\Z G, w)$ ~\cite{Wall-SCM} can also be applied to the group $L_{n+1}^{\mathrm{s},\tau}(\Z G, w)$, but it will only produce homotopy equivalences, not simple homotopy equivalences. 
Thus there is a map \[W_{\mathrm{s},\tau} \colon L_{n+1}^{\mathrm{s},\tau}(\Z G, w) \rightarrow \mathcal{S}^{\mathrm{h}}_{\mathrm{sCob}}(M),\] where $\mathcal{S}^{\mathrm{h}}_{\mathrm{sCob}}(M)$ is the set of homotopy equivalences $f \colon N \rightarrow M$ (as in $\mathcal{S}^{\mathrm{h}}(M)$), considered up to $s$-cobordism over $M$ (as in $\mathcal{S}^{\mathrm{s}}(M)$).
The Whitehead torsion of the resulting homotopy equivalence is given by the previously mentioned map $L_{n+1}^{\mathrm{s},\tau}(\Z G, w) \to \mathcal{J}_n(G,w)$.

The realisation part of the $s$-cobordism theorem determines a map \[W_{\Wh} \colon \Wh(G, w) \to \mathcal{S}^{\mathrm{hCob}}_{\mathrm{sCob}}(M),\] where $\mathcal{S}^{\mathrm{hCob}}_{\mathrm{sCob}}(M)$ is the set of $h$-cobordisms $(W;M,N)$ considered up to $s$-cobordism rel.\ $M$. This is a special case of the above construction via the maps $\Wh(G, w) \to L_{n+1}^{\mathrm{s},\tau}(\Z G, w)$ and $\mathcal{S}^{\mathrm{hCob}}_{\mathrm{sCob}}(M) \rightarrow \mathcal{S}^{\mathrm{h}}_{\mathrm{sCob}}(M)$. The Whitehead torsion of the homotopy equivalence induced by the $h$-cobordism we get is of course in $\mathcal{I}_{n+1}(G,w)$.
The realisation of $L_{n+1}^{\mathrm{h}}(\Z G, w)$, and the map $W_{\mathrm{h}}$, 
can also be defined using $W_{\mathrm{s},\tau}$, 
by first choosing a basis for the form/formation representing an element of $L_{n+1}^{\mathrm{h}}(\Z G, w)$, equivalently, choosing a lift in  $L_{n+1}^{\mathrm{s},\tau}(\Z G, w)$. As the choice of basis is not unique, we obtain a homotopy equivalence only up to $h$-cobordism (i.e.\ an element of $\mathcal{S}^{\mathrm{h}}(M)$), and its Whitehead torsion (the value of $\widehat{\tau}$) is only well-defined in the quotient $\wh{H}^{n+1}(C_2;\Wh(G, w))$. 
See \cref{rem:diag} for a diagram summarising the different realisation maps, and their relationships to the simple homotopy manifold sets.
\end{remark}

\section{Simple homotopy manifold sets}
\label{s:manifold-sets}

In this section we prove \cref{thmx:bijections-with-manifold-sets-intro} about the characterisation of simple homotopy manifold sets. First we look at the analogous problem in the setting of CW complexes, which has a simpler answer. Then we consider the case of manifolds, which will rely on the results of \cref{section:realisation}.

\subsection{CW complexes} \label{ss:CW-complexes}

Fix a (finite, connected) CW complex $X$ and let $G = \pi_1(X)$. Our goal is to understand the set $\CC^{\mathrm{h}}_{\mathrm{s}}(X)$ defined below.

\begin{definition}
Let 
$
\CC^{\mathrm{h}}_{\mathrm{s}}(X)  :=  \left\{ \text{CW complexes } Y \mid Y \simeq X \right\} / \simeq_{\mathrm{s}}.
$
\end{definition}

Note that the representatives of elements of $\CC^{\mathrm{h}}_{\mathrm{s}}(X)$ consist only of CW complexes $Y$, and do not include chosen homotopy equivalences $Y \rightarrow X$.

We now need some auxiliary definitions.

\begin{definition}
If $Y$ is a CW complex homotopy equivalent to $X$, then let 
\[
t_X(Y) = \left\{ \tau(f) \mid f \colon Y \to X \text{ is a homotopy equivalence} \right\} \subseteq \Wh(G).
\]
\end{definition}

\begin{definition} \label{def:haut-act}
For $x \in \Wh(G)$ and $g \in \hAut(X)$ let $x^g = g_*(x) + \tau(g)$.
\end{definition}

Clearly $x^{\id} = x$ and $x^{g \circ g'} = (x^{g'})^g$ (see \cref{prop:WT-composition}), so this defines an action of $\hAut(X)$ on the set $\Wh(G)$.
With this notation we have the following theorem.

\begin{theorem} \label{theorem:chs-bij}
The map $t_X$ induces a well-defined bijection \[\widetilde{t}_X \colon \CC^{\mathrm{h}}_{\mathrm{s}}(X) \to \Wh(G) / \hAut(X).\] 
\end{theorem}

The proof will consist of the following sequence of lemmas.

\begin{lemma} \label{lem:sh-class1}
If $Y$ is a CW complex homotopy equivalent to $X$, then $t_X(Y)$ is an orbit of the action of $\hAut(X)$ on $\Wh(G)$.
\end{lemma}

\begin{proof}
By \cref{prop:WT-composition} we have $\tau(g \circ f) = \tau(f)^g$ for every homotopy equivalence $f \colon Y \to X$ and $g \in \hAut(X)$. If $f \colon Y \to X$ is a homotopy equivalence and $g \in \hAut(X)$, then $g \circ f$ is also a homotopy equivalence, so if $x \in t_X(Y)$, then $x^g \in t_X(Y)$ for every $g$. On the other hand, if $f,f' \colon Y \to X$ are homotopy equivalences, then there is a $g \in \hAut(X)$ such that $f' \simeq g \circ f$, showing that if $x,x' \in t_X(Y)$, then $x'=x^g$ for some $g$.
\end{proof}

\begin{lemma} \label{lem:sh-class2}
If $Y$ is a CW complex homotopy equivalent to $X$ and $Y \simeq_{\mathrm{s}} Z$, then $t_X(Y) = t_X(Z)$. 
\end{lemma}

\begin{proof}
Let $h \colon Z \to Y$ be a simple homotopy equivalence. If $x \in t_X(Y)$, i.e.\ $x = \tau(f)$ for some homotopy equivalence $f \colon Y \to X$, then $f \circ h \colon Z \to X$ is a homotopy equivalence with $\tau(f \circ h) = \tau(f) = x$ by \cref{prop:WT-composition}. This shows that $t_X(Y) \subseteq t_X(Z)$. We get similarly that $t_X(Z) \subseteq t_X(Y)$, therefore $t_X(Y) = t_X(Z)$.
\end{proof}

\begin{lemma} \label{lem:sh-class3}
If $Y$ and $Z$ are CW complexes homotopy equivalent to $X$ and $t_X(Y) = t_X(Z)$, then $Y \simeq_{\mathrm{s}} Z$. 
\end{lemma}

\begin{proof}
Let $x \in t_X(Y) = t_X(Z)$ be an arbitrary element, then there are homotopy equivalences $f \colon Y \to X$ and $f' \colon Z \to X$ with $\tau(f) = \tau(f') = x$. Let $f^{-1} \colon X \to Y$ denote the homotopy inverse of $f$, then by \cref{prop:WT-composition} $0 = \tau(\id_Y) = \tau(f^{-1}) + f^{-1}_*(\tau(f)) = \tau(f^{-1}) + f^{-1}_*(x)$. Hence we have $\tau(f^{-1} \circ f') = \tau(f^{-1}) + f^{-1}_*(\tau(f')) = \tau(f^{-1}) + f^{-1}_*(x) = 0$, showing that $f^{-1} \circ f' \colon Z \to Y$ is a simple homotopy equivalence.
\end{proof}

\begin{lemma}[{\cite[(24.1)]{Co73}}] \label{lem:sh-class4}
For every $x \in \Wh(G)$ there is a CW complex $Y$ and a homotopy equivalence $f \colon Y \rightarrow X$ such that $\tau(f) = x$.
\qed
\end{lemma}

\begin{proof}[Proof of \cref{theorem:chs-bij}]
By \cref{lem:sh-class1}, $t_X$ takes values in $\Wh(G) / \hAut(X)$ and by \cref{lem:sh-class2} it induces a well-defined map $\widetilde{t}_X$ on $\CC^{\mathrm{h}}_{\mathrm{s}}(X)$. \cref{lem:sh-class3,lem:sh-class4} imply that $\widetilde{t}_X$ is injective and surjective, respectively.
\end{proof}

\begin{remark} \label{rem:haut-act-spec}
There are two special cases when the action of $\hAut(X)$ on $\Wh(G)$ has a simpler description. 

First, assume that $\pi_1(g) = \id_G$ for every $g \in \hAut(X)$. Then $\tau(g \circ g') = \tau(g) + \tau(g')$ for every $g, g' \in \hAut(X)$ and $x^g = x + \tau(g)$ for every $x \in \Wh(G)$ and $g \in \hAut(X)$. This implies that $\left\{ \tau(g) \mid g \in \hAut(X) \right\}$ is a subgroup of $\Wh(G)$ and $\Wh(G) / \hAut(X)$ is the corresponding quotient group.

Second, assume that $\tau(g) = 0$ for every $g \in \hAut(X)$. Then $x^g = g_*(x)$ for every $x \in \Wh(G)$ and $g \in \hAut(X)$. This means that the action of $\hAut(X)$ factors through the map $\pi_1 \colon \hAut(X) \to \Aut(G)$, in particular $\hAut(X)$ acts via automorphisms of the group $\Wh(G)$.
\end{remark}

\subsection{Manifolds} \label{ss:class-man}

Now we consider the problem in the manifold setting. Fix a closed connected $\CAT$ $n$-manifold $M$ and let $G = \pi_1(M)$ with orientation character $w \colon G \rightarrow \left\{ \pm 1 \right\}$. Then we can consider either all $n$-manifolds that are homotopy equivalent to $M$, or those that are $h$-cobordant to $M$, up to simple homotopy equivalence, or manifolds homotopy equivalent to $M$ up to the equivalence relation generated by simple homotopy equivalence and $h$-cobordism. 

\begin{definition} \label{def:man-sets}
Let 
\[
\begin{aligned}
\M^{\mathrm{h}}_{\mathrm{s}}(M) & :=  \left\{ \text{closed $\CAT$ $n$-manifolds } N \mid N \simeq M \right\} / \simeq_{\mathrm{s}} \\
\M^{\hCob}_{\mathrm{s}}(M) & :=  \left\{ \text{closed $\CAT$ $n$-manifolds } N \mid N \text{ is $h$-cobordant to } M \right\} / \simeq_{\mathrm{s}} \\
\M^{\mathrm{h}}_{\mathrm{s},\hCob}(M) & :=  \left\{ \text{closed $\CAT$ $n$-manifolds } N \mid N \simeq M \right\} / \langle \simeq_{\mathrm{s}}, \hCob \rangle
\end{aligned}
\]
where $\langle \simeq_{\mathrm{s}}, \hCob \rangle$ denotes the equivalence relation generated by simple homotopy equivalence and $h$-cobordism.
\end{definition}

As in \cref{ss:CW-complexes}, $\hAut(M)$ acts on the set $\Wh(G,w)$, and a subset $t_M(N) \subseteq \Wh(G,w)$ is defined for every manifold $N$ that is homotopy equivalent to $M$. Let $\pi \colon \mathcal{J}_n(G,w) \rightarrow \wh{H}^{n+1}(C_2;\Wh(G,w))$ denote the quotient map. By \cref{prop:WT-man} if $f \colon N \rightarrow M$ is a homotopy equivalence between manifolds, then $\tau(f) \in \mathcal{J}_n(G,w)$, so we can also define 
\[
u_M(N) = \pi(t_M(N)) = \left\{ \pi(\tau(f)) \mid f \colon N \to M \text{ is a homotopy equivalence} \right\}
\]
which is a subset of $\wh{H}^{n+1}(C_2;\Wh(G,w))$.

\begin{theorem}\label{theorem:man-class}
\leavevmode
\begin{clist}{(a)}
\item\label{item:thm-man-class-a} The subset $\mathcal{J}_n(G,w) \subseteq \Wh(G,w)$ is invariant under the action of $\hAut(M)$. The action of $\hAut(M)$ on $\mathcal{J}_n(G,w)$ induces an action on the underlying set of $\wh{H}^{n+1}(C_2;\Wh(G,w))$. 

\item\label{item:thm-man-class-b} The map $t_M$ induces well-defined injective maps \[\widetilde{t}_M \colon \M^{\mathrm{h}}_{\mathrm{s}}(M) \to \mathcal{J}_n(G,w) / \hAut(M) \text{ and } \widetilde{t}'_M \colon \M^{\hCob}_{\mathrm{s}}(M) \to q(\mathcal{I}_n(G,w)),\] where $q \colon \mathcal{J}_n(G,w) \to \mathcal{J}_n(G,w) / \hAut(M)$ denotes the quotient map, so that \[q(\mathcal{I}_n(G,w)) \subseteq \mathcal{J}_n(G,w) / \hAut(M).\] 
Similarly, the map $u_M$ induces a well-defined map 
\[\widetilde{u}_M \colon \M^{\mathrm{h}}_{\mathrm{s},\hCob}(M) \to \wh{H}^{n+1}(C_2;\Wh(G,w)) / \hAut(M).\]

\item\label{item:thm-man-class-c} There is a commutative diagram
\[
\xymatrix{
\M^{\hCob}_{\mathrm{s}}(M) \ar@{^(->}[r] \ar@{^(->}[d]_-{\widetilde{t}'_M} & \M^{\mathrm{h}}_{\mathrm{s}}(M) \ar@{->>}[r] \ar@{^(->}[d]_-{\widetilde{t}_M} & \M^{\mathrm{h}}_{\mathrm{s},\hCob}(M) \ar[d]^-{\widetilde{u}_M} \\
q(\mathcal{I}_n(G,w)) \ar@{^(->}[r] & \mathcal{J}_n(G,w) / \hAut(M) \ar@{->>}[r] & \wh{H}^{n+1}(C_2;\Wh(G,w)) / \hAut(M)
}
\] 
In each row the first map is injective and the second map is surjective. In the top row the composition of the maps is trivial, while the bottom row is an exact sequence of pointed sets.

\item\label{item:thm-man-class-d} If \cref{assumptions} is satisfied, then $\widetilde{t}'_M$ is surjective, $\widetilde{u}_M$ is injective, and the top row is an exact sequence of pointed sets.

\item\label{item:thm-man-class-e} If \cref{assumptions} is satisfied, then, using the notation from \cref{ss:wall},  the subsets 
\[\varrho^{-1}(\image \sigma_{\mathrm{s}}) \subseteq \mathcal{J}_n(G,w) \text{ and } (\varrho \circ \pi)^{-1}(\image \sigma_{\mathrm{s}}) \subseteq \wh{H}^{n+1}(C_2;\Wh(G,w))\] 
are invariant under the action of $\hAut(M)$, and we have 
\[\im \widetilde{t}_M = (\varrho \circ \pi)^{-1}(\image \sigma_{\mathrm{s}}) / \hAut(M) \text{ and } \im \widetilde{u}_M = \varrho^{-1}(\image \sigma_{\mathrm{s}}) / \hAut(M).\]
\end{clist}
\end{theorem}

This implies \cref{thmx:bijections-with-manifold-sets-intro} (see also \cref{theorem:man-bij} below). Parts (\ref{item:thm-man-class-a}), (\ref{item:thm-man-class-b}) and (\ref{item:thm-man-class-c}) do not require \cref{assumptions} and so apply to smooth/PL 4-manifolds as well as topological 4-manifolds with arbitrary fundamental group. Parts (\ref{item:thm-man-class-d}) and (\ref{item:thm-man-class-e}) rely on results from \cref{section:realisation}, and so they require \cref{assumptions}.

We use the following lemmas in the proof of \cref{theorem:man-class}.

\begin{lemma} \label{lem:sh-class5}
Let $N$ and $P$ be manifolds homotopy equivalent to $M$. If $N \simeq_{\mathrm{s}} P$ or $N$ is $h$-cobordant to $P$, then $u_M(N) = u_M(P)$.
\end{lemma}

\begin{proof}
If $N \simeq_{\mathrm{s}} P$, then $t_M(N) = t_M(P)$ by \cref{lem:sh-class2}, so $u_M(N) = \pi(t_M(N)) = \pi(t_M(P)) = u_M(P)$. 

If $N$ is $h$-cobordant to $P$ and $h \colon P \rightarrow N$ is a homotopy equivalence induced by an $h$-cobordism, then $\pi(\tau(f \circ h)) = \pi(\tau(f)) + \pi(f_*(\tau(h))) = \pi(\tau(f))$ for every homotopy equivalence $f \colon N \rightarrow M$ (because $f_*(\tau(h)) \in \mathcal{I}_n(G,w)$ by \cref{prop:WT-hcob}). Hence $u_M(N) \subseteq u_M(P)$. We get similarly that $u_M(P) \subseteq u_M(N)$, therefore $u_M(N) = u_M(P)$.
\end{proof}

\begin{lemma} \label{lem:sh-class6}
Suppose that \cref{assumptions} is satisfied, and $N$ and $P$ are $\CAT$ $n$-manifolds homotopy equivalent to $M$. If $u_M(N) = u_M(P)$, then there is a manifold $Q$ that is simple homotopy equivalent to $N$ and $h$-cobordant to $P$.
\end{lemma}

\begin{proof}
Since $u_M(N) = u_M(P)$, there are homotopy equivalences $f \colon N \rightarrow M$ and $g \colon P \rightarrow M$ such that $\pi(\tau(f)) = \pi(\tau(g))$, equivalently, $\tau(f) - \tau(g) \in \mathcal{I}_n(G,w)$. If $g^{-1}$ denotes the homotopy inverse of $g$, then $\tau(g^{-1} \circ f) = \tau(g^{-1}) + g^{-1}_*(\tau(f)) = \tau(g^{-1}) + g^{-1}_*(\tau(g)) + g^{-1}_*(\tau(f) - \tau(g)) = \tau(g^{-1} \circ g) + g^{-1}_*(\tau(f) - \tau(g)) = g^{-1}_*(\tau(f) - \tau(g))$ by \cref{prop:WT-composition}. So we can apply \cref{cor:she+hcob-realise} to the homotopy equivalence $g^{-1} \circ f \colon N \rightarrow P$.
\end{proof}

\begin{lemma} \label{lem:sh-class7}
Suppose that $N$ and $P$ are $CAT$ $n$-manifolds with $\CAT$ as in \cref{assumptions}.  If there is a manifold $Q$ that is simple homotopy equivalent to $N$ and $h$-cobordant to $P$, then there is a manifold $R$ that is $h$-cobordant to $N$ and simple homotopy equivalent to $P$.
\end{lemma}

\begin{proof}
Apply \cref{cor:she+hcob-realise} to the composition of a homotopy equivalence $P \rightarrow Q$ induced by an $h$-cobordism and a simple homotopy $Q \rightarrow N$.
\end{proof}

\begin{proposition} \label{prop:she+hcob}
Suppose that $N$ and $P$ are $\CAT$ $n$-manifolds with $\CAT$ as in \cref{assumptions}. Then the following are equivalent.  
\begin{clist}{(a)}
\item\label{item:prop-she-hcob-a} The manifolds $N$ and $P$ are equivalent under the equivalence relation generated by simple homotopy equivalence and $h$-cobordism.
\item\label{item:prop-she-hcob-b} There is a manifold $Q$ that is simple homotopy equivalent to $N$ and $h$-cobordant to $P$.
\end{clist}
\end{proposition}

\begin{proof}
A chain between $N$ and $P$ of alternating simple homotopy equivalences and $h$-cobordisms can be reduced to a chain of length two using \cref{lem:sh-class7}.
\end{proof}

We are now ready to begin the proof of \cref{theorem:man-class}.

\begin{proof}[Proof of \cref{theorem:man-class}]
\eqref{item:thm-man-class-a} If $g \in \hAut(M)$, then $\tau(g) \in \mathcal{J}_n(G,w)$ by \cref{prop:WT-man}. Moreover, $H_n(M;\Z^{w \circ \pi_1(g)}) \cong H_n(M;\Z^w) \cong \Z$, so $w \circ \pi_1(g) = w$. Thus $g_* \colon \Wh(G,w) \rightarrow \Wh(G,w)$ is compatible with the involution, so if $x = -(-1)^n \ol{x}$, then $g_*(x) = -(-1)^n \ol{g_*(x)}$. Hence, if $x \in \mathcal{J}_n(G,w)$, then $x^g \in \mathcal{J}_n(G,w)$, i.e. $\hAut(M)$ acts on $\mathcal{J}_n(G,w)$. 

If $x,y \in \mathcal{J}_n(G,w)$ and $x-y \in \mathcal{I}_n(G,w)$, then $x^g-y^g = (g_*(x) + \tau(g)) - (g_*(y) + \tau(g)) = g_*(x-y) \in \mathcal{I}_n(G,w)$ (because $g_*$ is compatible with the involution). Therefore there is an induced action on $\wh{H}^{n+1}(C_2;\Wh(G,w))$. 

\eqref{item:thm-man-class-b} If $f \colon N \rightarrow M$ is a homotopy equivalence for some $n$-manifold $N$, then $\tau(f) \in \mathcal{J}_n(G,w)$ by \cref{prop:WT-man}. Hence $t_M(N) \subseteq \mathcal{J}_n(G,w)$ for every manifold $N$ that is homotopy equivalent to $M$. \cref{lem:sh-class1,lem:sh-class2,,lem:sh-class3} can be used again to show that $t_M$ induces well-defined injective maps $\widetilde{t}_M \colon \M^{\mathrm{h}}_{\mathrm{s}}(M) \to \mathcal{J}_n(G,w) / \hAut(M)$ and $\widetilde{t}'_M \colon \M^{\hCob}_{\mathrm{s}}(M) \to \mathcal{J}_n(G,w) / \hAut(M)$. If $N$ is $h$-cobordant to $M$ and $f \colon N \rightarrow M$ is a homotopy equivalence induced by an $h$-cobordism, then $\tau(f) \in \mathcal{I}_n(G,w)$ by \cref{prop:WT-hcob}, showing that $t_M(N) \in q(\mathcal{I}_n(G,w))$. 

The analogue of \cref{lem:sh-class1} shows that $u_M(N)$ is an orbit of the action of $\hAut(M)$ on $\wh{H}^{n+1}(C_2;\Wh(G,w))$ (using that $\pi(\tau(g \circ f))= \pi(\tau(f)^g) = \pi(\tau(f))^g$ for every homotopy equivalence $f \colon N \to M$ and $g \in \hAut(M)$). \cref{lem:sh-class5} shows that $u_M$ induces a well-defined map on $\M^{\mathrm{h}}_{\mathrm{s},\hCob}(M)$.

\eqref{item:thm-man-class-c} This follows immediately from the definitions.

\eqref{item:thm-man-class-d} The surjectivity of $\widetilde{t}'_M$ and injectivity of $\widetilde{u}_M$ follow from \cref{corollary:realisation-of-I-by-hom-equivs} and \cref{lem:sh-class6}, respectively.

For the exactness of the top row, assume that $N$ is equivalent to $M$ under the equivalence relation generated by simple homotopy equivalence and $h$-cobordism. By \cref{prop:she+hcob} there is a manifold $P$ that is simple homotopy equivalent to $N$ and $h$-cobordant to $M$. This represents an element of $\M^{\hCob}_{\mathrm{s}}(M)$, which is mapped to the equivalence class of $N$ in $\M^{\mathrm{h}}_{\mathrm{s}}(M)$. 

\eqref{item:thm-man-class-e} First consider $\varrho^{-1}(\image \sigma_{\mathrm{s}})$, which is equal to $\image \widehat{\tau}$ by \cref{cor:braid}. This is $\hAut(M)$-invariant, because for any $[N,f] \in \mathcal{S}^{\mathrm{h}}(M)$ and $g \in \hAut(M)$ we have $\pi(\tau(f))^g = \pi(\tau(g \circ f)) = \widehat{\tau}([N,g \circ f])$. Since the action of $\hAut(M)$ on $\wh{H}^{n+1}(C_2;\Wh(G,w))$ is induced by its action on $\mathcal{J}_n(G,w)$, and $\varrho^{-1}(\image \sigma_{\mathrm{s}})$ is $\hAut(M)$-invariant, $\pi^{-1}(\varrho^{-1}(\image \sigma_{\mathrm{s}}))$ is $\hAut(M)$-invariant too.

Now suppose that $[N] \in \M^{\mathrm{h}}_{\mathrm{s},\hCob}(M)$. Then there is a homotopy equivalence $f \colon N \rightarrow M$ and $\widetilde{u}_M([N])$ is the orbit of $\pi(\tau(f)) = \widehat{\tau}([N,f]) \in \image \widehat{\tau} = \varrho^{-1}(\image \sigma_{\mathrm{s}})$. This shows that $\im \widetilde{u}_M \subseteq \varrho^{-1}(\image \sigma_{\mathrm{s}}) / \hAut(M)$. On the other hand, if $x \in \varrho^{-1}(\image \sigma_{\mathrm{s}}) = \image \widehat{\tau}$, then $x = \pi(\tau(f))$ for some homotopy equivalence $f \colon N \rightarrow M$, and then $[N] \in \M^{\mathrm{h}}_{\mathrm{s},\hCob}(M)$ and $\widetilde{u}_M([N])$ is the orbit of $x$. Therefore $\im \widetilde{u}_M \supseteq \varrho^{-1}(\image \sigma_{\mathrm{s}}) / \hAut(M)$.

By the commutativity of the diagram in part \eqref{item:thm-man-class-c}, $\im \widetilde{t}_M$ is contained in the inverse image of $\im \widetilde{u}_M = \varrho^{-1}(\image \sigma_{\mathrm{s}}) / \hAut(M)$, which is $\pi^{-1}(\varrho^{-1}(\image \sigma_{\mathrm{s}})) / \hAut(M)$. In the other direction, if $x \in \mathcal{J}_n(G, w)$ and $\pi(x) \in \varrho^{-1}(\image \sigma_{\mathrm{s}}) = \image \widehat{\tau}$, then there exists a manifold $N_1$ and a homotopy equivalence $f_1 \colon N_1 \to M$ with $\pi(\tau(f_1)) = \pi(x)$. It follows that $x = \tau(f_1) + y$ for some $y \in \mathcal{I}_n(G, w)$. Use $(f_1)_*$ to identify $\pi_1(N_1)$ with $G = \pi_1(M)$ (since $f_1$ is a homotopy equivalence, the orientation character of $N$ is also $w$). By \cref{corollary:realisation-of-I-by-hom-equivs} there exists a manifold $N_2$ and a homotopy equivalence $f_2 \colon N_2 \to N_1$ with $\tau(f_2) = y$. By \cref{prop:WT-composition} we have $\tau(f_1 \circ f_2) = \tau(f_1) + (f_1)_*(\tau(f_2)) = \tau(f_1) + \tau(f_2) = \tau(f_1) + y = x$ (where $(f_1)_* = \Id$ because we used $f_1$ to identify $\pi_1(N_1)$ with $G$). So for $[N_2] \in \M^{\mathrm{h}}_{\mathrm{s}}(M)$ we get that $\widetilde{t}_M([N_2])$ is the orbit of $x$, showing that $\im \widetilde{t}_M \supseteq (\varrho \circ \pi)^{-1}(\image \sigma_{\mathrm{s}}) / \hAut(M)$.
\end{proof}

In particular we have the following, which is immediate from \cref{theorem:man-class}.  

\begin{theorem} \label{theorem:man-bij}
If \cref{assumptions} is satisfied, then there are bijections
\[
\begin{aligned}
\widetilde{t}'_M &\colon \M^{\hCob}_{\mathrm{s}}(M) \xrightarrow{\cong} q(\mathcal{I}_n(G,w)) \\
\widetilde{t}_M &\colon \M^{\mathrm{h}}_{\mathrm{s}}(M) \xrightarrow{\cong} (\varrho \circ \pi)^{-1}(\image \sigma_{\mathrm{s}}) / \hAut(M) \\
\widetilde{u}_M &\colon \M^{\mathrm{h}}_{\mathrm{s},\hCob}(M) \xrightarrow{\cong} \varrho^{-1}(\image \sigma_{\mathrm{s}}) / \hAut(M).
\end{aligned}
\]
\end{theorem}

With extra input from the map $\psi$, the latter two statements become cleaner. 

\begin{corollary}
If \cref{assumptions} is satisfied and the map $\psi$ is surjective, then there are bijections $\M^{\mathrm{h}}_{\mathrm{s}}(M) \cong \mathcal{J}_n(G,w) / \hAut(M)$ and  $\M^{\mathrm{h}}_{\mathrm{s},\hCob}(M) \cong \wh{H}^{n+1}(C_2;\Wh(G,w)) / \hAut(M)$.
\end{corollary}

\begin{proof}
By \cref{cor:im-psi}, we have $\wh{H}^{n+1}(C_2;\Wh(G,w)) = \image \psi \subseteq \varrho^{-1}(\image \sigma_{\mathrm{s}})$. Hence $\varrho^{-1}(\image \sigma_{\mathrm{s}}) = \wh{H}^{n+1}(C_2;\Wh(G,w))$ and $(\varrho \circ \pi)^{-1}(\image \sigma_{\mathrm{s}}) = \mathcal{J}_n(G,w)$. 
\end{proof}

\begin{remark} \label{rem:haut-act-spec2}
If $\tau(g) \in \mathcal{I}_n(G,w)$ for every $g \in \hAut(M)$, then the proof of \cref{theorem:man-class}~\eqref{item:thm-man-class-a} shows that $\mathcal{I}_n(G,w)$ is also an invariant subset under the action of $\hAut(M)$, and $q(\mathcal{I}_n(G,w)) = \mathcal{I}_n(G,w) / \hAut(M)$. It also means that $\pi(\tau(g))=0$ for every $g$, so $\hAut(M)$ acts on the group $\wh{H}^{n+1}(C_2;\Wh(G,w))$ via automorphisms. 

As in \cref{rem:haut-act-spec}, if $\tau(g)=0$ for every $g$, then $\hAut(M)$ acts on the groups $\mathcal{I}_n(G,w)$ and $\mathcal{J}_n(G,w)$ via automorphisms. And if $\pi_1(g) = \id_G$ for every $g$, then $\mathcal{J}_n(G,w) / \hAut(M)$ and $\wh{H}^{n+1}(C_2;\Wh(G,w)) / \hAut(M)$ are quotient groups of $\mathcal{J}_n(G,w)$ and $\wh{H}^{n+1}(C_2;\Wh(G,w))$.
\end{remark}

We now consider some corollaries of \cref{theorem:man-class,theorem:man-bij}.

\begin{definition} \label{def:TU}
Let $M$ be a closed $\CAT$ $n$-manifold with $\pi_1(M) \cong G$ and orientation character $w \colon G \rightarrow \left\{ \pm 1 \right\}$. We define the subsets
\[
\begin{aligned}
T(M) &= \left\{ \tau(g) \mid g \in \hAut(M) \right\} \subseteq \mathcal{J}_n(G,w) \\
U(M) &= \left\{ \pi(\tau(g)) \mid g \in \hAut(M) \right\} \subseteq \wh{H}^{n+1}(C_2;\Wh(G,w)) 
\end{aligned}
\]
\end{definition}

\begin{proposition} \label{prop:hcob-not-she}
Let $M$ be a closed $\CAT$ $n$-manifold with $\CAT$ as in \cref{assumptions}. Let $G = \pi_1(M)$ with orientation character $w \colon G \rightarrow \left\{ \pm 1 \right\}$. Then $|\M^{\hCob}_{\mathrm{s}}(M)| > 1$ 
if and only if $\mathcal{I}_n(G,w) \setminus T(M)$ is nonempty.
\end{proposition}

\begin{proof}
Note that $T(M) = t_M(M) \subseteq \mathcal{J}_n(G,w)$ is the orbit of $0$ under the action of $\hAut(M)$. So $q(\mathcal{I}_n(G,w))$ contains more than one element if and only if $\mathcal{I}_n(G,w) \setminus T(M)$ is nonempty. By \cref{theorem:man-bij} this is equivalent to $\M^{\hCob}_{\mathrm{s}}(M)$ containing more than one element.
\end{proof}

\begin{proposition} \label{prop:he-not-she+hcob}
Let $M$ be a closed $\CAT$ $n$-manifold with $\CAT$ as in \cref{assumptions}. Let $G = \pi_1(M)$ with orientation character $w \colon G \rightarrow \left\{ \pm 1 \right\}$. Then $|\M^{\mathrm{h}}_{\mathrm{s},\hCob}(M)| > 1$ 
if and only if $\varrho^{-1}(\image \sigma_{\mathrm{s}}) \setminus U(M)$ is nonempty. In particular, it is a sufficient condition that $\image(\psi) \setminus U(M)$ is nonempty.
\end{proposition}

\begin{proof}
Again $U(M) = u_M(M) \subseteq \wh{H}^{n+1}(C_2;\Wh(G,w))$ is the orbit of $0$ under the action of $\hAut(M)$. So $\varrho^{-1}(\image \sigma_{\mathrm{s}}) / \hAut(M)$ contains more than one element if and only if $\varrho^{-1}(\image \sigma_{\mathrm{s}}) \setminus U(M)$ is nonempty. By \cref{theorem:man-bij} this is equivalent to $\M^{\mathrm{h}}_{\mathrm{s},\hCob}(M)$ containing more than one element. Finally, by \cref{cor:im-psi} we have $\varrho^{-1}(\image \sigma_{\mathrm{s}}) \setminus U(M) \supseteq \image(\psi) \setminus U(M)$. 
\end{proof}

\begin{proposition} \label{prop:he-not-she-equiv}
Let $M$ be a closed $\CAT$ $n$-manifold with $\CAT$ as in \cref{assumptions}. Then the following are equivalent: 
\begin{clist}{(i)}
\item\label{item:prop-he-not-she-equiv-i} $|\M^{\mathrm{h}}_{\mathrm{s}}(M)| > 1$.
\item\label{item:prop-he-not-she-equiv-ii} Either $|\M^{\hCob}_{\mathrm{s}}(M)| > 1$ or $|\M^{\mathrm{h}}_{\mathrm{s},\hCob}(M)| > 1$. 
\item\label{item:prop-he-not-she-equiv-iii} $(\varrho \circ \pi)^{-1}(\image \sigma_{\mathrm{s}}) \setminus T(M)$ is nonempty.
\end{clist}
\end{proposition}

\begin{proof}
$\eqref{item:prop-he-not-she-equiv-i} \Leftrightarrow \eqref{item:prop-he-not-she-equiv-ii}$. By \cref{theorem:man-class}~\eqref{item:thm-man-class-c} and~\eqref{item:thm-man-class-d} we have $|\M^{\mathrm{h}}_{\mathrm{s}}(M)|=1$ if and only if $|\M^{\hCob}_{\mathrm{s}}(M)|=1$ and $|\M^{\mathrm{h}}_{\mathrm{s},\hCob}(M)|=1$.

$\eqref{item:prop-he-not-she-equiv-i} \Leftrightarrow \eqref{item:prop-he-not-she-equiv-iii}$. Follows from \cref{theorem:man-bij} since $T(M)$ is the orbit of $0$.
\end{proof}

\begin{remark} \label{rem:diag}
We can relate simple homotopy manifold sets to the structure sets discussed in \cref{section:realisation}. As $\mathcal{S}^{\mathrm{h}}(M)$ contains all homotopy equivalences $f \colon N \rightarrow M$ up to $h$-cobordism, the forgetful map $[f] \mapsto [N]$ is a surjection $\mathcal{S}^{\mathrm{h}}(M) \rightarrow \M^{\mathrm{h}}_{\mathrm{s},\hCob}(M)$. Moreover, $\widehat{\tau} \colon \mathcal{S}^{\mathrm{h}}(M) \to \wh{H}^{n+1}(C_2;\Wh(G, w))$ descends to the map $\widetilde{u}_M$ (cf.\ the proof of \cref{theorem:man-class} \eqref{item:thm-man-class-e}). Similarly, there are surjections $\mathcal{S}^{\mathrm{hCob}}_{\mathrm{sCob}}(M) \rightarrow \M^{\hCob}_{\mathrm{s}}(M)$ and $\mathcal{S}^{\mathrm{h}}_{\mathrm{sCob}}(M) \rightarrow \M^{\mathrm{h}}_{\mathrm{s}}(M)$ (see \cref{rem:real} for the definitions of these structure sets), and $\widetilde{t}'_M$ and $\widetilde{t}_M$ are induced by the maps $\tau$ that take the Whitehead torsion of a homotopy equivalence (induced by the $h$-cobordism in the case of $\mathcal{S}^{\mathrm{hCob}}_{\mathrm{sCob}}(M)$). These observations, and the realisation maps from \cref{rem:real}, are summarised in the following commutative diagram. 
\[
\xymatrix@!C=1.75cm{
\Wh(G,w) \ar[rr] \ar[dd]_-{\substack{x \mapsto \\ -x+ \\ (-1)^n\ol{x}}} \ar[dr]^{W_{\Wh}} & & L_{n+1}^{\mathrm{s},\tau}(\Z G, w) \ar[rr] \ar[dd] \ar[dr]^-{W_{\mathrm{s},\tau}} & & L_{n+1}^{\mathrm{h}}(\Z G, w) \ar[dd]_(.35){\psi} \ar[dr]^-{W_{\mathrm{h}}} & \\
 & \mathcal{S}^{\mathrm{hCob}}_{\mathrm{sCob}}(M) \ar[rr] |!{[ur];[dr]}\hole \ar[dl]_-{\tau} \ar@{->>}[dd] |!{[dl];[dr]}\hole & & \mathcal{S}^{\mathrm{h}}_{\mathrm{sCob}}(M) \ar[rr] |!{[ur];[dr]}\hole \ar[dl]_-{\tau} \ar@{->>}[dd] |!{[dl];[dr]}\hole & & \mathcal{S}^{\mathrm{h}}(M) \ar[dl]_-{\widehat{\tau}} \ar@{->>}[dd] \\
\mathcal{I}_n(G,w) \ar[rr] \ar@{->>}[dd]_-{q} & & \mathcal{J}_n(G,w) \ar[rr]^(.35){\pi} \ar@{->>}[dd]_(.35){q} & & \wh{H}^{n+1}(C_2;\Wh(G, w)) \ar@{->>}[dd] & \\
 & \M^{\hCob}_{\mathrm{s}}(M) \ar[rr] |!{[ur];[dr]}\hole \ar[dl]_-{\widetilde{t}'_M} & & \M^{\mathrm{h}}_{\mathrm{s}}(M) \ar[rr] |!{[ur];[dr]}\hole \ar[dl]_-{\widetilde{t}_M} & & \M^{\mathrm{h}}_{\mathrm{s},\hCob}(M) \ar[dl]_-{\widetilde{u}_M} \\
q(\mathcal{I}_n(G,w)) \ar[rr] & & \mathcal{J}_n(G,w) / \hAut(M) \ar[rr] & & \wh{H}^{n+1}(C_2;\Wh(G,w)) / \hAut(M) & 
}
\]
\end{remark}

\part{The simple homotopy manifold set of $S^1 \times L$} \label{p:lens}

In this part we consider $S^1 \times L$, the product of the circle with a lens space $L$. We prove \cref{main-theorem,theorem:main-S^1xL,theorem:hcob-S^1xL,,theorem:s-hcob-S^1xL} about the simple homotopy manifold sets of $S^1 \times L$. 
As described in \cref{thmx:bijections-with-manifold-sets-intro,ss:intro-S1L}, we need to study the involution on $\Wh(C_\infty \times C_m)$ and the group $\hAut(S^1 \times L)$. \cref{s:inv-wh} contains our results about the groups $\mathcal{J}_n(C_\infty \times C_m)$, $\mathcal{I}_n(C_\infty \times C_m)$ and $\wh{H}^{n+1}(C_2;\Wh(C_\infty \times C_m))$ (for $n$ even), which rely on the results of \cref{p:algebra}. In \cref{section:torsion-of-self-equiv-lens-space-x-S1}, we prove that all homotopy automorphisms of $S^1 \times L$ are simple, and determine the automorphisms of $C_\infty \times C_m$ induced by them. In \cref{s:b-proof}, we combine these results to prove \cref{main-theorem,theorem:main-S^1xL,theorem:hcob-S^1xL,,theorem:s-hcob-S^1xL}.

\section{The involution on $\Wh(C_\infty \times C_m)$} \label{s:inv-wh}

In order to understand its involution, we first consider a direct sum decomposition of the Whitehead group $\Wh(C_\infty \times G)$. This decomposition is derived from the fundamental theorem for $K_1(\Z G[t,t^{-1}])$ ~\cite{Ba68}; see also \cite{Ra86}, \cite[III.3.6]{We13}.
Whilst this theorem appeared for the first time in Bass' book, the paper of Bass-Heller-Swan~\cite[Theorem 2']{BHS64} is often mentioned in conjunction with it, as an early version which contained several key ideas, and the theorem for left regular rings, appeared there.  See \cite[pXV]{Ba68} for further discussion.

We use the version given by Ranicki~\cite{Ra86}. We start by defining the terms appearing in the decomposition in \cref{ss:k0}. \cref{ss:BHS} contains the decomposition of $\Wh(C_\infty \times C_m)$ and the induced decompositions of $\mathcal{J}_n(C_\infty \times C_m)$, $\mathcal{I}_n(C_\infty \times C_m)$ and $\wh{H}^{n+1}(C_2;\Wh(C_\infty \times C_m))$. These are combined with the results of \cref{p:algebra} in \cref{ss:ijh}, to prove \cref{theorem:main12,theorem:main3}, which are the key algebraic ingredients in the proofs of \cref{theorem:main-S^1xL,theorem:hcob-S^1xL,,theorem:s-hcob-S^1xL}.

\subsection{The $\wt{K}_0$ and $NK_1$ groups} \label{ss:k0}

First we consider the $K_0$ groups. For a ring $R$, we define the abelian group $K_0(R)$ and, for a group $G$, the reduced group $\wt{K}_0(\Z G)$. A convenient reference for this material is \cite{We13}. In \cref{p:algebra}, we study the involution on $\wt{K}_0(\Z G)$ in more detail when $G$ is a finite cyclic group.

\begin{definition}
For a ring $R$ let $P(R)$ denote the set of isomorphism classes of finitely generated left projective $R$-modules, which is a monoid under direct sum. Define $K_0(R)$ to be the Grothendieck group of this monoid, 
i.e.\ the abelian group generated by symbols $[P]$, for every $P \in P(R)$, subject to the relations $[P_1 \oplus P_2] = [P_1] + [P_2]$ for $P_1, P_2 \in P(R)$.
\end{definition}

The assignment $R \mapsto K_0(R)$ defines a functor from the category of rings to the category of abelian groups. If $f \colon R \to S$ is a ring homomorphism, then $K_0(f) \colon K_0(R) \to K_0(S)$ is the map induced by extension of scalars $P \mapsto f_\#(P) := S \otimes_R P$ (see \cref{def:tensor-hom}) for $P \in P(R)$.

We say that a ring $R$ is \textit{$\Z$-augmented} if it comes equipped with a ring homomorphism $\varepsilon \colon R \to \Z$. If $i  \colon \Z \to R$ is the unique ring homomorphism, then $\varepsilon \circ i = \id_{\Z}$ and so $i$ is injective and $\varepsilon$ is surjective.
For example, $\Z G$ is $\Z$-augmented where $\varepsilon \colon \Z G \to \Z$ denotes the augmentation map.

\begin{definition}
For a $\Z$-augmented ring $R$, let
\[
\wt{K}_0(R)  :=  K_0(R)/K_0(\Z),
\]
where the map $K_0(\Z) \to K_0(R)$ is induced by the inclusion $i \colon \Z \rightarrow R$.
\end{definition}

\begin{definition}
For a $\Z$-augmented ring $R$ and a finitely generated projective $R$-module $P$, the \textit{rank} of $P$ is defined to be
\[ \rk(P) := \rank_{\Z}(\varepsilon_\#(P)) \]
where $\varepsilon \colon R \to \Z$ denotes the augmentation map and $\rank_{\Z}$ denotes the rank of a free $\Z$-module.
\end{definition}

Let $R$ be a $\Z$-augmented ring.
We have that $K_0(\Z) \cong \Z$ and, for $P \in P(R)$, the rank $\rk(P)$ coincides with the image of $[P]$ under the composition $K_0(R) \to K_0(\Z) \cong \Z$ induced by $\varepsilon$. Since $\varepsilon \circ i = \id_{\Z}$, we have a splitting of abelian groups
$ 
K_0(R) \cong \Z \oplus \wt{K}_0(R)
$
given by $[P] \mapsto (\rank(P),[P])$. 
In particular, $\wt{K}_0(R)$ is the set of equivalence classes of finitely generated projective $R$-modules $P$ where $P \sim Q$ if $P \oplus R^i \cong Q \oplus R^j$ for some $i,j \geq 0$.

For $R$ a ring with involution, we can define a natural involution on $K_0(R)$. If $P \in P(R)$, then $P^* \in P(R)$ (see \cref{def:dual}) since $P \oplus Q \cong R^n$ implies that $P^* \oplus Q^* \cong R^n$. 

\begin{definition} \label{def:k0-inv}
The \textit{standard involution} on $K_0(R)$ is given by $[P] \mapsto -[P^*]$. 
If $R$ is a $\Z$-augmented ring, this map preserves the rank of a projective $R$-module and so induces an involution on $\wt{K}_0(R)$.
\end{definition}

Next we define the Nil groups $NK_1(R)$ and recall some of their properties.

\begin{definition}
For a ring $R$, let
\[
NK_1(R)  :=  \coker(K_1(R) \rightarrow K_1(R[t])),
\]
where the map $K_1(R) \rightarrow K_1(R[t])$ is induced by the inclusion $R \rightarrow R[t]$. 
\end{definition}

Note that $NK_1$ is a functor from the category of rings to the category of abelian groups.
If $R$ is a ring with involution, then $R[t]$ has an involution such that $t \mapsto t^{-1}$ and which extends the involution on $R$.
Thus $K_1(R[t])$ has a natural involution (see \cref{ss:wh-group}), and this involution preserves the image of $K_1(R) \rightarrow K_1(R[t])$, and hence induces an involution on $NK_1(R)$, which we denote by $x \mapsto \overline{x}$.

\begin{definition} \label{def:nk1-inv}
We equip $NK_1(R)^2 = NK_1(R) \oplus NK_1(R)$ with the involution $(x,y) \mapsto (\overline{y},\overline{x})$.
\end{definition}

We need the following result of Farrell \cite{Fa77} (note that $\mathrm{Nil}$, which is also written as $\mathrm{Nil}_0$, coincides with $NK_1$ by \cite[Chapter~XII;~7.4(a)]{Ba68}). 

\begin{theorem}[Farrell] \label{theorem:nk1-nonfg}
For any ring $R$, if $NK_1(R) \ne 0$, then $NK_1(R)$ is not finitely generated.
\end{theorem}

We conclude this section with the following result.

\begin{theorem} \label{theorem:nk1}
Let $m \geq 1$. Then $NK_1(\Z C_m) = 0$ if and only if $m$ is square-free. 
\end{theorem}

\begin{proof}
If $m$ is square-free, then $NK_1(\Z C_m) = 0$ by \cite[Theorem 10.8 (d)]{BM67}. If $n \mid m$ and $NK_1(\Z C_n) \ne 0$, then $NK_1(\Z C_m) \ne 0$ (by, for example, \cite[Theorem 3.6]{Ma75}). It remains to show that $NK_1(\Z C_{p^2}) \ne 0$ for all primes $p$. This was achieved in the case where $p$ is odd in \cite[Theorem B]{Ma75} and in the case $p=2$ in \cite[Theorem 1.4]{We09}.	
\end{proof}

\subsection{Bass-Heller-Swan decomposition of $\Wh(C_\infty \times G)$} \label{ss:BHS}

By \cite[p.329, p.357]{Ra86} we have the following theorem. 

\begin{theorem}[Bass-Heller-Swan decomposition] \label{theorem:BHS}
Let $G$ be a group. Then there is an isomorphism of $\Z C_2$-modules, which is natural in $G$:
\[ 
\Wh(C_\infty  \times G) \cong \Wh(G) \oplus \wt{K}_0(\Z G) \oplus NK_1(\Z G)^2
\]
where the $\Z C_2$-module structure of each component is determined by the involution defined in \cref{ss:wh-group}, \cref{def:k0-inv} and \cref{def:nk1-inv} respectively.
\end{theorem}

\begin{proposition}\label{prop:I(ZxG)}
Let $G$ be a group and let $n \in \Z$. The decomposition of \cref{theorem:BHS} restricts to the following isomorphisms:
\[
\begin{aligned} 
\mathcal{J}_n(C_\infty \times G) &\cong \mathcal{J}_n(G) \oplus \{ x \in \wt{K}_0(\Z G) \mid x=-(-1)^n\ol{x} \} \oplus NK_1(\Z G) \\
\mathcal{I}_n(C_\infty \times G) &\cong \mathcal{I}_n(G) \oplus \{ x - (-1)^n\ol{x} \mid x \in \wt{K}_0(\Z G)\} \oplus NK_1(\Z G)
\end{aligned}
\]
where $NK_1(\Z G)$ is embedded into $NK_1(\Z G)^2$ by the map $x \mapsto (x,-(-1)^n\ol{x})$.
\end{proposition} 

\begin{proof}
Since \cref{theorem:BHS} gives a decomposition of $\Z C_2$-modules,  $\mathcal{I}_n(C_\infty \times G)$ and $\mathcal{J}_n(C_\infty \times G)$ are decomposed into the corresponding subgroups of the components on the right-hand side. If $(x,y) \in NK_1(\Z G)^2$, then $\ol{(x,y)} = (\ol{y},\ol{x})$, so $(x,y) = - (-1)^n \ol{(x,y)}$ if and only if $y=-(-1)^n\ol{x}$, so 
\[
\{ (x,y) \in NK_1(\Z G)^2 \mid (x,y) = - (-1)^n\ol{(x,y)}\} = \{ (x, - (-1)^n\ol{x}) \mid x \in NK_1(\Z G)\}
\] 
which is isomorphic to $NK_1(\Z G)$.
Similarly, we have 
\[ (x,y) - (-1)^n \ol{(x,y)} = (x-(-1)^n \ol{y}, -(-1)^n \ol{x-(-1)^n \ol{y}}).\] 
Therefore we have that 
\[
\{ (x,y) - (-1)^n\ol{(x,y)} \mid (x,y) \in NK_1(\Z G)^2\} = \{ (x, - (-1)^n\ol{x}) \mid x \in NK_1(\Z G)\}
\] 
which is isomorphic to $NK_1(\Z G)$.
\end{proof}

We immediately get the following corollary; see also \cite[p.358]{Ra86}.

\begin{corollary} \label{cor:tate-decomp}
Let $G$ be a group. The decomposition of \cref{theorem:BHS} induces an isomorphism
\[
\wh{H}^{n+1}(C_2;\Wh(C_\infty \times G)) \cong \wh{H}^{n+1}(C_2;\Wh(G)) \oplus \wh{H}^{n+1}(C_2;\wt{K}_0(\Z G)).
\]
\end{corollary}

\subsection{The groups $\mathcal{J}_n(C_\infty \times C_m)$, $\mathcal{I}_n(C_\infty \times C_m)$ and $\wh{H}^{n+1}(C_2;\Wh(C_\infty \times C_m))$ for $n$ even} \label{ss:ijh}

In this section we prove, using results from \cref{p:algebra}, our main results about the groups $\mathcal{J}_n(C_\infty \times C_m)$, $\mathcal{I}_n(C_\infty \times C_m)$, and $\wh{H}^{n+1}(C_2;\Wh(C_\infty \times C_m))$ for $n$ even.

By \cref{prop:I(ZxG)} and \cref{theorem:nk1-nonfg}, for any group $G$ and any $n$, if $NK_1(\Z G) \ne 0$, then $\mathcal{J}_n(C_\infty \times G)$ and $\mathcal{I}_n(C_\infty \times G)$ are not finitely generated. If $G$ is finite and $n$ is even, then we also have the following. 

\begin{lemma} \label{lem:I-finiteness}
Suppose that $n$ is even and $G$ is a finite group. Then the following are equivalent:
\begin{clist}{(i)}
\item $|\mathcal{J}_n(C_\infty \times G)| < \infty$
\item $|\mathcal{I}_n(C_\infty \times G)| < \infty$
\item $NK_1(\Z G) = 0$.	
\end{clist}
\end{lemma}

\begin{proof}
Let $SK_1(\Z G) = \ker(K_1(\Z G) \to K_1(\Q G))$ where the map is induced by inclusion $\Z G \subseteq \Q G$. It was shown by Wall \cite[Proposition 6.5]{Wa74} (see also \cite[Theorem 7.4]{Ol88}) that $SK_1(\Z G)$ is isomorphic to the torsion subgroup of $\Wh(G)$. Let $\Wh'(G) = \Wh(G)/SK_1(\Z G)$ be the free part. The standard involution on $\Wh(G)$ induces an involution on $\Wh'(G)$, and Wall  showed  this induced involution is the identity (see \cite[Corollary 7.5]{Ol88}). 

It follows that if $x = -\overline{x} \in \Wh(G)$, then $x$ maps to $0 \in \Wh'(G)$, so $x \in SK_1(\Z G)$. Hence $\mathcal{I}_n(G) \leq \mathcal{J}_n(G) \leq SK_1(\Z G)$. 
If $G$ is finite, then $SK_1(\Z G)$ is finite~\cite{Wa74} and $\wt{K}_0(\Z G)$ is finite~\cite{Sw60} (see also \cref{prop:proj=>LF} (\ref{item:prop-proj-LF-ii})).
Therefore the lemma follows from \cref{prop:I(ZxG)} and \cref{theorem:nk1-nonfg}. 
\end{proof}

\begin{proposition}\label{prop:j0cm}
Suppose that $n$ is even. Then for every $m \geq 2$ we have $\mathcal{J}_n(C_m) = 0$, and hence $\mathcal{I}_n(C_m) = \wh{H}^{n+1}(C_2;\Wh(C_m)) = 0$. 
\end{proposition}

\begin{proof}
By \cref{prop:cyclic-wh}, $\Wh(C_m)$ is a finitely generated free abelian group. By \cite[Proposition 4.2]{Ba74}, the involution acts trivially on $\Wh(G)$ for all $G$ finite abelian. Hence $\mathcal{J}_n(C_m) = \{ x \in \Wh(C_m) \mid x=-x \} = 0$.
\end{proof}

The next two theorems are the main results of \cref{s:inv-wh}. They are established as a consequence of Theorems \ref{theorem:main12-red} and \ref{theorem:main3-red} respectively (see \cref{s:proofs-algebra}).

\begin{theorem} \label{theorem:main12}
Suppose that $n$ is even and $m \ge 2$ is an integer. Then
\begin{clist}{(i)}
\item\label{item:theorem:main12-i}
$|\mathcal{J}_n(C_\infty \times C_m)|=1$ if and only if $m \in \{2, 3, 5, 6, 7, 10, 11, 13, 14, 17, 19\}$;
\item\label{item:theorem:main12-ii} 
$|\mathcal{I}_n(C_\infty \times C_m)|=1$ if and only if $m \in \{2, 3, 5, 6, 7, 10, 11, 13, 14, 15, 17, 19, 29\}$; 
\item\label{item:theorem:main12-iii} 
$|\mathcal{J}_n(C_\infty \times C_m)|=\infty$ if and only if $|\mathcal{I}_n(C_\infty \times C_m)|=\infty$ if and only if $m$ is not square-free;
\item\label{item:theorem:main12-iv} 
$|\mathcal{I}_n(C_\infty \times C_m)| \to \infty$ super-exponentially in $m$, and hence we also have that \\
$|\mathcal{J}_n(C_\infty \times C_m)| \to \infty$ super-exponentially in $m$. 
\end{clist}
\end{theorem}

Part \eqref{item:theorem:main12-i} says that the $m \ge 2$ for which $|\mathcal{J}_n(C_\infty \times C_m)|=1$ are as follows:
\[ m =  
\begin{cases}
\text{$p$}, & \text{where $p \le 19$ is prime} \\
\text{$2p$}, & 	\text{where $p \le 7$ is an odd prime}.
\end{cases}
\]
Part \eqref{item:theorem:main12-ii} says that $|\mathcal{I}_n(C_\infty \times C_m)|=1$ if and only if $|\mathcal{J}_n(C_\infty \times C_m)|=1$ or $m \in \{15, 29\}$.

\begin{proof}
First note that \eqref{item:theorem:main12-iii} immediately follows from \cref{theorem:nk1} and \cref{lem:I-finiteness}. For the proofs of \eqref{item:theorem:main12-i}, \eqref{item:theorem:main12-ii} and \eqref{item:theorem:main12-iv}, we can therefore restrict to the case where $m$ is square-free, when $NK_1(\Z C_m) = 0$. By \cref{prop:I(ZxG),prop:j0cm} we then have:
\[ 
\begin{aligned}
\mathcal{J}_n(C_\infty \times C_m) &\cong \{ x \in \wt{K}_0(\Z C_m) \mid \overline{x}=-x\} \\
\mathcal{I}_n(C_\infty \times C_m) &\cong \{ x - \overline{x} \mid x \in \wt{K}_0(\Z C_m)\}. 
 \end{aligned}
\] 
The results now follow from  \cref{theorem:main12-red}.
\end{proof}

\begin{theorem} \label{theorem:main3}
Suppose that $n$ is even and $m \ge 2$ is an integer.
\begin{clist}{(i)}
\item\label{item:theorem:main3-i} $|\wh{H}^{n+1}(C_2;\Wh(C_\infty \times C_m))|=1$ for infinitely many $m$. 
\item\label{item:theorem:main3-ii} $|\wh{H}^{n+1}(C_2;\Wh(C_\infty \times C_m))|<\infty$ for every $m$.
\item\label{item:theorem:main3-iii} $\displaystyle \sup_{l \le m} |\wh{H}^{n+1}(C_2;\Wh(C_\infty \times C_l))| \to \infty$ exponentially in $m$. 
\end{clist}
\end{theorem}

Note that \eqref{item:theorem:main3-i} and \eqref{item:theorem:main3-iii} imply that 
\[ 
\liminf_{m \to \infty} |\wh{H}^{n+1}(C_2;\Wh(C_\infty \times C_m))| = 1, \quad \limsup_{m \to \infty} |\wh{H}^{n+1}(C_2;\Wh(C_\infty \times C_m))| = \infty.
\] 

\begin{proof}
It follows from \cref{cor:tate-decomp} and \cref{prop:j0cm} that
\[ 
\wh{H}^{n+1}(C_2;\Wh(C_\infty \times C_m)) \cong \wh{H}^1(C_2;\wt{K}_0(\Z C_m)) \cong \frac{\{x \in \wt{K}_0(\Z C_m) \mid x = -\ol{x}\}}{\{x-\ol{x} \mid x \in \wt{K}_0(\Z C_m)\}}.
\] 
Hence parts (\ref{item:theorem:main3-i}) and (\ref{item:theorem:main3-iii}) follows from \cref{theorem:main3-red}, which is proved in \cref{s:proofs-algebra}. Part (\ref{item:theorem:main3-ii}) follows from the fact that $\wt K_0(\Z C_m)$ is finite \cite{Sw60} (see also \cref{prop:proj=>LF} (\ref{item:prop-proj-LF-ii})).
\end{proof}

\section{Homotopy automorphisms of \texorpdfstring{$S^1 \times L$}{the product of a lens space with a circle}}
\label{section:torsion-of-self-equiv-lens-space-x-S1}

If $k, m \geq 2$ and $q_1, \ldots, q_k$ are integers such that $\gcd(m,q_j)=1$ for all $j$, then the lens space $L_{2k-1}(m;q_1,\ldots,q_k)$ is defined to be  
the quotient of $S^{2k-1} \subseteq \C^k$ by the free action of $C_m$ generated by
\[ 
(z_1,\ldots,z_k) \mapsto (e^{2\pi iq_1/m}z_1, \ldots, e^{2\pi iq_k/m}z_k) 
\]
where $z_1, \ldots, z_k \in \C$. This is a $(2k-1)$-dimensional manifold with fundamental group $C_m$.

From now on, let $k, m \geq 2$ be fixed integers, and $L$ be a $(2k-1)$-dimensional lens space with $\pi_1(L) \cong C_m$, i.e.\ we take $L = L_{2k-1}(m;q_1,\ldots,q_k)$ for some $q_1, \ldots, q_k$ such that $\gcd(m,q_j)=1$.

We are interested in the $2k$-dimensional product manifold $S^1 \times L$ and its homotopy automorphisms. Let $\Map(L)$ denote the space of continuous maps $L \rightarrow L$ (with basepoint $\id_L$) and note that there is a map $\pi_1 (\Map(L)) \to \hAut(S^1 \times L)$ given by sending $h \colon S^1 \rightarrow \Map(L)$ to $r_h \colon S^1 \times L \rightarrow S^1 \times L$, $r_h(x,y) = (x,h(x)(y))$. The homotopy automorphisms are now given as follows.

\begin{lemma}[{\cite[Corollary 6.2]{Kh17}}] \label{lem:haut-decomp}
We have 
\[
\hAut(S^1 \times L) \cong \pi_1(\Map(L)) \rtimes (\hAut(S^1) \times \hAut(L)).
\] 
In particular, $\hAut(S^1 \times L)$ is generated by maps of the form $f \times g$ and $r_h$, where $f \in \hAut(S^1)$, $g \in \hAut(L)$ and $h \in \pi_1(\Map(L))$.
\end{lemma}

We can deduce that every homotopy automorphism of $S^1 \times L$ is simple.

\begin{theorem}\label{theorem:wh-torsion-vanishes-of-he-s}
If $f \in \hAut(S^1 \times L)$, then $\tau(f)=0$.
\end{theorem}

\begin{proof}
By \cref{prop:WT-composition}, it suffices to show that the generators of the form $f \times g$ and $r_h$ from \cref{lem:haut-decomp} have vanishing Whitehead torsion.
Since $S^1$ and $L$ are odd dimensional, it follows from \cref{cor:odd-product} that homotopy equivalences of the form $f \times g$ have vanishing Whitehead torsion. If $h \colon S^1 \rightarrow \Map(L)$ is any pointed map, then $r_h$ is a fibre homotopy equivalence covering the identity of $S^1$, and acting as the identity on the copy of $L$ that is the fibre over the basepoint:
\[
\xymatrix{
L \ar[r] \ar[d]_-{\id} & S^1 \times L \ar[r] \ar[d]_-{r_h} & S^1 \ar[d]^-{\id} \\
L \ar[r] & S^1 \times L \ar[r] & S^1. 
}
\]
Hence $\tau(r_h)=0$ by \cref{prop:WT-fiber}, and because $\tau(\Id)=0$.
\end{proof}

Let $G = C_{\infty} \times C_m \cong \pi_1(S^1 \times L)$. Our next goal is to determine which automorphisms of $G$ can be realised by homotopy automorphisms of $S^1 \times L$, i.e.\ to describe the image of the map $\pi_1 \colon \hAut(S^1 \times L) \rightarrow \Aut(G)$. The automorphisms of $G = C_{\infty} \times C_m$ can be expressed as matrices of the form $\left(\begin{smallmatrix} a & b \\ 0 & c \end{smallmatrix}\right)$, where $a \in \{ \pm 1\}$, $b \in C_m$ and $c \in C_m^{\times} \cong \Aut(C_m)$. In particular, we have $\Aut(G) \cong C_m \rtimes (C_2 \times C_m^\times)$.

Recall from \cref{lem:haut-decomp} that the generators of $\hAut(S^1 \times L)$ are of the form $f \times g$ or $r_h$. The induced automorphisms of $G$ are given by $\pi_1(f \times g) = \left(\begin{smallmatrix} \pi_1(f) & 0 \\ 0 & \pi_1(g) \end{smallmatrix}\right)$ and $\pi_1(r_h) = \left(\begin{smallmatrix} 1 & [\ev \circ h] \\ 0 & 1 \end{smallmatrix}\right)$, where $\ev \colon \Map(L) \rightarrow L$ is given by evaluation at the basepoint $y_0 \in L$, so that $(\ev \circ h)(x) = h(x)(y_0)$.

Both elements of $\{ \pm 1\} \cong \Aut(C_{\infty})$ can be realised by homotopy automorphisms (or even diffeomorphisms) of $S^1$. The subgroup of realisable automorphisms of $C_m$ depends on the dimension $(2k-1)$ of the lens space $L$ as follows.

\begin{lemma}[{\cite[Statement~29.5]{Co73}}] \label{lem:haut-L}
Let $c \in C_m^{\times}$. There is a homotopy automorphism $g \colon L \rightarrow L$ such that $\pi_1(g) = c$ if and only if $c^k \equiv \pm 1 \mod m$.
\end{lemma}

Finally, we have the following. 

\begin{lemma} \label{lem:diff-L}
There is a map $h \colon S^1 \rightarrow \Diff(L) \subset \Map(L)$ such that $[\ev \circ h] = 1 \in C_m \cong \pi_1(L)$. Hence the resulting diffeomorphism $r_h \in \Diff(S^1 \times L)$ has $\pi_1(r_h) = \left(\begin{smallmatrix} 1 & 1 \\ 0 & 1 \end{smallmatrix}\right)$.
\end{lemma}

\begin{proof}
Let $S^{2k-1} = \{(z_1, \ldots, z_k) \mid \sum |z_j|^2 = 1 \} \subseteq \C^k$. Recall that $L = L_{2k-1}(m;q_1, \ldots, q_k)$ is the quotient of $S^{2k-1}$ by the $C_m$-action generated by $(z_1, \ldots, z_k) \mapsto (\zeta^{q_1}z_1, \ldots, \zeta^{q_k}z_k)$, where $\zeta = e^{2\pi i/m}$. Define $h \colon S^1 \to \Diff(L)$ by $e^{2\pi i t} \mapsto ( [z_1, \ldots, z_k] \mapsto [\zeta^{q_1t}z_1, \ldots, \zeta^{q_kt}z_k] )$. Then the loop $\ev \circ h \colon S^1 \rightarrow L$ is the standard generator of $\pi_1(L) \cong C_m$.
\end{proof}

\begin{definition} \label{def:a-group}
For a positive integer $k$ and $G = C_\infty \times C_m$, let $A_{2k}(m) \leq \Aut(G)$ denote the subgroup of matrices $\left(\begin{smallmatrix} a & b \\ 0 & c \end{smallmatrix}\right)$ such that $c^k \equiv \pm 1$ mod $m$.
\end{definition}

\begin{theorem} \label{theorem:aut-image}
Let $L$ be a $(2k-1)$-dimensional lens space with $\pi_1(L) \cong C_m$. Then
\[\im(\pi_1 \colon \hAut(S^1 \times L) \rightarrow \Aut(G)) = A_{2k}(m).\] 
\end{theorem}

\begin{proof}
By \cref{lem:haut-decomp} the subgroup of realisable automorphisms is generated by $\pi_1(f \times g)$ and $\pi_1(r_h)$ for all $f \in \hAut(S^1)$, $g \in \hAut(L)$ and $h \colon S^1 \rightarrow \Map(L)$. By our earlier observations and \cref{lem:haut-L}, the automorphisms realised as $\pi_1(f \times g)$ are precisely those of the form $\left(\begin{smallmatrix} a & 0 \\ 0 & c \end{smallmatrix}\right)$ with $a \in \{ \pm 1\}$ and $c^k \equiv \pm 1$ mod $m$, and $\pi_1(r_h)$ is of the form $\left(\begin{smallmatrix} 1 & b \\ 0 & 1 \end{smallmatrix}\right)$, showing that $\im(\pi_1) \leq A_{2k}(m)$. On the other hand, by \cref{lem:diff-L} the matrix $\left(\begin{smallmatrix} 1 & 1 \\ 0 & 1 \end{smallmatrix}\right)$ is realisable, and together with the matrices $\left(\begin{smallmatrix} a & 0 \\ 0 & c \end{smallmatrix}\right)$ with $a \in \{ \pm 1\}$ and $c^k \equiv \pm 1$ mod $m$ they generate $A_{2k}(m)$, therefore $\im(\pi_1) \geq A_{2k}(m)$.
\end{proof}

\section{The proof of \cref{main-theorem,theorem:main-S^1xL,theorem:hcob-S^1xL,,theorem:s-hcob-S^1xL}} \label{s:b-proof}

We can now combine the results of the previous sections to prove \cref{main-theorem,theorem:main-S^1xL,theorem:hcob-S^1xL,,theorem:s-hcob-S^1xL}. Let~$n = 2k \geq 4$ be an even integer and fix a category $\CAT$ satisfying \cref{assumptions}. 

We begin by establishing the following key fact concerning the map $\psi$ from the exact sequence \eqref{eq:SRR-sequence} for the groups $C_\infty \times C_m$.

\begin{proposition}\label{prop:psi-surjective}
Let $m \geq 2$. For any even integer $n = 2k$, the map
\[\psi \colon L_{n+1}^{\mathrm{h}}(\Z[C_{\infty} \times C_m]) \to \wh{H}^{n+1}(C_2;\Wh(C_{\infty} \times C_m))\]
is surjective.
\end{proposition}

\begin{proof}
Let $G_m  :=   C_{\infty} \times C_m$. We verify that the forgetful map $F \colon L^{\mathrm{s}}_{2k}(\Z G_m) \to L^{\mathrm{h}}_{2k}(\Z G_m)$ is injective. The conclusion then follows from the exact sequence~\eqref{eq:SRR-sequence}.  
First we apply Shaneson splitting~\cite{Shaneson-GxZ} to the domain and the codomain, to obtain a commutative diagram whose rows are split short exact sequences and whose vertical maps are the forgetful maps~\cite{Ra86}.
\[\begin{tikzcd}
0 \arrow[r] &  L^{\mathrm{s}}_{2k}(\Z C_m) \ar[r] \ar[d] & L^{\mathrm{s}}_{2k}(\Z G_m) \ar[r]  \ar[d,"F"] &  L^{\mathrm{h}}_{2k-1}(\Z C_m) \ar[r] \ar[d] & 0 \\
0 \arrow[r] &  L^{\mathrm{h}}_{2k}(\Z C_m) \ar[r]  & L^{\mathrm{h}}_{2k}(\Z G_m) \ar[r]  &  L^{\mathrm{p}}_{2k-1}(\Z C_m) \ar[r] & 0.
\end{tikzcd}\]
By the five lemma it suffices to show that the left and right vertical maps are injective.   
In~\cite{Bak-computation-L-groups}, Bak gave computations of $L_i^x(\Z G)$ for $x \in \{\mathrm{s},\mathrm{h},\mathrm{p}\}$ and $G$ a finite group whose 2-hyperelementary subgroups are abelian, which certainly holds for finite cyclic groups.
See also \cite{Bak-odd-L-groups-of-odd-order-groups, Bak-even-L-groups-of-odd-order-groups}.
By \cite[Theorem~8]{Bak-computation-L-groups}, the forgetful map $L^{\mathrm{h}}_{2k-1}(\Z C_m) \to L^{\mathrm{p}}_{2k-1}(\Z C_m)$ is injective, so the right vertical map is injective. 

To prove that the left vertical map is injective we also use \cite{Bak-computation-L-groups}. 
For a finite group $G$ whose 2-Sylow subgroup $G_2$ is normal and abelian, Bak defined $r_2:= \rk H^1(C_2;\Wh(G))$.  On~\cite[p.~386]{Bak-computation-L-groups}, he noted that if $G_2$ is cyclic, as in our case $G = C_m$, then $r_2=0$.  Therefore by the sequence~\eqref{eq:SRR-sequence}, as shown on~\cite[p.~390]{Bak-computation-L-groups}, it follows that $L_{2k}^{\mathrm{s}}(\Z C_m) \to L_{2k}^{\mathrm{h}}(\Z C_m)$ is injective. 
Hence $F \colon L_{2k}^{\mathrm{s}}(\Z G_m) \to L_{2k}^{\mathrm{h}}(\Z G_m)$ is injective as desired. 
\end{proof}

Recall the definition of the sets $\M^{\mathrm{h}}_{\mathrm{s}}(M)$, $\M^{\hCob}_{\mathrm{s}}(M)$ and $\M^{\mathrm{h}}_{\mathrm{s},\hCob}(M)$ from \cref{def:man-sets}. Also recall the definition of $A_{2k}(m)$ (\cref{def:a-group}): 
$A_{2k}(m) \leq \Aut(C_{\infty} \times C_m)$ denotes the subgroup of matrices $\left(\begin{smallmatrix} a & b \\ 0 & c \end{smallmatrix}\right)$ such that $c^k \equiv \pm 1$ mod $m$. 
Note that $A_{2k}(m)$ acts on $\Wh(C_{\infty} \times C_m)$ and on the subgroups $\mathcal{I}_n(C_{\infty} \times C_m)$ and $\mathcal{J}_n(C_{\infty} \times C_m)$. 

\begin{proposition} \label{prop:man-set-bij}
Let $m \geq 2$ and let $L$ be a lens space of dimension $n-1=2k-1$ with $\pi_1(L) \cong C_m$. Let $G := C_{\infty} \times C_m \cong \pi_1(S^1 \times L)$. Then there are bijections of pointed sets $\M^{\mathrm{h}}_{\mathrm{s}}(S^1 \times L) \cong \mathcal{J}_n(G) / A_{2k}(m)$, $\M^{\hCob}_{\mathrm{s}}(S^1 \times L) \cong \mathcal{I}_n(G) / A_{2k}(m)$ and $\M^{\mathrm{h}}_{\mathrm{s},\hCob}(S^1 \times L) \cong \wh{H}^{n+1}(C_2;\Wh(G)) / A_{2k}(m)$.
\end{proposition}

\begin{proof}
Firstly, since $C_\infty \times C_m$ is polycyclic, it is a good group~\cite{FQ,Freedman-book-goodgroups}, so \cref{assumptions} is satisfied.  Moreover, since the map $\psi$ is surjective for $G = C_{\infty} \times C_m$, $w\equiv 1$ and $n$ even by \cref{prop:psi-surjective}, the vertical maps in the diagram of \cref{theorem:man-class}~\eqref{item:thm-man-class-c} are bijections if $M = S^1 \times L$. Therefore $\M^{\mathrm{h}}_{\mathrm{s}}(S^1 \times L) \cong \mathcal{J}_n(G) / \hAut(S^1 \times L)$, $\M^{\hCob}_{\mathrm{s}}(S^1 \times L) \cong q(\mathcal{I}_n(G))$, and $\M^{\mathrm{h}}_{\mathrm{s},\hCob}(S^1 \times L) \cong \wh{H}^{n+1}(C_2;\Wh(G)) / \hAut(S^1 \times L)$.

By \cref{theorem:wh-torsion-vanishes-of-he-s}, every homotopy automorphism of $S^1 \times L$ is simple, therefore the action of $\hAut(S^1 \times L)$ on $\Wh(G)$ factors through the action of $\Aut(G)$ (see \cref{def:haut-act} and \cref{rem:haut-act-spec2}). By \cref{theorem:aut-image}, the image of $\pi_1\colon\hAut(S^1 \times L) \rightarrow \Aut(G)$ is $A_{2k}(m)$. In particular, the orbits of the action of $\hAut(S^1 \times L)$ and $A_{2k}(m)$ on $\mathcal{J}_n(G)$, $\mathcal{I}_n(G)$ and $\wh{H}^{n+1}(C_2;\Wh(G))$ coincide.
This implies that $\mathcal{J}_n(G) / \hAut(S^1 \times L) = \mathcal{J}_n(G) / A_{2k}(m)$, $q(\mathcal{I}_n(G)) = \mathcal{I}_n(G) / A_{2k}(m)$ and $\wh{H}^{n+1}(C_2;\Wh(G)) / \hAut(S^1 \times L) = \wh{H}^{n+1}(C_2;\Wh(G)) / A_{2k}(m)$, as required.
\end{proof}

In particular \cref{prop:man-set-bij} implies that 
$|\M^{\mathrm{h}}_{\mathrm{s}}(S^1 \times L)|$, $|\M^{\hCob}_{\mathrm{s}}(S^1 \times L)|$ and $|\M^{\mathrm{h}}_{\mathrm{s},\hCob}(S^1 \times L)|$ are independent of the choice of the $q_j$ and of $\CAT$, and only depend on $n$ and $m$. This proves part \eqref{item:thm-1-i} of \cref{theorem:main-S^1xL,theorem:hcob-S^1xL,,theorem:s-hcob-S^1xL}.

From now on, we write $M^n_m$ for $S^1 \times L$, where $L$ is any $(n-1)$-dimensional lens space with $\pi_1(L) \cong C_m$.

\begin{lemma} \label{lem:man-set-equiv}
Let $n = 2k \geq 4$ be an even integer, $m \geq 2$ and $G = C_{\infty} \times C_m$. Then the following hold.
\benum
\item\label{item:lem-man-set-equiv-a} $|\M^{\mathrm{h}}_{\mathrm{s}}(M^n_m)| = 1$ if and only if $\mathcal{J}_n(G) = 0$.
\item\label{item:lem-man-set-equiv-b} $|\M^{\mathrm{h}}_{\mathrm{s}}(M^n_m)| = \infty$ if and only if $|\mathcal{J}_n(G)| = \infty$.
\item\label{item:lem-man-set-equiv-c} If $\M^{\mathrm{h}}_{\mathrm{s}}(M^n_m)$ is finite, then $\frac{|\mathcal{J}_n(G)|}{2m^2} < |\M^{\mathrm{h}}_{\mathrm{s}}(M^n_m)| \leq |\mathcal{J}_n(G)|$.
\item\label{item:lem-man-set-equiv-d} $|\M^{\hCob}_{\mathrm{s}}(M^n_m)| = 1$ if and only if $\mathcal{I}_n(G) = 0$.
\item\label{item:lem-man-set-equiv-e} $|\M^{\hCob}_{\mathrm{s}}(M^n_m)| = \infty$ if and only if $|\mathcal{I}_n(G)| = \infty$.
\item\label{item:lem-man-set-equiv-f} If $\M^{\hCob}_{\mathrm{s}}(M^n_m)$ is finite, then $\frac{|\mathcal{I}_n(G)|}{2m^2} < |\M^{\hCob}_{\mathrm{s}}(M^n_m)| \leq |\mathcal{I}_n(G)|$.
\item\label{item:lem-man-set-equiv-g} $|\M^{\mathrm{h}}_{\mathrm{s},\hCob}(M^n_m)| = 1$ if and only if $\wh{H}^{n+1}(C_2;\Wh(G)) = 0$.
\item\label{item:lem-man-set-equiv-h} $|\M^{\mathrm{h}}_{\mathrm{s},\hCob}(M^n_m)| = \infty$ if and only if $|\wh{H}^{n+1}(C_2;\Wh(G))| = \infty$.
\item\label{item:lem-man-set-equiv-i} If $\M^{\mathrm{h}}_{\mathrm{s},\hCob}(M^n_m)$ is finite, then $\frac{|\wh{H}^{n+1}(C_2;\Wh(G))|}{2m^2} < |\M^{\mathrm{h}}_{\mathrm{s},\hCob}(M^n_m)| \leq |\wh{H}^{n+1}(C_2;\Wh(G))|$.
\eenum
\end{lemma}

\begin{proof}
\eqref{item:lem-man-set-equiv-a} By \cref{prop:man-set-bij}, $|\M^{\mathrm{h}}_{\mathrm{s}}(M^n_m)| = |\mathcal{J}_n(G) / A_{2k}(m)|$. The group $A_{2k}(m) \leq \Aut(G)$ acts on $\Wh(G)$, and hence on $\mathcal{J}_n(G)$, by automorphisms. So $0 \in \mathcal{J}_n(G)$ is a fixed point of this action, and $|\mathcal{J}_n(G) / A_{2k}(m)|=1$ if and only if $\mathcal{J}_n(G) = 0$.

\eqref{item:lem-man-set-equiv-b} $\Aut(G)$ is finite, and hence so is its subgroup $A_{2k}(m)$.
 Hence  $|\mathcal{J}_n(G) / A_{2k}(m)| = \infty$ if and only if $|\mathcal{J}_n(G)| = \infty$. 

\eqref{item:lem-man-set-equiv-c} It is easy to see that $\frac{|\mathcal{J}_n(G)|}{|A_{2k}(m)|} \leq |\mathcal{J}_n(G) / A_{2k}(m)| \leq |\mathcal{J}_n(G)|$. We have $|A_{2k}(m)| \leq |\Aut(G)| \leq 2m(m-1) < 2m^2$ since elements of $\Aut(G)$ can be represented as matrices of the form $\left(\begin{smallmatrix} a & b \\ 0 & c \end{smallmatrix}\right)$, where $a \in \{ \pm 1\}$, $b \in C_m$, $c \in C_m^{\times}$. 

The proofs of parts
\eqref{item:lem-man-set-equiv-d},
\eqref{item:lem-man-set-equiv-e}, and
\eqref{item:lem-man-set-equiv-f} (resp.\ parts \eqref{item:lem-man-set-equiv-g},
\eqref{item:lem-man-set-equiv-h}  and
\eqref{item:lem-man-set-equiv-i}) are entirely analogous to those of parts
\eqref{item:lem-man-set-equiv-a},
\eqref{item:lem-man-set-equiv-b}, and
\eqref{item:lem-man-set-equiv-c} respectively, and so are omitted for brevity. 
\end{proof}

We can now prove the following, strengthened form of \cref{main-theorem}.

\begin{theorem} 
Let $n \geq 4$ be even, and fix $\CAT \in \{\Diff, \PL, \TOP\}$ satisfying \cref{assumptions}. Let $m \geq 2$ be an integer that is not square-free (e.g.\ $m=4$), and let $L$ be an $(n-1)$-dimensional lens space with $\pi_1(L) \cong C_m$. Then for $M^n_m := S^1 \times L$, the set $\M^{\mathrm{h}}_{\mathrm{s}}(M^n_m)$ is infinite.

Consequently, there exists an infinite collection of closed, connected, orientable, $\CAT$ $n$-manifolds that are all homotopy equivalent to one another but are pairwise not simple homotopy equivalent.
\end{theorem}

\begin{proof}
By \cref{theorem:nk1} and \cref{lem:I-finiteness}, we have that $|\mathcal{J}_n(C_\infty \times C_m)| = \infty$. 
Hence it follows from \cref{lem:man-set-equiv}~\eqref{item:lem-man-set-equiv-b} that $\M^{\mathrm{h}}_{\mathrm{s}}(M^n_m)$ is infinite.
Note that $M^n_m = S^1 \times L$ is orientable, so the same is true for every $n$-manifold homotopy equivalent to it. Therefore we obtain a suitable infinite collection by choosing a representative from each element of $\M^{\mathrm{h}}_{\mathrm{s}}(M^n_m)$.
\end{proof}

The following completes the proof of the remaining parts of \cref{theorem:main-S^1xL,theorem:hcob-S^1xL,,theorem:s-hcob-S^1xL}, subject to the proofs of \cref{theorem:main12-red,theorem:main3-red} which are postponed until \cref{s:proofs-algebra}.

\begin{theorem} \label{theorem:man-set-summary}
Let $n = 2k \geq 4$ be an even integer and $m \geq 2$. Then the following hold. 
\benum
\item\label{item:lem-man-set-summary-a}
 $|\M^{\mathrm{h}}_{\mathrm{s}}(M^n_m)| = 1$ if and only if $m \in \{2,3,5,6, 7,10, 11,13,14,17,19\}$.
\item\label{item:lem-man-set-summary-b}
 $|\M^{\mathrm{h}}_{\mathrm{s}}(M^n_m)| = \infty$ if and only if $m$ is not square-free.
\item\label{item:lem-man-set-summary-c}
 $|\M^{\mathrm{h}}_{\mathrm{s}}(M^n_m)| \to \infty$ as $m \to \infty$ $($uniformly in $n)$.
\item\label{item:lem-man-set-summary-d}
 $|\M^{\hCob}_{\mathrm{s}}(M^n_m)|=1$ if and only if $m \in \{2, 3, 5, 6, 7, 10, 11, 13, 14, 15, 17, 19, 29\}$.
\item\label{item:lem-man-set-summary-e}
 $|\M^{\hCob}_{\mathrm{s}}(M^n_m)| = \infty$ if and only if $m$ is not square-free.
\item\label{item:lem-man-set-summary-f}
 $|\M^{\hCob}_{\mathrm{s}}(M^n_m)| \to \infty$ as $m \to \infty$ $($uniformly in $n)$.
\item\label{item:lem-man-set-summary-g}
 There are infinitely many $m$ such that $|\M^{\mathrm{h}}_{\mathrm{s},\hCob}(M^n_m)| = 1$ for every $n$.
\item\label{item:lem-man-set-summary-h}
 $|\M^{\mathrm{h}}_{\mathrm{s},\hCob}(M^n_m)|$ is finite for every $n$ and $m$.
\item\label{item:lem-man-set-summary-i}
 $\displaystyle \limsup_{m \to \infty} \Bigl( \inf_n |\M^{\mathrm{h}}_{\mathrm{s},\hCob}(M^n_m)| \Bigr) = \infty$.
\eenum
\end{theorem}

\begin{proof}
\eqref{item:lem-man-set-summary-a}, \eqref{item:lem-man-set-summary-b}, \eqref{item:lem-man-set-summary-d}, \eqref{item:lem-man-set-summary-e}, \eqref{item:lem-man-set-summary-g}, and \eqref{item:lem-man-set-summary-h} follow immediately from \cref{lem:man-set-equiv} and  \cref{theorem:main12,theorem:main3}.

\eqref{item:lem-man-set-summary-c} For every $m$, it follows from \cref{lem:man-set-equiv} \eqref{item:lem-man-set-equiv-c} that $|\M^{\mathrm{h}}_{\mathrm{s}}(M^n_m)| > \frac{|\mathcal{J}_n(C_\infty \times C_m)|}{2m^2}$ for every $n$. By \cref{theorem:main12}~\eqref{item:theorem:main12-iv} $\frac{|\mathcal{J}_n(C_\infty \times C_m)|}{2m^2} \to \infty$ as $m \to \infty$. Item \eqref{item:lem-man-set-summary-f} can be proved similarly.

\eqref{item:lem-man-set-summary-i} For every $m$, it follows from \cref{lem:man-set-equiv} \eqref{item:lem-man-set-equiv-i} that $\inf_n |\M^{\mathrm{h}}_{\mathrm{s},\hCob}(M^n_m)| > \frac{|\wh{H}^{n+1}(C_2;\Wh(C_\infty \times C_m))|}{2m^2}$. By \cref{theorem:main3}~\eqref{item:theorem:main3-iii}, we see that $\frac{1}{2m^2} \sup_{l \le m} |\wh{H}^{n+1}(C_2;\Wh(C_\infty \times C_l))| \to \infty$ as $m \to \infty$. This implies that 
$\sup_{l \le m} \frac{|\wh{H}^{n+1}(C_2;\Wh(C_\infty \times C_l))|}{2l^2} \to \infty$ as $m \to \infty$. 
So the map $m \mapsto \frac{|\wh{H}^{n+1}(C_2;\Wh(C_\infty \times C_m))|}{2m^2}$ is unbounded. Therefore $\limsup_{m \to \infty} \frac{|\wh{H}^{n+1}(C_2;\Wh(C_\infty \times C_m))|}{2m^2} = \infty$.
\end{proof}

\part{The involution on $\wt{K}_0(\Z C_m)$} \label{p:algebra}

The aim of this part is to prove \cref{theorem:main12-red,,theorem:main3-red}, which are key ingredients in the proofs of \cref{theorem:main12,theorem:main3} respectively. In \cref{s:class-groups,s:tate}, we recall the necessary background on class groups, Tate cohomology, and $\Z C_2$-modules. The main technical heart of this part is \cref{s:involution-ZC_m} where we investigate the involution on $\wt K_0(\Z C_m)$ and prove general results which allow it to be computed. In \cref{s:proofs-algebra}, we make use of the results in \cref{s:involution-ZC_m} to prove \cref{theorem:main12-red,theorem:main3-red}. Throughout, we assume that all modules are left modules.

\section{Locally free class groups} \label{s:class-groups}

We now recall the theory of locally free class groups for orders in semisimple $\Q$-algebras. Good references for this material are \cite[Section 1-3]{Sw80} and \cite[Section 49A \& 50E]{CR87}. 

\subsection{Definitions and properties} \label{ss:class-groups-basics}

For a ring $A$, a nonzero $A$-module is \textit{simple} if it contains no $A$-submodules other than itself and $0$, and is \textit{semisimple} if it is isomorphic as an $A$-module to a direct sum of its simple $A$-submodules. 
We say that a ring $A$ is \textit{semisimple} if $A$, viewed as an $A$-module, is semisimple. 

Let $A$ be a finite-dimensional semisimple $\Q$-algebra. An \textit{order} in $A$ is a subring~$\l \subseteq A$ that is finitely generated as an abelian group and which has $\Q \cdot \l = A$. 
For example, let $G$ be a finite group. Then $A = \Q G$ is a finite dimensional $\Q$-algebra which is semisimple by Maschke's theorem in representation theory, and $\l = \Z G$ is an order in $A$. 

From now on, fix an order $\l$ in a finite-dimensional semisimple $\Q$-algebra $A$. For a prime $p$, let $\l_p = \Z_p \otimes_{\Z} \l$ and let $A_p = \Q_p \otimes_{\Q} A$ denote the $p$-adic completions of $\l$ and $A$. 

\begin{definition}
A $\l$-module $M$ is \textit{locally free} if it is finitely generated and $M_p = \Z_p \otimes_{\Z} M$ is a free $\l_p$-module for all primes~$p$.
\end{definition}

The following is \cite[Lemma 2.1]{Sw80}. The converse need not hold, i.e.\ there exist orders $\l$ and finitely generated projective $\l$-modules that are not locally free \cite[p.~156]{Sw80}.

\begin{proposition} \label{prop:LF=>proj}
If $M$ is a locally free $\l$-module, then $M$ is projective.
\end{proposition}

We say that two locally free $\l$-modules $M$ and $N$ are \textit{stably isomorphic}, written $M \cong_{\st} N$, if there exists $r,s \ge 0$ such that $M \oplus \l^r \cong N \oplus \l^s$ are isomorphic as $\l$-modules.

\begin{definition}
Define the \textit{locally free class group} $C(\l)$ to be the set of equivalence classes of locally free $\l$-modules up to stable isomorphism. This is an abelian group under direct sum (since the direct sum of locally free modules is locally free).
\end{definition}

It follows that $C(\l) \le \wt K_0(\l)$ is a subgroup, where $\wt K_0$ is the $0$th reduced algebraic $K$-group as defined in \cref{ss:k0}. It is a consequence of the Jordan-Zassenhaus theorem that $C(\l)$ is finite \cite[Remark~49.11~(ii)]{CR87}.

We now specialise to the case where $\l = \Z G$ for $G$ a finite group. In contrast to the situation for general orders, we have the following proposition \cite[p.~156]{Sw80}. By \cref{prop:LF=>proj}, this implies that a finitely generated $\Z G$-module is projective if and only if it is locally free.

\begin{proposition} \label{prop:proj=>LF}
Let $G$ be a finite group.
\begin{clist}{(i)}
\item\label{item:prop-proj-LF-i}
If $M$ is a finitely generated projective $\Z G$-module, then $M$ is locally free. 
\item\label{item:prop-proj-LF-ii}
There is an isomorphism of abelian groups $\wt K_0(\Z G) \cong C(\Z G)$. In particular, $\wt K_0(\Z G)$ is finite.
\end{clist}
\end{proposition}

Finally we note the following which relates locally free class groups to ideal class groups (see \cite[Section 35]{Re75}). For a field $K/\Q$, we let $\mathcal{O}_K$ denote the ring of integers.

\begin{proposition} \label{prop:LF=ideals}
Let $K/\Q$ be a finite field extension. Then $C(\mathcal{O}_K)$ coincides with the ideal class group of $\mathcal{O}_K$.
\end{proposition}

\subsection{Kernel groups} \label{ss:kernel-groups}

Let $A$ be a finite-dimensional semisimple $\Q$-algebra. An order in $A$ is said to be \textit{maximal} if it is not properly contained in another order in $A$. Since every finite field extension of $\Q$ is separable, $A$ is a separable algebra and so every order in $A$ is contained in a maximal order \cite[Proposition 5.1]{Sw70}.

Let $\l$ be an order in $A$ and let $\Gamma$ be a maximal order in $A$ containing $\l$. The inclusion map~$i \colon \l \hookrightarrow \Gamma$ induces a map $i_* \colon C(\l) \to C(\Gamma)$ given by extension of scalars $[M] \mapsto [\Gamma \otimes_{\l} M]$ which is necessarily surjective by \cite[Theorem 49.25]{CR87}.

\begin{definition}
Define the \textit{kernel group} $D(\l)$ to be the kernel of the map $i_* \colon C(\l) \to C(\Gamma)$. (This is often also referred to as the defect group.)
\end{definition}

The group $C(\Gamma)$ does not depend up to isomorphism on the choice of maximal order in $A$, i.e. if $\Gamma_1$, $\Gamma_2$ are maximal orders in $A$ containing $\l$, then there is an isomorphism $C(\Gamma_1) \cong C(\Gamma_2)$ \cite[Theorem 49.32]{CR87}. Furthermore, the kernel group $D(\l)$ does not depend on the choice of maximal order. This can be seen from the fact that it can be defined without reference to a maximal order: if $M$ is a locally free $\l$-module, then $[M] \in D(\l)$ if and only if there exists a finitely generated $\l$-module $X$ such that $M \oplus X \cong \l^n \oplus X$ for some $n$ \cite[Proposition 49.34]{CR87}.

In particular, we have a well-defined exact sequence of abelian groups:
\[ 0 \to D(\l) \to C(\l) \to C(\Gamma) \to 0 \]
where $\Gamma$ can be taken to be any maximal order in $A$ containing $\l$.

\subsection{The id\`{e}lic approach to locally free class groups}

Let $\l$ be an order in a finite-dimensional semisimple $\Q$-algebra $A$.

\begin{definition}
Define the \textit{id\`{e}le group} 
\[J(A) = \{(\alpha_p) \in \prod_p A_p^\times \mid \alpha_p \in \l_p^\times \textit{ for all but finitely many $p$}  \} \subseteq \prod_p A_p^\times.\] 
\end{definition}

As a subgroup of $\prod_p A_p^\times$, this is independent of the choice of order $\l$ \cite[p.~218]{CR87}. 
Every class in $C(\l)$ is represented by a locally free $\l$-module $M \subseteq A$ \cite[p.~218]{CR87}. For each $p$, there exists $\alpha_p \in A_p$ such that $M_p = \l_p \alpha_p \subseteq A_p$. For all but finitely many $p$, $M_p \cong \l_p$ and so $\alpha_p \in \l_p^\times$. In particular, $\alpha = (\alpha_p) \in J(A)$. Conversely, given an id\`{e}le $\alpha \in J(A)$, we have that $\l \alpha = A \cap \bigcap_p \l_p \alpha_p \subseteq A$ is a locally free $\l$-ideal. 
Let $\alpha, \beta \in J(A)$. Then $\l \alpha \cong \l \beta$ as $\l$-modules if and only if $\beta \in U(\l) \cdot \alpha \cdot A^\times$ where $A^\times \subseteq J(A)$ by sending $a \in A^\times$ to $\alpha_p = 1 \otimes_{\Q} a$ for all $p$, and $U(\l) = \{ (\alpha_p) \in \prod_p A_p^\times\mid \alpha_p \in \l_p^\times \text{ for all $p$}\} \subseteq J(A)$ \cite[49.6]{CR87}.
Furthermore, we have that $\l \alpha \oplus \l \beta \cong \l \oplus \l \alpha \beta$ \cite[49.8]{CR87}. This leads to the following.

\begin{proposition} \label{prop:idele-surj}
There is a surjective group homomorphism 
\[ [\l \,\cdot\,] \colon J(A) \twoheadrightarrow C(\l), \quad \alpha \mapsto [\l \alpha].\]
\end{proposition}

\begin{remark} \label{remark:CviaJ}
This leads to the formula 
\[ C(\l) \cong \frac{J(A)}{J_0(A) \cdot A^\times \cdot U(\l)}\]
for the locally free class group, where $J_0(A) = \{x \in J(A) \mid \nr(x) = 1\}$ and $\nr \colon J(A) \to J(Z(A))$ is induced by the reduced norm. This is due to Fr\"{o}hlich \cite{Fr75} (see also \cite[Theorem 49.22]{CR87}). 
\end{remark}

\subsection{Involutions on locally free class groups} \label{ss:class-groups-involution}

Let $\l$ be an order in a finite-dimensional semisimple $\Q$-algebra $A$. Suppose further that $A$ is a ring equipped with an involution $\ol{\cdot} \colon A \to A$ which preserves $\l$, i.e.\ the map $\ol{\cdot}$ is an involution on $A$ as an abelian group which satisfies $\ol{xy} = \ol{y} \cdot \ol{x}$ for all $x,y \in A$ and $\ol{x} \in \l$ for all $x \in \l$.
For example, if $G$ is a finite group and $w \colon G \to \{\pm 1\}$ is a homomorphism, then $A=\Q G$ has an involution given by $\sum_{i=1}^k n_i g_i \mapsto \sum_{i=1}^k w(g_i) n_i g_i^{-1}$ for $n_i \in \Z$ and $g_i \in G$ which preserves $\l = \Z G$ (see \cref{ss:wh-group}). Given a (left) $\l$-module $M$, we now have the notion of a dual (left) $\l$-module $M^*$ (see \cref{def:dual}).

We now note the following properties of the dual of locally free modules. An $R$-module $M$ is \textit{reflexive} if the evaluation map $\text{\normalfont ev} \colon M \to M^{**}$, $m \mapsto (f \mapsto f(m))$ is bijective.

\begin{lemma} \label{lemma:dual-LF}
If $M$ is a locally free $\l$-module, then $M$ is reflexive and $M^*$ is locally free.
\end{lemma}

\begin{proof}
By \cref{prop:LF=>proj}, locally free modules are projective. It is well known that finitely generated projective modules are reflexive, hence $M$ is reflexive.
The fact that $M^*$ is locally free follows from that fact that $\Z_p \otimes_{\Z} M^* \cong (\Z_p \otimes_{\Z} M)^* \cong \l_p^* \cong \l_p$ for all primes $p$.
\end{proof}

We can use this to define an involution on the locally free class group.

\begin{definition}\label{defn:involution-ast-class-groups}
Define the \textit{standard involution on $C(\l)$} to be the map:
\[ \ast \colon C(\l) \to C(\l), \quad [M] \mapsto -[M^*].\]
\end{definition}

It can be shown that this is an involution of abelian groups in the sense that $\ast$ is a group homomorphism such that $\ast^2 = \id_{C(\l)}$.
In the case $\l=\Z G$, we have an isomorphism $\wt K_0(\Z G) \cong C(\Z G)$ and the involution above coincides with the standard involution on $\wt K_0$ as defined in \cref{ss:k0}.

We now explore some properties of this involution. In what follows we refer to \cite[pp.~275-6]{CR87}. This deals only with the case $\l = \Z G$, though the arguments there apply to the more general setting described above.
Observe that an involution on an abelian group is the same structure as a $\Z C_2$-module, where the $C_2$-action is given by the involution.

\begin{proposition} \label{prop:SES-ZC_2}
Let $\Gamma$ be a maximal order in $A$ containing $\l$ and let $i \colon \l \hookrightarrow \Gamma$ denote the inclusion map. Then there is a short exact sequence of $\Z C_2$-modules
\[ 0 \to D(\l) \hookrightarrow C(\l) \xrightarrow[]{i_*} C(\Gamma) \to 0 \]
where $D(\l)$, $C(\l)$, and $C(\Gamma)$ are $\Z C_2$-modules under the standard involutions.
\end{proposition}

\begin{proof}
It is an immediate consequence of the alternative description of the kernel group given in \cref{ss:kernel-groups} that, if $[M] \in D(\l)$, then $[M^*] \in D(\l)$ \cite[p.~275]{CR87}. In particular, the involution $\ast$ restricts to $D(\l)$, and the map $D(\l) \hookrightarrow C(\l)$ is a $\Z C_2$-module homomorphism.
Thus $\ast$ induces an involution on $C(\Gamma)$ via the map $i_*$, and this coincides with the involution on $C(\Gamma)$ coming from the fact that $\Gamma$ is an order in $A$.
\end{proof}

We conclude this section by noting that the id\`{e}lic approach to class groups gives a different way to define an involution on $C(\l)$.
The involution $\ol{\cdot} \colon A \to A$ induces involutions on $A_p$ for each $p$ and so on $J(A)$. It can be shown that the involution fixes the subgroups $A^\times$, $U(\l)$ and $J_0(A)$ and so, by \cref{remark:CviaJ}, induces an involution on $C(\l)$ given by $[\l \alpha] \mapsto [\l \ol{\alpha}]$ (see \cite[p.~274]{CR87}).

The following is proven in \cite[p.~274]{CR87}.

\begin{proposition} \label{prop:ideles}
The involution on $C(\l)$ induced by the involution on $J(A)$ is the standard involution $\ast \colon C(\l) \to C(\l)$, $[M] \mapsto -[M^*]$.
\end{proposition}

This gives an alternate way to understand the standard involution on $\wt K_0(\Z G)$. We make further use of this description in the following section.

\section{Tate cohomology and \texorpdfstring{$\Z C_2$-modules}{modules with an involution}} \label{s:tate}

In this section, we recall some basic facts about $\Z C_2$-modules which are used throughout the proofs of \cref{theorem:main12-red,theorem:main3-red}. We also explain how the methods of Tate cohomology can be applied to $\Z C_2$-modules in preparation for the proof of \cref{theorem:main3-red}.

\subsection{Tate cohomology}

The following can be found in~\cite[VI.4]{Br94}.

\begin{definition}
Given a finite group $G$ and a $\Z G$-module $A$, the \emph{Tate cohomology groups} $\wh H^n(G;A)$ for $n \in \Z$ are defined as follows. Let $A^G = \{ x \in A \mid g \cdot x = x \text{ for all $g \in G$}\}$ be the invariants, let $A_G = A/\langle g\cdot x-x \mid g \in G, x \in A\rangle$ be the coinvariants and let $N \colon A_G \to A^G$ be the norm map $x \mapsto \sum_{g \in G} g \cdot x$ which is a well-defined homomorphism of abelian groups. Then define:
\[
\wh H^n(G;A) = \begin{cases}
H^n(G;A), & \text{if $n \ge 1$} \\
\coker(N \colon A_G \to A^G), & \text{if $n = 0$} \\
\ker(N \colon  A_G \to A^G), & \text{if $n = -1$} \\
H_{-n-1}(G;A), & \text{if $n \le -2$},
\end{cases}
\]
where $H^n$, $H_{-n-1}$ denote the usual group cohomology and homology groups.
\end{definition}

We now recall the following basic properties. The first can be found in \cite[VI.5.1]{Br94}, the second follows from the first since functoriality means that $\alpha_\ast$ is split whenever $\alpha$ is, the third is \cite[XII.2.5]{CE56}, and the fourth is \cite[XII.2.7]{CE56}.

\begin{proposition} \label{prop:tate-LES}
Let $G$ be a finite group.
\begin{clist}{(i)}
\item\label{item-prop:tate-LES-i}
Let $0 \to A \xrightarrow[]{\alpha} B \xrightarrow[]{\beta} C \to 0$ be a short exact sequence of $\Z G$-modules. Then there is a long exact sequence of Tate cohomology groups:
\begin{align*}
    \cdots \to \wh H^{n-1}(G;C) \xrightarrow[]{\partial}  \wh H^n(G;A) \xrightarrow[]{\alpha_*} \wh H^n(G;B) \xrightarrow[]{\beta_*} \wh H^n(G;C) \xrightarrow[]{\partial} \wh H^{n+1}(G;A) \to \cdots
\end{align*} 
\item\label{item-prop:tate-LES-ii}
Let $A, B$ be $\Z G$-modules. Then $\wh H^n(G;A \oplus B) \cong \wh H^n(G;A) \oplus \wh H^n(G;B)$ for all $n \in \Z$.
\item\label{item-prop:tate-LES-iii}
Let $A$ be a $\Z G$-module. Then $|G| \cdot \wh H^n(G;A)=0$, i.e.\ $|G|\cdot x = 0$ for all $x \in \wh H^n(G;A)$.
\item\label{item-prop:tate-LES-iv}
Let $A$ be a finite $\Z G$-module and suppose $(|G|,|A|)=1$. Then $\wh H^n(G;A)=0$ for all $n \in \Z$.
\end{clist}
\end{proposition}

\subsection{Tate cohomology and the structure of $\Z C_2$-modules} 

Let $A$ be a $\Z C_2$-module or, equivalently, an abelian group with an involution $\ol{\cdot} \colon A \to A$.
In order to prove \cref{theorem:main12-red,theorem:main3-red}, we would like to find techniques to determine the following groups associated to $A$:
\[ \{ x \in A \mid x=(-1)^n\ol{x}\}, \quad \{x+(-1)^n\overline{x} \mid x \in A \} , \quad \frac{\{ x \in A \mid x=(-1)^n\ol{x}\}}{\{x+(-1)^n\overline{x} \mid x \in A \} }.\]
For the groups on the left, the notation $A^{-} = \{ x \in A \mid x=-\ol{x}\}$ and $A^{+} = \{ x \in A \mid x = \ol{x}\}$ is often used since they are the $(-1)$ and $(+1)$-eigenspaces of the involution action.

The following lemma will suffice for the study of the first two groups associated to $A$. Part \eqref{item:lemma-useful-i} follows from the fact that $A \mapsto A^{C_2}$ is a left-exact functor where $A$ is given the altered involution $x \mapsto (-1)^n\overline{x}$, and part \eqref{item:lemma-useful-ii} is immediate.

\begin{lemma} \label{lemma:useful}
Let $0 \to A \xrightarrow[]{\alpha} B \xrightarrow[]{\beta} C \to 0$ be an exact sequence of $\Z C_2$-modules and let $n \in \Z$.
\begin{clist}{(i)}
\item\label{item:lemma-useful-i}
Then $\alpha, \beta$ induce an exact sequence of abelian groups:
\[ 0 \to \{x\in A \mid \overline{x}=(-1)^nx\} \xrightarrow[]{\alpha} \{x\in B \mid \overline{x}=(-1)^nx\} \xrightarrow[]{\beta} \{x\in C \mid \overline{x}=(-1)^nx\}.  \]
\item\label{item:lemma-useful-ii}
There are injective and surjective maps induced by $\alpha, \beta$:
\[ \{x+(-1)^n\overline{x} \mid x \in A \} \xhookrightarrow[]{\alpha} \{x+(-1)^n\overline{x} \mid x \in B \} \xrightarrowdbl[]{\beta} \{x+(-1)^n\overline{x} \mid x \in C \}.  \] 
\end{clist}
\end{lemma}

The third group associated to $A$ can be studied using Tate cohomology via the following \cite[p.~251]{CE56}.

\begin{proposition} \label{prop:tate-C2}
    Let $A$ be a $\Z C_2$-module and let $n \in \Z$. Then
    \[
\wh H^n(C_2;A) \cong \frac{\{x \in A \mid x= (-1)^n\ol{x}\}}{\{x+(-1)^n\ol{x} \mid x \in A\}}.
    \]
\end{proposition}

We now recall a series of special facts about the Tate cohomology of $C_2$. 
The first is a consequence of \cref{prop:tate-LES}~\eqref{item-prop:tate-LES-i} and the fact that finite cyclic groups have $2$-periodic cohomology (see \cite[p.~133]{Se79}). This also applies for $C_2$ replaced by an arbitrary finite cyclic group.

\begin{proposition} \label{prop:hexagon}
Let $0 \to A \xrightarrow[]{\alpha} B \xrightarrow[]{\beta} C \to 0$ be a short exact sequence of $\Z C_2$-modules. Then there is a $6$-periodic exact sequence of abelian groups: 
\[
\begin{tikzcd}
\wh H^1(C_2;A) \ar[r,"\alpha_*"] & \wh H^1(C_2;B) \ar[r,"\beta_*"] & \wh H^1(C_2;C) \ar[d,"\partial"] \\
\wh H^0(C_2;C) \ar[u,"\partial"] & \wh H^0(C_2;B) \ar[l,"\beta_*"] & \wh H^0(C_2;A). \ar[l,"\alpha_*"] 
\end{tikzcd}
\]
\end{proposition}

The following is a consequence of the theory of Herbrand quotients \cite[Chapter VIII \S 4]{Se79}.

\begin{proposition} \label{prop:tate-herbrand}
Let $A$ be a finite $\Z C_2$-module. Then there exists $d \ge 0$ such that 
\[ \wh H^n(C_2;A) \cong (\Z/2)^d\] 
for all $n \in \Z$. In particular, $|\wh H^n(G;A)|$ is independent of $n \in \Z$.       
\end{proposition}

For a finite abelian group $A$ and a prime $p$, let $A_{p} = \{ x \in A \mid p^n \cdot x = 0 \text{ for some $n \ge 1$}\}$ denote the $p$-primary component of $A$. 
The following is standard.

\begin{proposition} \label{prop:tate-2-part}
Let $A$ be a finite $\Z C_2$-module and let $n \in \Z$. Then
$\wh H^n (C_2;A) \cong \wh H^n(C_2;A_{2})$.
\end{proposition}

\section{Computing the involution on $\wt K_0(\Z C_m)$} \label{s:involution-ZC_m}

The aim of this section is to investigate the involution on $\wt K_0(\Z C_m)$ in preparation for the proofs of \cref{theorem:main12-red,theorem:main3-red} in \cref{s:proofs-algebra}.
We view $\wt K_0(\Z C_m)$ as a $\Z C_2$-module with the $C_2$-action coming from the standard involution on $\wt K_0$ as defined in \cref{ss:k0}.
We saw in \cref{ss:class-groups-basics} that $\wt K_0(\Z C_m) \cong C(\Z C_m)$ is an isomorphism of $\Z C_2$-modules where $C(\Z C_m)$ denotes the locally free class group, and in particular it is finite.
Our basic approach for computing $C(\Z C_m)$ is to use the following short exact sequence of $\Z C_2$-modules established in \cref{ss:class-groups-involution}
\[ 0 \to D(\Z C_m) \to C(\Z C_m) \to C(\Gamma_m) \to 0 \]
where $\Gamma_m$ is a maximal order in $\Q C_m$ containing $\Z C_m$ and $D(\Z C_m)$ has the induced involution.

The plan for this section is as follows. In \cref{ss:induced-inv}, we relate the involution on $C(\Gamma_m)$ to the involution on $C(\Z[\zeta_d])$ induced by conjugation.
In \cref{ss:ideal-class-groups}, we study the conjugation action on $C(\Z[\zeta_d])$ and its relation to the class numbers $h_d = |C(\Z[\zeta_d])|$. In \cref{ss:class-numbers}, we survey results on divisibility of class numbers as well as make minor extensions (see \cref{prop:odd(h_m)=1}~\eqref{item:prop-odd-hm-1-ii}).
In \cref{ss:kernel-groups-ZC_m}, we investigate the involution on $D(\Z C_m)$. 

\subsection{The induced involution on the maximal order} \label{ss:induced-inv}

Let $m \ge 2$ and let $C_m = \langle x \mid x^m \rangle$. Then there is an isomorphism of $\Q$-algebras:
\[ \Q C_m \cong \prod_{d \mid m} \Q(\zeta_d), \quad x \mapsto \prod_{d \mid m} (\zeta_d) \]
where $\zeta_d = e^{2 \pi i/d}$, which is a $d$th primitive root of unity. Since $\mathcal{O}_{\Q(\zeta_d)} = \Z[\zeta_d]$, it follows that
$ \Gamma_m = \prod_{d \mid m} \Z[\zeta_d]$
is a maximal order in $\Q C_m$. The image of $\Z C_m$ under the isomorphism above is contained in $\Gamma_m$ and so $\Gamma_m$ contains $\Z C_m$. In fact, $\Gamma_m$ is the unique maximal order in $\Q C_m$ containing $\Z C_m$ \cite[p.~243]{CR87}. This implies that there is an isomorphism of abelian groups
\[ C(\Gamma_m) \cong \bigoplus_{d \mid m} C(\Z[\zeta_d]), \]
where, by \cref{prop:LF=ideals}, $C(\Z[\zeta_d])$ coincides with the ideal class group of $\Z[\zeta_d]$.

For an integer $d \ge 1$, let $\ol{\,\cdot\,} \colon C(\Z[\zeta_d]) \to C(\Z[\zeta_d])$ denote the map induced by conjugation, i.e.\ if $\sigma\colon \Z[\zeta_d] \to \Z[\zeta_d]$ is the ring homomorphism generated by $\zeta_d \mapsto \zeta_d^{-1}$, then $\ol{\,\cdot\,} = \sigma_*$ is the induced map on $C(\Z[\zeta_d])$. We now compute the induced involution on $\bigoplus_{d \mid m} C(\Z[\zeta_d])$. 

\begin{proposition} \label{prop:C(ZC_m)-via-ideals}
Let $i \colon \Z C_m \hookrightarrow \Gamma_m$, $x \mapsto \prod_{d \mid m} (\zeta_d)$ 
and let $i_* \colon C(\Z C_m) \to \bigoplus_{d \mid m} C(\Z [\zeta_d])$ denote the induced map.
Then, under $i_*$, the standard involution on $C(\Z C_m)$ induces the conjugation map on each $C(\Z[\zeta_d])$.
\end{proposition}

This was shown in the case where $m$ is prime in \cite{Re68} (see also \cite[p.~275]{CR87}).
Since $\Gamma_m$ is a maximal order, the map $i_*$ is surjective (see \cref{ss:kernel-groups}).

\begin{proof}
For each $d \mid m$, let $i^{(d)} \colon \Z C_m \to \Z[\zeta_d]$, $x \mapsto \zeta_d$ where $x$ denotes the generator of $C_m$. It suffices to prove that, under $i^{(d)}_*\colon  C(\Z C_m) \to C(\Z[\zeta_d])$, the involution on $C(\Z C_m)$ induces conjugation on $C(\Z[\zeta_d])$.
By \cref{prop:ideles}, the standard involution on $C(\Z C_m)$ is induced by the involution on the id\`{e}le group $J(\Q C_m)$. Note that $i^{(d)}_*$ is induced by the map $J(i^{(d)})\colon J(\Q C_m) \to J(\Q(\zeta_d))$.
Under the map $i^{(d)}$, the involution on $\Q C_m$ induces conjugation on $\Q(\zeta_d)$. In particular, the involution on $C(\Z[\zeta_d])$ induced by $i^{(d)}_*$ coincides with the involution induced by conjugation on $J(\Q(\zeta_d))$. The result now follows since, if a locally free $\Z[\zeta_d]$-ideal $M = (x_1,\dots,x_n) \subseteq \Q(\zeta_d)$ is represented by $\alpha \in J(\Q(\zeta_d))$, then $\ol{M} = (\ol{x}_1,\dots,\ol{x}_n) \subseteq \Q(\zeta_d)$ is represented by $\ol{\alpha} \in J(\Q(\zeta_d))$.
\end{proof}

In summary, we have shown that there is a short exact sequence of $\Z C_2$-modules
\[ 0 \to D(\Z C_m) \to C(\Z C_m) \to \bigoplus_{d \mid m} C(\Z[\zeta_d])\to 0 \]
where $C(\Z C_m)$ has the standard involution, $D(\Z C_m)$ has the induced involution, and each $C(\Z[\zeta_d])$ has the involution induced by conjugation.

\subsection{Ideal class groups of cyclotomic fields} \label{ss:ideal-class-groups}

For every integer $m \ge 2$, let $\lambda_m = \zeta_m+\zeta_m^{-1}$.
Let $i \colon \Z[\lambda_m] \hookrightarrow \Z[\zeta_m]$ denote the inclusion map and recall that $i_* \colon C(\Z[\lambda_m]) \to C(\Z[\zeta_m])$ is injective \cite[Theorem 4.2]{La78}.
Furthermore, the norm map gives a surjection $N \colon C(\Z[\zeta_m]) \to C(\Z[\lambda_m])$ such that the composition
\[ C(\Z[\zeta_m]) \xrightarrow[]{N} C(\Z[\lambda_m]) \xrightarrow[]{i_*} C(\Z[\zeta_m]) \]
is the map $x \mapsto x+\ol{x}$ (see, for example, \cite[pp.~83-4]{La78}). 
By viewing $C(\Z[\zeta_d])$ and $C(\Z[\lambda_m])$ as $\Z C_2$-modules under the conjugation action, the maps $i_*$ and $N$ are $\Z C_2$-module homomorphisms. The conjugation action induces the identify on $C(\Z[\lambda_m])$.

This has the following useful consequences. Recall that, for $A$ a $\Z C_2$-module, we defined $A^{-} = \{ x \in A \mid x=-\ol{x}\}$ and $A^{+} = \{ x \in A \mid x = \ol{x}\}$.

\begin{lemma} \label{lemma:C-} 
\mbox{}
\begin{clist}{(i)}
\item\label{item:lemma-C-i}
The map $i_*$ induces an isomorphism $C(\Z[\lambda_m]) \cong \{x+\ol{x} \mid x \in C(\Z[\zeta_m]) \}$.
\item\label{item:lemma-C-ii}
There is a short exact sequence of $\Z C_2$-modules:
\[ 0 \to C(\Z[\zeta_m])^{-} \to C(\Z[\zeta_m]) \xrightarrow[]{N} C(\Z[\lambda_m]) \to 0.\]
\end{clist}	
\end{lemma}
 
\begin{remark}
Since $\Q(\lambda_m)$ is the maximal real subfield of $\Q(\zeta_m)$, it is often written as $\Q(\zeta_m)^+$. However, whilst $C(\Z[\lambda_m]) \subseteq C(\Z[\zeta_m])^+$ is a subgroup, these groups are not equal in general. For example, if $m=29$, then $C(\Z[\lambda_{29}]) = 0$ and $C(\Z[\zeta_{29}])^+ \cong (\Z/2)^3$. 
\end{remark}
 
\begin{proof}
\eqref{item:lemma-C-i} Since $N$ is surjective, we have 
\[\IM(i_*) = \IM(i_* \circ N) = \{x+\ol{x} \mid x \in C(\Z[\zeta_m])\}.\]
\eqref{item:lemma-C-ii} Since $i_*$ is injective, we have 
\[\ker(N) = \ker(i_* \circ N) = \{ x \in C(\Z[\zeta_m]) \mid x+\ol{x}=0\}.\qedhere\]
\end{proof}

In order to set up later applications, we now use \cref{lemma:C-} to obtain information about each of the following groups:
\[ \{ x \in C(\Z[\zeta_m]) \mid x=-\ol{x}\}, \, \{x-\overline{x} \mid x \in C(\Z[\zeta_m]) \} , \, \underbrace{\frac{\{ x \in C(\Z[\zeta_m]) \mid x=-\ol{x}\}}{\{x-\overline{x} \mid x \in C(\Z[\zeta_m]) \} }}_{\cong \wh H^1(C_2 ; C(\Z[\zeta_m]))}.\]

Since $C(\Z[\zeta_m])$ is a finite abelian group (see \cref{ss:class-groups-basics}), we can make the following definition.

\begin{definition} \label{def:class-number}
Define the \textit{class number} of the cyclotomic integers to be $h_m := |C(\Z[\zeta_m])|$. Define the \textit{minus part} of the class number to be $h_m^{-} := |C(\Z[\zeta_m])^{-}|$ and the \textit{plus part} of the class number to be $h_m^+ := |C(\Z[\lambda_m])|$. By \cref{lemma:C-}~\eqref{item:lemma-C-ii}, we have $h_m = h_m^{-} h_m^{+}$.
\end{definition}

For an integer $m$, let $\odd(m)$ denote the unique odd integer $r$ such that $m = 2^k r$ for some $k$.

\begin{proposition} \label{prop:inv-ideal-class-1}
$\odd(h_m^{-})$ divides $|\{x-\ol{x} \mid x \in C(\Z[\zeta_m])\}|$.
\end{proposition}

\begin{proof}
Since $\ol{x} = - x$ for all $x \in C(\Z[\zeta_m])^{-}$, we have $2 \cdot C(\Z[\zeta_m])^{-} \le \{x-\ol{x} \mid x \in C(\Z[\zeta_m])\}$. Since $C(\Z[\zeta_m])^{-}$ is a finite abelian group, we have
$C(\Z[\zeta_m])^{-} \cong A \oplus B$
where $|A|$ is even and $|B| = \odd(h_m^{-})$ is odd. Since $B$ is odd, $2 \cdot B = B$ and so
\[ B \le  2 \cdot A \oplus B = 2 \cdot C(\Z[\zeta_m])^{-} \le  \{x-\ol{x} \mid x \in C(\Z[\zeta_m])\}.\]
Then use that $\{x-\ol{x} \mid x \in C(\Z[\zeta_m])\}$ has a subgroup of size $\odd(h_m^{-})$.
\end{proof}

\subsection{Divisibility of class numbers of cyclotomic fields} \label{ss:class-numbers}

The aim of this section is to survey results on the divisibility of the class numbers $h_m$ and $h_m^{-}$. The most basic divisibility results are that, for $n \mid m$, we have $h_n \mid h_m$ \cite[p.~205]{Wa97} and $h_n^{-} \mid h_m^{-}$ \cite[Lemma 5]{MM76}.
Motivated by \cref{prop:inv-ideal-class-1}, we pursue divisibility results of two distinct types. We start by considering $\odd(h_m^{-})$, and we then consider the parity of $h_m^{-}$.

Recall the following theorem of Masley-Montgomery \cite{MM76} (see also \cite[p.~205]{Wa97}). In anticipation of its application in the proof of \cref{theorem:main12-red}, we separate out the case where $m$ is square-free.

\begin{proposition} \label{prop:h_m=1}
The integers $m \ge 2$ for which $h_m^{-}=1$ are as follows. 
\begin{clist}{(i)}
\item\label{item:prop:h-m=1=i}
If $m$ is square-free, then
\[ m =  
\begin{cases} 
\text{$p$}, & \text{where $p \in \{2, 3,5,7,11,13,17,19\}$} \\
\text{$2p$}, & \text{where $p \in \{3,5,7,11,13,17,19\}$} \\
\text{$pq$ or $2pq$}, &	\text{where $(p,q) \in \{(3,5),(3,7),(3,11),(5,7)\}$}.
\end{cases}
\]
\item\label{item:prop:h-m=1=ii}
If $m$ is not square-free, then 
\[ m \in \{4,8,9,12,16,18,20,24,25,27, 28, 32, 36,40,44,45,48,50,54,60,84,90\}. \]
\end{clist}
Furthermore, $h_m^{-}=1$ if and only if $h_m=1$.
\end{proposition}

The proof of \cref{prop:h_m=1} is based on the fact that $h_m^{-} \to \infty$ as $m \to \infty$. We now establish lower bounds on the growth rate of $h_m^{-}$, and hence on $h_m$ since $h_m \ge h_m^{-}$. 
Let $\varphi(m)$ denote Euler's totient function.

\begin{proposition} \label{prop:h_m-bound}
There exists a constant $C > 0$ such that, for all $m \ge 1$, we have: 
\[ h_m^{-} \ge e^{C \tfrac{m \log m}{\log \log m}}.\]
In particular, $h_m^{-} \to \infty$ super-exponentially in $m$.	
\end{proposition}

\begin{proof}
It is shown in \cite[Theorem 4.20]{Wa97} that $\log h_m^{-}/(\frac{1}{4}\varphi(m) \log m) \to 1$ as $m \to \infty$. This implies that $\log h_m^{-} \ge C_0 \varphi(m) \log m$ for some $C_0>0$.
The result now follows from the fact that $\varphi(m) \ge m/(2\log\log m)$ for $m$ sufficiently large \cite[Theorem 328]{HW54}.
\end{proof}

The following result gives the analogue of \cref{prop:h_m=1} for $\odd(h_m^{-})$.
Note that $h_m^{-} =1$ implies $\odd(h_m^{-})=1$, so we need not consider these $m$ since they are classified in \cref{prop:h_m=1}.
 This was established by Horie \cite[Theorems 2 and 3]{Ho89} and builds on Friedman's theorem from Iwasawa theory \cite{Fr82} and the Brauer-Siegel theorem for abelian fields \cite{Uc71}.

\begin{proposition} \label{prop:odd(h_m)=1}
The complete list of $m \ge 2$ for which $\odd(h_m^{-})=1$ and $h_m^{-}\ne 1$ is as follows.
\begin{clist}{(i)}
\item\label{item:prop-odd-hm-1-i}
If $m$ is square-free, then $m \in \{29,39,58,65,78,130\}$.
\item\label{item:prop-odd-hm-1-ii}
If $m$ is not square-free, then $m \in \{56,68,120\}$.
\end{clist}
Furthermore, $\odd(h_m^{-})=1$ if and only if $\odd(h_m)=1$.
\end{proposition}

In \cite[Theorem 1]{Ho89}, Horie also showed that $\odd(h_m^{-}) \to \infty$ as $m \to \infty$ but gave no bound on the growth rate. In fact, we have the following result analogous to \cref{prop:h_m-bound}.

\begin{proposition} \label{prop:odd(h_m)-bound}
There exists a constant $C > 0$ such that, for all $m \ge 1$, we have: 
\[ \odd (h_m^{-}) \ge e^{C \tfrac{m \log m}{\log \log m}}.\]
In particular, $\odd(h_m^{-}) \to \infty$ super-exponentially in $m$.
\end{proposition}

We prove this by tracing through Horie's proof of \cite[Theorem 1]{Ho89}.
An \textit{abelian field} $K$ is a finite Galois extension $K/\Q$ with $\Gal(K/\Q)$ abelian, and we can assume that $K \subseteq \C$. For an abelian field $K$ with maximal real subfield $K^+$, let $h_K = |C(\mathcal{O}_K)|$, $h_K^+ = |C(\mathcal{O}_{K^+})|$ and $h_K^{-} = h_K/ h_K^+$ (which is an integer).
Let $\disc(K)$ denote the discriminant of a number field $K$.

\begin{proof}
We use that, if $L/K$ is an extension of abelian fields, then $\odd(h_K^{-}) \mid \odd(h_L^{-})$ \cite[Lemma 1]{Ho89}. For each abelian field $K$, let $K'$ denote the maximal subfield of $K$ with degree a power of 2. By the fundamental theorem of Galois theory and the fact that a finite abelian group $A$ has a subgroup of order $d$ for all $d \mid |A|$, we get that $|K'/\Q|$ is the highest power of 2 dividing $|K/\Q|$.
For an integer $n \ge 1$, let $\mathcal{A}_n = \{ m \mid \odd(h_m^{-}) \le n\}$ and $\mathcal{B}_n = \{ |\Q(\zeta_m)'/\Q| \mid m \in \mathcal{A}_n\}$. 

We begin by finding a bound for $\sup(\mathcal{B}_n)$.
Let $K = \Q(\zeta_m)'$ for some $m \in \mathcal{A}_n$. Then $\Q(\zeta_m)/K$ is an extension of abelian fields and so $\odd(h_K^{-}) \mid h_m^{-}$ and so $\odd(h_K^{-}) \le n$. Since $|K/\Q|$ is a power of 2, $h_K$ must be odd \cite{Wa97} and so $h_K^{-} = \odd(h_K^{-}) \le n$. Furthermore, $K$ is imaginary unless $K = \Q$ \cite[p.~468]{Ho89}.

It follows from the proof of Theorem 1 and Proposition 1 in \cite{Uc71} that $\frac{|K/\Q|}{\log |\disc(K)|}$ is uniformly bounded across imaginary abelian fields $K$, where $\disc(K)$ denotes the discriminant, and that $h_{K}^{-} \ge |\disc(K)|^6$ for all but finitely many imaginary abelian fields $K$. This implies that there exists a constant $C' > 0$ such that $|K/\Q| \le C' \log(h_{K}^{-})$ for all imaginary abelian fields $K$.
Hence, if $K = \Q(\zeta_m)'$ for $m \in \mathcal{A}_n$, then $|K/\Q| \le C' \log n$. This implies that $\sup(\mathcal{B}_n) \le C' \log n$. It also follows that there are only finitely many fields of the form $\Q(\zeta_m)'$ for $m \in \mathcal{A}_n$.

We now aim to find a bound for $\sup(\mathcal{A}_n)$. 
First let $S$ denote the set of primes which are ramified in some field $\Q(\zeta_m)'$ for $m \in \mathcal{A}_n$. This is finite since there are finitely many such fields, and coincides with the primes which are ramified in $\Q(\zeta_m)$ for some $m \in \mathcal{A}_n$ \cite[p.~468]{Ho89}.
Let $S = \{p_1, \dots, p_s\}$ for distinct primes $p_i$.
By \cite[p.~469]{Ho89} there exists a cyclotomic field $L = \Q(\zeta_\ell)$ such that $L \subseteq \Q(\zeta_m) \subseteq L_\infty$ where $L_\infty$ is the basic $\Z_{p_1} \times \cdots \times \Z_{p_s}$-extension over $L$. Furthermore, we have $|\Q(\zeta_m)/L| = \prod_{i=1}^s p_i^{n(p_i)}$ for some $n(p_i) \ge 1$ and so, in the notation of \cite{Fr82}, we can write $\Q(\zeta_m) = L_N$ where $N = (n(p_1),\dots,n(p_s))$.
Let $e_N^{(2)}$ denote the highest power of 2 dividing $h_m$. By \cite[Theorem B]{Fr82}, we have that $e_N^{(2)} = A \cdot n(2)+B$ for all but finitely many $N$, where $A, B \ge 0$ are integers that do not depend on $N$. Here we define $n(2) = n(p_i)$ if $p_i=2$ for some $1 \le i \le s$, and $n(2)=0$ otherwise. This implies that there exists a constant $C'' > 0$ such that $e_N^{(2)} \le C'' \cdot n(2)$ for all $N$.

Let  $K = \Q(\zeta_m)'$.
Then $|\Q(\zeta_m)/\Q| = 2^r t$ for some $r \ge 0$ and $t$ odd, where $2^r = |K/\Q| \le C' \log n$ by the bound on $\mathcal{B}_n$.
Since $|\Q(\zeta_m)/L| \mid |\Q(\zeta_m)/\Q|$, we have that $2^{n(2)} \le 2^r \le C'\log n$. 
Hence $h_m^{-}/\odd(h_m^{-}) \le h_m/\odd(h_m) = 2^{e_N^{(2)}} \le (C' \log n)^{C''}$. Since $m \in \mathcal{A}_n$, we have $\odd(h_m^{-}) \le n$. This gives that $h_m^{-} \le a (\log n)^b n$ for some constants $a,b > 0$ and so, for any $\varepsilon>0$, we have that $h_m^{-} \le a n^{1+\varepsilon}$.
Combining this with \cref{prop:h_m-bound} gives that $\log(an^{1+\varepsilon}) \ge C_0 \frac{m \log m}{\log \log m}$ for some $C_0 > 0$, which implies that $\log n \ge C \frac{m \log m}{\log \log m}$ for some $C > 0$.

Finally, fix $m \ge 2$. Then $m \in \mathcal{A}_n$ where $n = \odd(h_m^{-})$, and so gives 
\[\textstyle \log(\odd(h_m^{-})) = \log n \ge C \frac{m \log m}{\log \log m},\]
which is the required bound.
\end{proof}

We now consider the parity of $h_m$ and $h_m^{-}$. Whilst we do not explicitly make use of it, we record the following basic observation which dates back to Kummer (see \cite[Satz 45]{Ha52}, \cite[Remark 1]{Yo98}).

\begin{lemma} \label{lemma:h_m-general}
Let $m \ge 2$. Then $h_m$ is odd if and only if $h_m^{-}$ is odd.
\end{lemma}

We now state more detailed results in the case where $m$ is a prime power. 

\begin{lemma} \label{lemma:h_p}
Let $p$ be a prime such that $p \le 509$ and let $n \ge 1$. Then:
\begin{clist}{(i)}
\item\label{item:lemma-h-p-i}
$h_p$ is odd if and only if 
\[ p \not \in \{29, 113, 163, 197, 239, 277, 311, 337, 349, 373, 397, 421, 463, 491 \}. \]
\item\label{item:lemma-h-p-ii}
$h_{p^n}$ is odd if and only if $h_p$ is odd.
\end{clist}
\end{lemma}

\begin{proof}
\eqref{item:lemma-h-p-i} This was proven by Schoof \cite[Table 4.4]{Sc98}.

\eqref{item:lemma-h-p-ii} The case $p=2$ is proven in \cite[Satz 36']{Ha52}, but the original result is attributed to an 1886 article of Weber \cite{We86}. See also \cite[p.~2590]{Yo98}.
For $n \ge 1$ and $p \le 509$ an odd prime, it was shown by Ichimura--Nakajima that $h_{p^n}/h_p$ is odd \cite[Theorem 1~(II)]{IN12}. The result follows.
\end{proof}

\begin{remark}
Prior to the results of Ichimura--Nakajima, it was shown by Washington that $h_{p^n}/h_p$ is odd for $p=3,5$ \cite{Wa75}. Both results of Washington and Ichimura--Nakajima all depend on Iwasawa theory. As far as we are aware, this is an essential ingredient in all known proofs that there exists an odd prime $p$ such that $h_{p^n}$ is odd for all $n \ge 1$.
\end{remark}

\subsection{Kernel groups of $\Z C_m$} \label{ss:kernel-groups-ZC_m}

The aim of this section is to determine the involution on $D(\Z C_m)$ which is induced by the involution on $C(\Z C_m)$ (see \cref{prop:SES-ZC_2}).

We begin with the following classical result due to Rim \cite[Theorem 6.24]{Ri59} (see also \cite[Theorem 50.2]{CR87}). Recall that, if $p$ is a prime and $\Gamma_p$ is the maximal order in $\Q C_p$ containing $\Z C_p$, then $\Gamma_p \cong \Z \times \Z[\zeta_p]$ and so $C(\Gamma_p) \cong C(\Z[\zeta_p])$ since $\Z$ is a PID.

\begin{lemma}
Let $p$ be a prime. Then the map $\Z C_p \to \Z[\zeta_p]$, $x \mapsto \zeta_p$ induces an isomorphism $C(\Z C_p) \cong C(\Z[\zeta_p])$. In particular, $D(\Z C_p)=0$.
\end{lemma}

We now determine $D(\Z C_m)$, as well as its involution, in the case where $m$ is square-free. This is used in the proof of \cref{theorem:main12-red}. As usual, we view an abelian group with involution as a $\Z C_2$-module.

Let $\pi_1=1$, let $\pi_p=1-\zeta_p$ for a prime $p$ and, more generally, let $\pi_m = \prod_{p \mid m}\pi_p$ for an integer $m \ge 2$. Note that $\Z[\zeta_m]/\pi_m \cong \bigoplus_{p\mid m} \Z[\zeta_m]/\pi_p$ since the $\pi_p$ are coprime (see, for example, \cite[p.~249]{CR87}). 
Let $\Psi_m \colon \Z[\zeta_m]^\times \to (\Z[\zeta_m]/\pi_m)^\times$ be the natural map.

\begin{definition} \label{def:V_m}
For a square-free integer $m \ge 2$, define
\[ V_m = \coker(\Psi_m \colon \Z[\zeta_m]^\times \to (\Z[\zeta_m]/\pi_m)^\times). \] 
We view this as a $\Z C_2$-module with the involution induced by the conjugation map 
\[ \ol{\cdot} \colon (\Z[\zeta_m]/\pi_m)^\times \to (\Z[\zeta_m]/\pi_m)^\times, \quad \zeta_m \mapsto \zeta_m^{-1}.\]	
\end{definition}

If $p$ is prime, then $\Z[\zeta_p]/\pi_p \cong \F_p$. We can then see that $\Psi_p \colon \Z[\zeta_p]^\times \to \F_p^\times$ is surjective by considering the cyclotomic units $1+\zeta_p + \cdots + \zeta_p^{i-1} \in \Z[\zeta_p]^\times$ for $(i,p)=1$. In particular, $V_p=1$.
We now show how $D(\Z C_m)$ is related to $V_d$ for $d \mid m$.

\begin{lemma} \label{lemma:D-subnormal-series}
Let $m \ge 2$ be a square-free integer. Let $d_1,\dots, d_n$ be the distinct nontrivial positive divisors of $m$, ordered such that $d_{i+1}$ has at least as many prime factors as $d_i$ $($so $d_1$ is prime and $d_n=m)$.
Then there is a chain of $\Z C_2$-modules
\[1 = A_0 \le \cdots \le A_n = D(\Z C_m)\] 
such that $A_{i}/A_{i-1} \cong V_{d_{n-i+1}}$ for $1 \le i \le n$. 
\end{lemma}

It is proven in \cite[Theorem 50.6]{CR87} that $|D(\Z C_m)| = \prod_{d \mid n} |V_d|$. Our proof will involve following the argument given there, and extending it to determine the group structure and the involution.

\begin{proof}
Firstly, $\Gamma = \bigoplus_{d \mid m} \Z[\zeta_d]$ is the unique maximal order in $\Q C_m$ which contains $\Z C_m$. This implies that there is a pullback square:
\[
\begin{tikzcd}
\Z C_m \ar[r,"i_2"] \ar[d,"i_1"] & \Gamma \ar[d,"j_2"] \\
\Z C_m/m\Gamma \ar[r,"j_1"] & \Gamma/m\Gamma
\end{tikzcd}
\]
where the inclusion $m \Gamma \subseteq \Z C_m$ follows from \cite[Theorem 41.1]{Re75}.

By \cite[p.~246]{CR87}  this induces an exact sequence
\[(\Z C_m/ m \Gamma)^\times \oplus \Gamma^\times \xrightarrow[]{(j_1,j_2)} (\Gamma/m\Gamma)^\times \xrightarrow[]{\partial} D(\Z C_m) \to 0\] 
where $\partial \colon u \mapsto M(u)$ where 
\[ M(u) = \{(x,y) \in (\Z C_m/m\Gamma) \times \Gamma \mid j_1(x)= j_2(y)u \in \Gamma/m\Gamma \} \] 
is the $\Z C_m$-module with action $\lambda \cdot (x,y) = (i_1(\lambda)x,i_2(\lambda)y)$ for $\lambda \in \Z C_m$. 
We now claim that the conjugation map on $(\Gamma/m\Gamma)^\times$ induces the involution on $D(\Z C_m)$. Firstly, by \cref{prop:ideles}, we have that the involution on $D(\Z C_m)$ is induced by the natural involution $x \mapsto x^{-1}$ on the id\`{e}le group $J(\Q C_m) \subseteq \prod_p \Q_p C_m$. For all primes $p$, we have $\Z_p C_m \subseteq \Gamma_p \oplus (\Z C_m/m\Gamma)_p \subseteq \Q_p C_m$.
If $M(u) = \Z C_m \alpha$ for $\alpha = (\alpha_p) \in J(\Q C_m)$, then \cite[Exercise 53.1]{CR87} implies that
\[ \alpha_p = \begin{cases}
 (1,1) \in \Gamma_p \oplus (\Z C_m/m\Gamma)_p, & \text{if $(p,m)=1$} \\
 	(u_p,1) \in \Gamma_p \oplus (\Z C_m/m\Gamma)_p, & \text{if $(p,m)\ne 1$}, 
 \end{cases}
 \]
where $u_p \in \Gamma_p$ is any element such that $j_2(u_p) = [u] \in (\Gamma/m\Gamma)_p$.
By the same argument as in the proof of \cref{prop:C(ZC_m)-via-ideals}, the involution on $J(\Q C_m)$ induces an involution on $J(\Q(\zeta_d))$ which coincides with the involution induced by conjugation. In particular, the involution maps $\alpha_p \mapsto (1,1)$ or $(\ol{u}_p,1)$ where $\ol{\,\cdot\,}\colon \Gamma_p \to \Gamma_p$ is induced by conjugation on $\Gamma$. In particular, this coincides with the involution induced by conjugation on $(\Gamma/m\Gamma)^\times$.

By \cite[Lemma 50.7]{CR87}, $j_1$ can be replaced by the map  $\alpha \colon (\Z C_m /m\Z C_m)^\times \to (\Gamma/m\Gamma)^\times$. By \cite[Lemma 50.8]{CR87}, $\coker(\alpha) \cong \bigoplus_{d \mid m} (\Z[\zeta_d]/\pi_d)^\times$ and it follows from the proof that the conjugation map on $(\Gamma/m\Gamma)^\times$ induces conjugation on $(\Z[\zeta_d]/\pi_d)^\times$ for each $d \mid m$.
If $\gamma \colon \Gamma^\times \to \coker(\alpha)$ is the map induced by $j_2$, then we obtain an exact sequence
\[ \bigoplus_{d \mid m} \Z[\zeta_d]^\times \xrightarrow[]{\gamma} \bigoplus_{d \mid m} (\Z[\zeta_d]/\pi_d)^\times \xrightarrow[]{\overline{\partial}} D(\Z C_m) \to 0. \]

Let $d_1, \dots, d_n$ be the ordered sequence of divisors of $m$.
By \cite[p.~250]{CR87} we have a map
\[ \gamma \mid_{\Z[\zeta_{d_i}]^\times} \colon \Z[\zeta_{d_i}]^\times  \to (\Z[\zeta_{d_i}]/\pi_{d_i})^\times \oplus \bigoplus_{p \mid \frac{m}{d_i}} (\Z[\zeta_{p{d_i}}]/\pi_p)^\times \subseteq \bigoplus_{j \ge i} (\Z[\zeta_{d_j}]/\pi_{d_j})^\times \]
given by $x \mapsto (x,x^{-1}, \dots, x^{-1})$,
where the last inclusion comes from the fact that $\Z[\zeta_d]/\pi_p \subseteq \bigoplus_{q \mid d} \Z[\zeta_d]/\pi_q \cong \Z[\zeta_d]/\pi_d$ for primes $q$ since the $\pi_q$ are pairwise coprime in $\Z[\zeta_d]$. If $x \in \Z[\zeta_k]^\times$, then $x^{-1} \in (\Z[\zeta_k]/\pi_p)^\times \subseteq (\Z[\zeta_{pk}]/\pi_p)^\times$.

For $1 \le i \le n$, this shows that $\gamma$ restricts to a map $\gamma_i \colon \bigoplus_{j \ge i} \Z[\zeta_{d_j}]^\times \to \bigoplus_{j \ge i} (\Z[\zeta_{d_j}]/\pi_j)^\times$ where $\gamma=\gamma_1$.
By a mild generalisation of \cite[Exercise 50.2]{CR87}, this implies that there is an exact sequence induced by the projection map:
\[ 1 \to \coker(\gamma_{i+1}' \colon W \to \bigoplus_{j \ge i+1} (\Z[\zeta_{d_j}]/\pi_j)^\times   )      \to \coker(\gamma_i) \to \coker(\Psi_{d_i}) \to 1,\]
where $W = \{x \in \Z[\zeta_{d_i}]^\times \mid x \equiv 1 \mod \pi_{d_i}\} \oplus \bigoplus_{j \ge i+1} \Z[\zeta_{d_j}]^\times$. By \cite[p.~252]{CR87}, we have $\IM(\gamma_{i+1}')=\IM(\gamma_{i+1})$ and so $\coker(\gamma_{i+1}') = \coker(\gamma_{i+1})$. Let $A_i =\coker(\gamma_{n-i+1})$ for $1 \le i \le n$ and $A_0=1$. Then we have that $1 = A_0 \le \cdots \le A_n = \coker(\gamma) = D(\Z C_m)$ such that there are isomorphisms $A_{i}/A_{i-1} \cong \coker(\Psi_{d_{n-i+1}})$ as abelian groups. 

Since the involution on $D(\Z C_m)$ is induced by conjugation on $(\Gamma/m\Gamma)^\times$, it follows that it restricts to $A_i$ where it acts via conjugation on the $\Z[\zeta_d]$. 
Hence, with respect to the involution, $A_i \le D(\Z C_m)$ is a $\Z C_2$-module and the chain $A_0 \le \cdots \le A_n$ is a chain of $\Z C_2$-modules.
Under the abelian group isomorphism 
\[ A_i/A_{i-1} \cong \coker(\Psi_{d_{n-i+1}} \colon \Z[\zeta_{d_{n-i+1}}]^\times \to (\Z[\zeta_{d_{n-i+1}}]/\pi_{d_{n-i+1}})^\times),\] 
the involution on $\coker(\Psi_{d_{n-i+1}})$ induced by the involution on $A_i$ coincides with the involution induced by conjugation on $(\Z[\zeta_{d_{n-i+1}}]/\pi_{d_{n-i+1}})^\times$. Hence there is an isomorphism of $\Z C_2$-modules $A_i/A_{i-1} \cong V_{n-i+1}$, as required. 
\end{proof}

Let $m \ge 2$ be a square-free integer.
We now give a method for analysing the involution on $V_m = \coker(\Psi_m)$. For a field $\F$ and $\alpha_1, \dots, \alpha_n \in \C$, we write $\F[\alpha_1, \dots, \alpha_n]$ to denote $\F \otimes_{\Z} \Z[\alpha_1, \dots, \alpha_n]$. Note that this may not be a field and so need not coincide with $\F(\alpha_1, \dots, \alpha_n)$.

Firstly, as described above, the quotient map induces a group homomorphism
\[ \Psi_m \colon \Z[\zeta_m]^\times \to \bigoplus_{p \mid m} (\Z[\zeta_m]/\pi_p)^\times, \]
where $p$ ranges over the prime factors of $m$. If $p \mid m$ then, since $m$ is square-free, $m$ is a product of coprime integers $m/p$ and $p$, and so we can take $\zeta_m = \zeta_p \cdot \zeta_{m/p}$. This implies that $\Z[\zeta_m] = \Z[\zeta_p,\zeta_{m/p}]$ and so 
\[ \Z[\zeta_m]/\pi_p \cong \Z[\zeta_p,\zeta_{m/p}]/\pi_p \cong \F_p[\zeta_{m/p}].\]
Let $m = p_1 \cdots p_n$ for distinct primes $p_i$. Then $\Psi_m$ can be written as
\[ \Psi_m \colon \Z[\zeta_m]^\times \to \bigoplus_{i=1}^n \F_{p_i}[\zeta_{m/p_i}]^\times, \]
where $\zeta_m = \prod_{j} \zeta_{p_j}$, $\zeta_{m/p_i} = \prod_{j \ne i} \zeta_{p_j}$ and the map $\Z[\zeta_m]^\times \to \F_{p_i}[\zeta_{m/p_i}]^\times$ is the map sending $\zeta_{p_i} \mapsto 1$. This motivates the following definition.

\begin{definition} \label{def:wt-V_m}
For a square-free integer $m \ge 2$, define
\[ \wt V_m \cong  \coker\left( \Psi_m^{+} \colon \Z[\zeta_m]^\times \to \bigoplus_{i=1}^n \smfrac{\F_{p_i}[\zeta_{m/p_i}]^\times}{\F_{p_i}[\lambda_{m/p_i}]^\times} \right), \]
where $\F_{p_i}[\lambda_{m/p_i}]^\times \le \F_{p_i}[\zeta_{m/p_i}]^\times$ is induced by the inclusion $\Z[\lambda_{m/p_i}] \le \Z[\zeta_{m/p_i}]$ and $\Psi_m^{+}$ is the composition of $\Psi_m$ with the quotient maps $\F_{p_i}[\zeta_{m/p_i}]^\times \twoheadrightarrow \F_{p_i}[\zeta_{m/p_i}]^\times/\F_{p_i}[\lambda_{m/p_i}]^\times$.

We view this as a $\Z C_2$-module with the involution induced by the conjugation map 
\[ \ol{\cdot} \colon \F_{p_i}[\zeta_{m/p_i}]^\times \to \F_{p_i}[\zeta_{m/p_i}]^\times, \quad \zeta_{m/p_i} \mapsto \zeta_{m/p_i}^{-1}.\]	
\end{definition}

\begin{lemma} \label{lemma:negation}
Let $m \ge 2$ be a square-free integer.
\begin{clist}{(i)}
\item\label{item:lemma-negation-i}
There is a surjective $\Z C_2$-module homomorphism $V_m \twoheadrightarrow \wt V_m$.
\item\label{item:lemma-negation-ii}
	$\wt V_m = \wt V_m^{-}$, i.e.\ if $x \in \wt V_m$, then $\ol{x} = -x \in \wt V_m$.
\end{clist}
\end{lemma}

\begin{proof}
\eqref{item:lemma-negation-i} For $1 \le i \le n$, we let $f_i = \id_{\F_{p_i}} \otimes_{\Z} \iota_i$, where $\iota_i \colon \Z[\lambda_{m/p_i}] \hookrightarrow \Z[\zeta_{m/p_i}]$ is the natural inclusion map.
Then the surjective homomorphism $V_m \twoheadrightarrow \wt V_m$ is induced by noting that
\begin{align*}  \textstyle\wt V_m &\cong \coker( (\Psi_m, f_1,\dots,f_n) \colon \Z[\zeta_m]^\times \oplus \bigoplus_{i=1}^n \F_{p_i}[\lambda_{m/p_i}]^\times \to \bigoplus_{i=1}^n \F_{p_i}[\zeta_{m/p_i}]^\times) \\
\textstyle & \cong \coker( (f_1,\dots,f_n) \colon \bigoplus_{i=1}^n \F_{p_i}[\lambda_{m/p_i}]^\times \to V_m).
\end{align*}

\eqref{item:lemma-negation-ii}
We first claim that, if $p$ is prime and $n$ is an integer, then $\alpha \in \F_p[\zeta_n]$ implies $\alpha \cdot \overline{\alpha} \in \F_p[\lambda_n]$. Let $\beta = \alpha \cdot \overline{\alpha}$. Since $\Z[\zeta_n]$ has integral basis $\{\zeta_n^i\}_{i=0}^{n-1}$, we can write $\beta = \sum_{i=0}^{n-1} a_i \otimes \zeta_n^i$ for $a_i \in \F_p$. Note that $\overline{\beta}=\beta$ which implies that 
\[ \textstyle \sum_{i=0}^{n-1} a_i \otimes \zeta_n^i = \sum_{i=0}^{n-1} a_{n-i} \otimes \zeta_n^i \in \F_p[\zeta_n].\] 
Since $\F_p[\zeta_n] = \F_p \otimes_{\Z} \Z[\zeta_n] \cong \F_p^n$ as an abelian group, we get that $a_n = a_{n-i} \in \F_p$ for all $i$ and so 
\[ \textstyle \beta = a_0 + \sum_{i=1}^{\lfloor \frac{n-1}{2} \rfloor} a_i \otimes (\zeta_n^i + \zeta_n^{-i}) + \varepsilon \in \F_p[\lambda_n],\]  
where $\varepsilon = -a_{n/2}$ for $n$ even and $\varepsilon = 0$ for $n$ odd.

Finally, let $f \colon V_m \to \wt V_m$ be the map described above and let $\alpha = [(\alpha_1, \dots \alpha_n)]$, where $\alpha_i \in \F_{p_i}[\zeta_{d/p_i}]^\times$. We have shown that $\alpha_i \cdot \overline{\alpha}_i \in \F_{p_i}[\lambda_{d/p_i}]^\times$ for all $i$ and so 
\[ f(\alpha) \cdot f(\overline{\alpha}) = [(\alpha_1\cdot \overline{\alpha}_1, \dots, \alpha_n \cdot \overline{\alpha}_n)] = [(1, \dots, 1)] = 1\] 
and $f(\overline{\alpha}) = f(\alpha)^{-1}$.	Since $f$ is surjective and induces the involution on $\wt V_m$, this implies that $\ol{x} = -x$ for all $x \in \wt V_m$,  where we now write the inverse as $-x$ rather than $x^{-1}$ since $\wt V_m$ is an abelian group.
\end{proof}

We now deduce the following, which is the main result of this section.
This is analogous to \cref{prop:inv-ideal-class-1} which applied in the case of ideal class groups.
Recall that, for an integer $m$, we let $\odd(m)$ denote the unique odd integer $r$ such that $m = 2^k r$ for some $k$.

\begin{proposition} \label{prop:D-conjugation}
Let $m \ge 2$ be a square-free integer. Let $d_1,\dots, d_n$ be the distinct nontrivial positive divisors of $m$, ordered such that $d_{i+1}$ has at least as many prime factors as $d_i$, and let $A_i$ be the $\Z C_2$-modules defined in \cref{lemma:D-subnormal-series}. Then there is a chain of abelian subgroups
\[ 1 = \{x-\ol{x} \mid x \in A_0\} \le \cdots \le \{x-\ol{x} \mid x \in A_n\} = \{x-\ol{x} \mid x \in D(\Z C_m)\} \]
and, for each $1 \le i \le n$, there are surjective group homomorphisms 
\[ \{x-\ol{x} \mid x \in A_i\}/\{x-\ol{x} \mid x \in A_{i-1}\} \twoheadrightarrow \{x-\ol{x} \mid x \in V_{d_{n-i+1}}\} \twoheadrightarrow 2 \cdot \wt V_{d_{n-i+1}}.\]
In particular, $\prod_{d \mid m} \odd(|\wt V_d|)$ divides $|\{x-\ol{x} \mid x \in D(\Z C_m)\}|$.
\end{proposition}

\begin{proof}
The chain of abelian subgroups follows direct from \cref{lemma:D-subnormal-series} and \cref{lemma:useful}~\eqref{item:lemma-useful-ii}. Let $1 \le i \le n$. By \cref{lemma:D-subnormal-series}, there is a short exact sequence of $\Z C_2$-modules
\[ 0 \to A_{i-1} \to A_i \to V_{d_{n-i+1}} \to 0. \]
By \cref{lemma:useful}~\eqref{item:lemma-useful-ii}, there are injective and surjective maps:
\[ \{x-\ol{x} \mid x \in A_{i-1}\}  \hookrightarrow \{x-\ol{x} \mid x \in A_i\} \twoheadrightarrow \{x-\ol{x} \mid x \in V_{d_{n-i+1}}\}.\]
Since the composition is necessarily the zero map, this gives the first surjective homomorphism. The second is a direct consequence of both parts of \cref{lemma:negation} as well as \cref{lemma:useful}~\eqref{item:lemma-useful-ii} again.

For the last part, the two statements we have proved so far imply that
\[ |\{x-\ol{x} \mid x \in D(\Z C_m)\}| = \prod_{i=1}^n |\{x-\ol{x} \mid x \in A_i\}/\{x-\ol{x} \mid x \in A_{i-1}\}| \]
and $|2 \cdot \wt V_{d_{n-i+1}}|$ divides $|\{x-\ol{x} \mid x \in A_i\}/\{x-\ol{x} \mid x \in A_{i-1}\}|$ for all $1 \le i \le n$. It now suffices to observe that $\odd(|\wt V_{d_{n-i+1}}|)$ divides $|2 \cdot \wt V_{d_{n-i+1}}|$ by the same argument as in the proof of \cref{prop:inv-ideal-class-1}.
\end{proof}

\begin{remark}
To obtain a complete analogue of \cref{prop:inv-ideal-class-1}, it would be desirable to also obtain bounds on $|\{x \in D(\Z C_m) \mid x=-\ol{x}\}|$. However, unlike \cref{prop:inv-ideal-class-1}, the bounds we obtain in \cref{prop:D-conjugation} are obtained by subquotients rather than just subgroups. We therefore cannot  apply \cref{lemma:useful}~\eqref{item:lemma-useful-i} since the final map in the sequence need not be surjective in general. As we shall see in \cref{s:proofs-algebra}, it is possible to circumvent the need for such bounds in the proof of \cref{theorem:main12-red}.
\end{remark}

We now conclude this section with a result which holds in the case that $m$ is not square-free. Firstly an analogue of \cref{lemma:D-subnormal-series} holds in the case that $m$ is a prime power, by Kervaire--Murthy~\cite[Theorem 1.2]{KM77}. For brevity, we do not state this result here. We instead make do with the following consequence of their result in the case that $p=2$ which is used in the proof of \cref{theorem:main3-red}. Note that $V_{2^{n+1}}$ is directly analogous to the $\Z C_2$-modules $V_m$ defined in \cref{def:V_m} when $m$ is square-free.

\begin{proposition} \label{prop:KM}
Let $n \ge 1$. Then there exists an exact sequence of $\Z C_2$-modules
\[ 0 \to V_{2^{n+1}}  \to D(\Z C_{2^{n+1}}) \to D(\Z C_{2^n}) \to 0, \]
where $V_{2^{n+1}} = \bigoplus_{i=1}^{n-2} (\Z/2^i)^{2^{n-i-2}}$ with the involution acting by negation.
\end{proposition}

\begin{proof}
This is a consequence of results of Kervaire--Murthy \cite{KM77}. In \cite[p.~419]{KM77} they show that there is an exact sequence of $\Z C_2$-modules
\[ 0 \to V_{2^{n+1}} \to \wt K_0(\Z C_{2^{n+1}}) \xrightarrow[]{\alpha} \wt K_0(\Z C_{2^n}) \oplus \wt K_0(\Z [\zeta_{2^{n+1}}]) \to 0 \]
where $\alpha$ is induced by the natural map of rings $\Z C_{2^{n+1}} \to \Z C_{2^n} \times \Z[\zeta_{2^{n+1}}]$. Since the maximal order $\Z C_{2^n} \subseteq \Gamma_{2^n} \subseteq \Q C_{2^n}$ is given by $\Gamma_{2^n} = \bigoplus_{i=1}^n \Z[\zeta_{2^i}]$, we get that 
\[ \ker(\alpha) \cong \ker( \beta \colon D(\Z C_{2^{n+1}}) \to D(\Z C_{2^n}))\]
where $\beta$ is the $\Z C_2$-module homomorphism induced by map $\Z C_{2^{n+1}} \to \Z C_{2^n}$. 
This gives an exact sequence of the required form. It follows from \cite[Theorem 1.1]{KM77} that $V_{2^{n+1}}$ is as described.
\end{proof}

\subsection{Divisibility and lower bounds for kernel groups}

The aim of this section is to establish divisibility results for $|D(\Z C_m)|$ and $\odd(|\wt V_m|)$. These results are necessary for determining the involution on $D(\Z C_m)$ in an analogous way to how divisibility results for class numbers $h_m$ were necessary for determining the involution on $C(\Z[\zeta_m])$ (see \cref{ss:class-numbers}). 
The results on $\odd(|\wt V_m|)$ are motivated by \cref{prop:D-conjugation}.

We begin by recalling the following, which is \cite[Theorem 50.18]{CR87}.

\begin{proposition} \label{prop:D-odd}
If $p$ is a prime and $G$ is a finite $p$-group, then $D(\Z G)$ is an abelian $p$-group. In particular, if $p \ne 2$, then $|D(\Z G)|$ is odd.
\end{proposition}

We now find conditions on $m \ge 2$ square-free for which $\odd(|\wt V_m|) \ne 1$. Our strategy is motivated by the bounds $d$ such that $d \mid |D(\Z C_m)|$ which were obtained by Cassou-Nogu\`{e}s in \cite{CN72, CN74}.
In particular, our argument shows that these bounds actually give factors of $|\wt V_m|$.
Recall from \cref{def:wt-V_m} that, for $m \ge 2$ square-free, we define $\wt V_m = \coker(\Psi_m^+)$ where 
\[ \Psi_m^{+} \colon \Z[\zeta_m]^\times \to \bigoplus_{i=1}^n \smfrac{\F_{p_i}[\zeta_{m/p_i}]^\times}{\F_{p_i}[\lambda_{m/p_i}]^\times} \]
and $p_1, \cdots, p_n$ are the distinct primes for which $m = p_1 \cdots p_n$.

\begin{lemma} \label{lemma:F_p[lambda]}
Let $m \ge 2$ be an integer, let $p$ be a prime such that $p \nmid m$, let $f_p = \ord_m(p)$ denote the order of $p$ in $\Z/m$ and let $g_p = \varphi(m)/2f_p$. 
\begin{clist}{(i)}
\item\label{item:lemma-F-p-lambda-i}
The inclusion $\Z[\lambda_m] \subseteq \Z[\zeta_m]$ induces inclusion $\F_p[\lambda_m] \subseteq \F_p[\zeta_m]$.

\item\label{item:lemma-F-p-lambda-ii}
If $m \ge 3$, then $|\F_p[\zeta_m]^\times| = (p^{f_p}-1)^{g_p}$.

\item\label{item:lemma-F-p-lambda-iii}
If $m \ge 3$, then $|\F_p[\lambda_m]^\times| = \begin{cases} 
(p^{\frac{f_p}{2}}-1)^{g_p}, & \text{if $f_p$ is even} \\ 
(p^{f_p}-1)^{\frac{g_p}{2}}, & \text{if $f_p$ is odd.} \end{cases}$
\end{clist}	
If $m=2$, then $\zeta_2, \lambda_2 \in \Z$ and so $|\F_p[\zeta_2]^\times|=|\F_p[\lambda_2]^\times|=p-1$. 
\end{lemma}

\begin{proof}
The proofs of \eqref{item:lemma-F-p-lambda-ii} and \eqref{item:lemma-F-p-lambda-iii} are analogous. We prove \eqref{item:lemma-F-p-lambda-iii} only as it is more complicated.
Firstly, $\Q(\lambda_m)/\Q$ is a Galois extension and $p$ is unramified in $\Q(\lambda_m)/\Q$ since it is unramified in $\Q(\zeta_m)/\Q$ (see, for example, \cite[Lemma 15.48]{Wa97}). This implies that $p\cdot \Z[\lambda_m] = \mathcal{P}_1 \cdots \mathcal{P}_{g}$ for some $g \ge 1$, where the $\mathcal{P}_i \subseteq  \Z[\lambda_m]$ are distinct prime ideals.
The $\mathcal{P}_i$ coincide under the Galois action and the $\Z[\lambda_m]/\mathcal{P}_i$ all coincide with the splitting field $\F_p(\lambda_m)$ and so
\[ \F_p[\lambda_m] \cong \Z[\lambda_m]/(p) \cong \F_p(\lambda_m)^g\]
which implies that $\F_p[\lambda_m]^\times \cong (\F_p(\lambda_m)^\times)^g$.

Let $f = [\F_p(\lambda_m):\F_p]$. Since $\Gal(\F_p(\lambda_m)/\F_p)$ is generated by the Frobenius element $\Frob_p \colon x \mapsto x^p$, we get that $f$ is the smallest positive integer such that $\Frob_p^f = \id_{\F_p(\lambda_m)}$. Note that $\Frob_p^f(\lambda_m) = \zeta_m^{p^f}+\zeta_m^{-p^f}$. This implies that $\Frob_p^f = \id_{\F_p(\lambda_m)}$ if and only if $\zeta_m^{p^f} = \zeta_m^{\pm 1}$ and so $f$ is the order of $p$ in $(\Z/m)^\times/\{\pm 1\}$.
Since $[\Q(\lambda_m):\Q] = \varphi(m)/2$, we have that $|\F_p[\lambda_m]| = p^{\varphi(m)/2}$. Since $|\F_p(\lambda_m)^g| = p^{fg}$, this gives that $g = \varphi(m)/2f$. Hence we have $|\F_p[\lambda_m]^\times| = |\F_p(\lambda_m)^\times|^g = (p^f-1)^g$.
Note that $f = f_p$ if and only if $f_p$ is odd, and otherwise $f = f_p/2$.

Finally, \eqref{item:lemma-F-p-lambda-i} follows by comparing the expressions for $\F_p[\lambda_m]$ and $\F_p[\zeta_m]$ as products of fields.
\end{proof}

We now use \cref{lemma:F_p[lambda]} to obtain bounds on $|\wt V_{pq}|$ for $p,q$ distinct odd primes, and for $|\wt V_{2p}|$, where $p$ is an odd prime respectively.

\begin{proposition} \label{prop:pq-case}
Let $p$ and $q$ be distinct odd primes. Let $f_p = \ord_q(p)$, $g_p = (q-1)/2f_p$, $f_q = \ord_p(q)$ and $g_q = (p-1)/2f_q$. Define
\[c_{pq} = \begin{cases} \frac{1}{2pq}(p^{\frac{f_p}{2}}+1)^{g_p}(q^{\frac{f_q}{2}}+1)^{g_q}, & \text{if $f_p, f_q$ are even} \\ 
\frac{1}{2pq}(p^{\frac{f_p}{2}}+1)^{g_p}(q^{f_q}-1)^{\frac{g_q}{2}}, & \text{if $f_p$ is even and $f_q$ is odd} \\
\frac{1}{2pq}(p^{f_p}-1)^{\frac{g_p}{2}}(q^{\frac{f_q}{2}}+1)^{g_q}, & \text{if $f_p$ is odd and $f_q$ is even} \\
\frac{1}{2pq}(p^{f_p}-1)^{\frac{g_p}{2}}(q^{f_q}-1)^{\frac{g_q}{2}}, & \text{if $f_p, f_q$ are odd.} \end{cases}\]
Then $c_{pq} \mid |\wt V_{pq}|$. In particular, if $\odd(c_{pq}) \ne 1$, then $\odd(|\wt V_{pq}|) \ne 1$.	
\end{proposition}

\begin{proof}
Let $E = \langle \zeta_{pq}, \Z[\lambda_{pq}]^\times \rangle \le \Z[\zeta_{pq}]^\times$. By \cite[Corollary 4.13]{Wa97}, $E$ has index two with $\Z[\zeta_{pq}]^\times/ E \cong \Z/2$ generated by $1 - \zeta_{pq}$. Let $\psi_p \colon \Z[\zeta_{pq}]^\times \to \F_p[\zeta_q]^\times$ be the map sending $\zeta_p \mapsto 1$. If $\alpha \in \Z[\lambda_{pq}]^\times$, then $\alpha = \sum_{i=0}^{\frac{pq-1}{2}} a_i (\zeta_{pq}^i + \zeta_{pq}^{-i})$ for some $a_i \in \Z$ and so 
\[\textstyle \psi_p(\alpha) = \sum_{i=0}^{\frac{q-1}{2}} \wt a_i (\zeta_q^i+\zeta_q^{-i}) \in \F_p[\lambda_q]^\times\] 
where $\wt a_i = \sum_{k=0}^{\frac{p-1}{2}} a_{i+kq}$. Hence the composition $\Z[\zeta_{pq}]^\times \xrightarrow[]{\psi_p} \F_p[\zeta_q]^\times \to \frac{\F_p[\zeta_q]^\times}{\F_p[\lambda_q]^\times}$ is trivial and so 
\[ \wt V_{pq} = \coker(\Psi_{pq}^+) = \coker\Big(\Z/pq \oplus \Z/2 \to \smfrac{\F_p[\zeta_q]^\times}{\F_p[\lambda_q]^\times} \oplus \smfrac{\F_q[\zeta_p]^\times}{\F_q[\lambda_p]^\times}\Big)\] 
where $1 \in \Z/pq$ maps to $\Psi_{pq}^+(\zeta_{pq})$ and $1 \in \Z/2$ maps to $\Psi_{pq}^+(1-\zeta_{pq})$. In particular, this implies that $|\wt V_{pq}|$ is divisible by $\frac{1}{2pq} \cdot | \frac{\F_p[\zeta_q]^\times}{\F_p[\lambda_q]^\times} | \cdot |  \frac{\F_q[\zeta_p]^\times}{\F_q[\lambda_p]^\times} |$. The result now follows from \cref{lemma:F_p[lambda]}.	
\end{proof}

\begin{proposition} \label{prop:2p-case}
Let $p$ be an odd prime. Let $f_2 = \ord_p(2)$ and $g_2 = (p-1)/2f_2$. Define
\[c_{2p} = \begin{cases} \frac{1}{p}(2^{\frac{f_2}{2}}+1)^{g_2}, & \text{if $f_2$ is even} \\ \frac{1}{p}(2^{f_2}-1)^{\frac{g_2}{2}}, & \text{if $f_2$ is odd.} \end{cases}\] 
Then $c_{2p} \mid |\wt V_{2p}|$. In particular, if $\odd(c_{2p}) \ne 1$, then $\odd(|\wt V_{2p}|) \ne 1$.	
\end{proposition}

\begin{proof}
Let $\pi_{2p} \colon \Z[\zeta_p]^\times \to \F_2[\zeta_p]^\times$ be reduction mod $2$. Since $\zeta_{2p}=-\zeta_p$ and $\zeta_2 = -1$, we have $\Psi_{2p} \colon \Z[\zeta_p]^\times \to \F_2[\zeta_p]^\times \oplus \F_p^\times$. It is shown in \cite[Theorem 50.14]{CR87} that projection induces an isomorphism $V_{2p} = \coker(\Psi_{2p}) \cong \coker(\pi_{2p})$. Similarly, if $\pi_{2p}^+ \colon \Z[\zeta_p]^\times \to \frac{\F_2[\zeta_p]^\times}{\F_2[\lambda_p]^\times}$, then $\wt V_{2p} = \coker(\Psi_{2p}^+) \cong \coker(\pi_{2p}^+)$.
By \cite[Corollary 4.13]{Wa97}, $\Z[\zeta_p]^\times = \langle \zeta_p, \Z[\lambda_p]^\times \rangle$. The same argument as for \cref{prop:pq-case} implies that $|\coker(\pi_{2p}^+)|$ is divisible by 
\[ \smfrac{1}{p} \cdot \Big| \smfrac{\F_p[\zeta_2]^\times}{\F_p[\lambda_2]^\times}\Big| \cdot \Big|  \smfrac{\F_2[\zeta_p]^\times}{\F_2[\lambda_p]^\times} \Big| = \smfrac{1}{p}  \cdot \Big|  \smfrac{\F_2[\zeta_p]^\times}{\F_2[\lambda_p]^\times} \Big|.\] 
The result now follows from \cref{lemma:F_p[lambda]}.	
\end{proof}

\section{Proof of main results on the involution on \texorpdfstring{$\wt{K}_0(\Z C_m)$}{projective class groups}} \label{s:proofs-algebra}

The aim of this section is to prove the following theorems. \cref{theorem:main12-red} coincides with \cref{theorem:main12-red-intro} and \cref{theorem:main3-red} is a more detailed version of \cref{theorem:main3-red-intro}.
As we saw previously, these theorems are required to prove \cref{theorem:main12,theorem:main3} respectively.

\begin{theorem} \label{theorem:main12-red}
Let $m \ge 2$ be a square-free integer. Then 
\begin{clist}{(i)}
\item\label{item:theorem:main12-red-i} $|\{ y \in \wt{K}_0(\Z C_m) \mid \overline{y}=-y\}|=1$ if and only if $m \in \{2, 3, 5, 6, 7, 10, 11, 13, 14, 17, 19\}$;
\item\label{item:theorem:main12-red-ii} $|\{ x - \overline{x} \mid x \in \wt{K}_0(\Z C_m)\}|=1$ if and only if $m \in \{2, 3, 5, 6, 7, 10, 11, 13, 14, 15, 17, 19, 29\}$; and
\item\label{item:theorem:main12-red-iii} 
$|\{ x - \overline{x} \mid x \in \wt{K}_0(\Z C_m)\}| \to \infty$ super-exponentially in $m$, and hence we also have that $|\{ y \in \wt{K}_0(\Z C_m) \mid \overline{y}=-y\}| \to \infty$ super-exponentially in $m$.
\end{clist}
\end{theorem}

\begin{theorem} \label{theorem:main3-red}
Let $m \ge 2$ be an integer. Then 
\begin{clist}{(i)}
\item\label{item:theorem:main3-red-ii} $|\{ y \in \wt{K}_0(\Z C_{p^n}) \mid \overline{y}=-y\}/\{ x - \overline{x} \mid x \in \wt{K}_0(\Z C_{p^n})\}|=1$  if $n \ge 1$ and $p \le 509$ is an odd prime such that $h_p$ is odd (e.g. $p=3$);
\item\label{item:theorem:main3-red-iii} 
The sequence $|\{ y \in \wt{K}_0(\Z C_{2^n}) \mid \overline{y}=-y\}/\{ x - \overline{x} \mid x \in \wt{K}_0(\Z C_{2^n})\}|$ is unbounded as $n \to \infty$.
Thus, $\displaystyle \sup_{k \le m} |\{ y \in \wt{K}_0(\Z C_k) \mid \overline{y}=-y\}/\{ x - \overline{x} \mid x \in \wt{K}_0(\Z C_k)\}| \to \infty$  exponentially in $m$.  
\end{clist}
\end{theorem}

\subsection{Proof of \cref{theorem:main12-red}} 
\label{ss:proofs-algebra-1/2}

In light of the fact that
\[\{ x - \overline{x} \mid x \in \wt{K}_0(\Z C_m)\} \le \{ x \in \wt{K}_0(\Z C_m) \mid \overline{x}=-x\},\] 
it suffices to prove the following four statements for square-free integers $m \ge 2$.
\begin{enumerate}
\item[(A1)]
$|\{ x - \overline{x} \mid x \in \wt{K}_0(\Z C_m)\}| \to \infty$ super-exponentially in $m$.
\item[(A2)]
If $m \not \in \{2, 3, 5, 6, 7, 10, 11, 13, 14, 15, 17, 19, 29\}$, then  $|\{ x - \overline{x} \mid x \in \wt{K}_0(\Z C_m)\}| \ne 1$.
\item[(A3)]
If $m \in \{2, 3, 5, 6, 7, 10, 11, 13, 14, 17, 19\}$, then  $|\{ x \in \wt{K}_0(\Z C_m) \mid \overline{x}=-x\}|=1$.
\item[(A4)]
If $m \in \{15,29\}$, then $|\{ x - \overline{x} \mid x \in \wt{K}_0(\Z C_m)\}| = 1$ and  $|\{ x \in \wt{K}_0(\Z C_m) \mid \overline{x}=-x\}|\ne 1$.
\end{enumerate}

To see this, note that (A1) implies \cref{theorem:main12-red} \eqref{item:theorem:main12-red-iii}.
The forwards direction of \cref{theorem:main12-red} \eqref{item:theorem:main12-red-iii} is implied by (A2) and (A4), and the backwards direction coincides with (A3). The forwards direction of \cref{theorem:main12-red} \eqref{item:theorem:main12-red-ii} coincides with (A2) and the backwards direction is implied by (A3) and (A4).

Recall from \cref{ss:class-groups-basics} that $\wt K_0(\Z C_m) \cong C(\Z C_m)$ is an isomorphism of $\Z C_2$-modules, where $C(\Z C_m)$ denotes the locally free class group. We therefore have the following short exact sequence of $\Z C_2$-modules established in \cref{ss:induced-inv}:
\[ 0 \to D(\Z C_m) \to \wt K_0(\Z C_m) \to \bigoplus_{d \mid m} C(\Z[\zeta_d])\to 0, \]
where $D(\Z C_m)$ has the induced involution, and each $C(\Z[\zeta_d])$ has the involution induced by conjugation. 
Each of statements (A1)--(A4) are proven via the following lemma.

\begin{lemma} \label{lemma:very-useful}
Let $m \ge 2$ be an integer. Then 
\begin{clist}{(i)}
\item\label{item:lemma-very-useful-i}
$|\{x-\overline{x} \mid x \in D(\Z C_m)\}| \le |\{x-\overline{x} \mid x \in \wt K_0(\Z C_m)\}|$;
\item\label{item:lemma-very-useful-ii}
$\prod_{d \mid m} |\{x-\overline{x} \mid x \in C(\Z[\zeta_d])\}| \le |\{x-\overline{x} \mid x \in \wt K_0(\Z C_m)\}|$; 
\item\label{item:lemma-very-useful-iii}
if $|\{x-\overline{x} \mid x \in D(\Z C_m)\}| \ne 1$ or $|\{x-\overline{x} \mid x \in C(\Z[\zeta_m])\}| \ne 1$, then  $|\{x-\overline{x} \mid x \in \wt K_0(\Z C_m)\}| \ne 1$ $($and so $|\{x \in \wt K_0(\Z C_m) \mid \overline{x}=-x\}| \ne 1)$; 
\item\label{item:lemma-very-useful-iv}
if $|\{x \in D(\Z C_m) \mid \overline{x}=-x\}|=1$ and  $\prod_{d \mid m} |\{x \in C(\Z[\zeta_d]) \mid \overline{x}=-x\}| = 1$, then  $|\{x \in \wt K_0(\Z C_m) \mid \overline{x}=-x\}| = 1$ $($and so $|\{x-\overline{x} \mid x \in \wt K_0(\Z C_m)\}| = 1)$.
\end{clist}
\end{lemma}

\begin{proof}
By \cref{lemma:useful}, there is an exact sequence
\begin{align*} 0 \to \{x \in D(\Z C_m) \mid \overline{x}=-x\} &\to \{x \in \wt K_0(\Z C_m) \mid \overline{x}=-x\} \\
& \qquad\qquad   \to  \bigoplus_{d \mid m} \{ x \in C(\Z[\zeta_d]) \mid x = - \ol{x} \}
\vspace{-5mm}
\end{align*}
as well as injective and surjective maps
\[ \{x-\overline{x} \mid x \in D(\Z C_m)\} \hookrightarrow \{x-\overline{x} \mid x \in \wt K_0(\Z C_m)\} \twoheadrightarrow \bigoplus_{d \mid m} \{ x -\ol{x} \mid x \in C(\Z[\zeta_d]) \}. \]
The second sequence implies \eqref{item:lemma-very-useful-i} and \eqref{item:lemma-very-useful-ii}, with \eqref{item:lemma-very-useful-iii} as a corollary. The first sequence implies \eqref{item:lemma-very-useful-iv}.
\end{proof}

We now proceed to prove each of statements (A1)-(A4). We begin with the following which, by \cref{lemma:very-useful}~\eqref{item:lemma-very-useful-ii}, implies (A1).

\begin{proposition}
We have that $|\{ x - \overline{x} \mid x \in C(\Z[\zeta_m])\}| \to \infty$ super-exponentially in $m$.
\end{proposition}

\begin{proof}
It was shown in \cref{prop:inv-ideal-class-1} that $\odd(h_m^{-})$ divides $|\{x-\ol{x} \mid x \in C(\Z[\zeta_m])\}|$. The result now follows from that fact that, by \cref{prop:odd(h_m)-bound}, $\odd(h_m^{-}) \to \infty$ super-exponentially in $m$.
\end{proof}

We now prove (A2). Our approach is to use \cref{lemma:very-useful}~\eqref{item:lemma-very-useful-iii}. In particular, we begin by classifying the $m \ge 2$ square-free for which $|\{x - \overline{x} \mid x \in C(\Z[\zeta_m])\}| = 1$. We then determine the subset of these values for which $|\{x - \overline{x} \mid x \in D(\Z C_m)\}| = 1$.

\begin{proposition} \label{prop:part1}
The complete list of $m \ge 2$ square-free for which $|\{x - \overline{x} \mid x \in C(\Z[\zeta_m])\}| = 1$ is as follows:
\[ m =  
\begin{cases}
\text{$p$}, & \text{where $p \in \{2, 3,5,7,11,13,17,19,29\}$}\\
\text{$2p$}, & \text{where $p \in \{3,5,7,11,13,17,19,29\}$} \\
\text{$pq$ or $2pq$}, & 	\text{where $(p,q) \in \{(3,5),(3,7),(3,11),(5,7),(3,13)\}$}.
\end{cases}
\]
\end{proposition}

\begin{proof}
Firstly, $h_m^{-}=1$ implies $|\{x - \overline{x} \mid x \in C(\Z[\zeta_m])\}| = 1$. On the other hand, it follows from \cref{prop:inv-ideal-class-1} that $\odd(h_m^{-}) \ne 1$ implies $|\{x - \overline{x} \mid x \in C(\Z[\zeta_m])\}| \ne 1$. 
In \cref{prop:odd(h_m)=1}~\eqref{item:prop-odd-hm-1-i}, it is shown that the $m \ge 2$ square-free for which $\odd(h_m^{-})=1$ and $h_m^{-} \ne 1$ are precisely the $m \in S$, where $S = \{29,39,58,65,78,130\}$. It remains to determine for which $m \in S$ we have $|\{x - \overline{x} \mid x \in C(\Z[\zeta_m])\}| = 1$.

Suppose $m \in S$. By \cite[p.~421]{Wa97}, we have $h_m^+=1$ and so $C(\Z[\zeta_m]) = C(\Z[\zeta_m])^{-}$ by \cref{lemma:C-} \eqref{item:lemma-C-ii}. In particular, we have:
\[ \{x - \overline{x} \mid x \in C(\Z[\zeta_m])\} = 2 \cdot C(\Z[\zeta_m]). \] 
If $m=29$ or $68$, then $C(\Z[\zeta_{m}]) \cong (\Z/2)^3$ and $2 \cdot C(\Z[\zeta_m]) = 0$~\cite[p.~412]{Wa97}.
If $m=39$ or $78$, then $h_m=2$ and $C(\Z[\zeta_m]) \cong \Z/2$ and $2 \cdot C(\Z[\zeta_m]) = 0$~\cite[p.~412]{Wa97}.
If $m=65$ or $130$, then \cite[Proposition~1~(iv)]{Ho93} gives that
$C(\Z[\zeta_{m}]) \cong (\Z/2)^2 \times (\Z/4)^2$
and $2 \cdot C(\Z[\zeta_m]) \cong (\Z/2)^2$.

Hence we have shown that, for $m \ge 2$ square-free, $|\{x - \overline{x} \mid x \in C(\Z[\zeta_m])\}| = 1$ if and only if $h_m^{-}=1$ or $m \in \{29,39,58,78\}$. The result now follows by \cref{prop:h_m=1}~\eqref{item:prop:h-m=1=i}.
\end{proof} 

We now prove the following. By \cref{lemma:very-useful}~\eqref{item:lemma-very-useful-iii}, this completes the proof of (A2).

\begin{proposition} \label{prop:part2}
The complete list of $m \ge 2$ square-free for which $|\{x - \overline{x} \mid x \in C(\Z[\zeta_m])\}| = 1$ and $|\{x - \overline{x} \mid x \in D(\Z C_m)\}| = 1$ is as follows:
\[ m =  
\begin{cases}
\text{$p$}, & \text{where $p \in \{2,3,5,7,11,13,17,19,29\}$} \\
\text{$2p$}, & 	\text{where $p \in \{3,5,7\}$}, \\
\text{$pq$}, & \text{where $(p,q) = (3,5)$}.
\end{cases}
\]
That is, $m \in \{2, 3, 5, 6, 7, 10, 11, 13, 14, 15, 17, 19, 29\}$.
\end{proposition}

Before turning to the proof, we begin by recalling that Cassou-Nogu\`{e}s determined the integers $m \ge 2$ for which $D(\Z C_m)=0$ \cite[Theor\`eme 1]{CN74} (see also \cite{CN72}). The following can be deduced by comparing this with \cref{prop:part1}.

\begin{lemma} \label{lemma:D=0}
Let $m \ge 2$ be a square-free integer such that $|\{x - \overline{x} \mid x \in C(\Z[\zeta_m])\}| = 1$.
Then $D(\Z C_m) = 0$ if and only if $m$ is prime or $m=2p$ where $p \in \{3,5,7\}$.
\end{lemma}

\begin{proof}[Proof of \cref{prop:part2}]
Let $m \in \{2, 3, 5, 6, 7, 10, 11, 13, 14, 15, 17, 19, 29\}$. If $m \ne 15$, then \cref{lemma:D=0} implies that $D(\Z C_m)=0$ and so $|\{x - \overline{x} \mid x \in D(\Z C_m)\}| = 1$. If $m = 15$, then it is shown in \cite[p.~48]{CN72} that $|D(\Z C_{15})|=2$. This implies that $D(\Z C_{15}) \cong \Z/2$ has the trivial action and so $|\{x - \overline{x} \mid x \in D(\Z C_{15})\}| = 1$.

By \cref{prop:part1}, the remaining $m \ge 2$ square-free for which $|\{x - \overline{x} \mid x \in C(\Z[\zeta_m])\}| = 1$ are:
\[ m =  
\begin{cases}
\text{$2p$}, & \text{where $p \in \{11,13,17,19,29\}$} \\
\text{$pq$}, & \text{where $(p,q) \in \{(3,7),(3,11),(5,7),(3,13)\}$} \\
\text{$2pq$}, & \text{where $(p,q) \in \{(3,5),(3,7),(3,11),(5,7),(3,13)\}$}.
\end{cases}
\]
By \cref{prop:D-conjugation}, we have that $\prod_{d \mid m} \odd(|\wt V_d|)$ divides $|\{x - \overline{x} \mid x \in D(\Z C_m)\}|$, where $\wt V_d$ is as defined in \cref{def:wt-V_m}. It therefore suffices to prove that, for each $m$ listed above, we have $\odd(|\wt V_d|) \ne 1$ for some $d \mid m$.

In the case $m=2p$, the bound $c_{2p}$ from \cref{prop:2p-case} is computed as in the following table. 
\[
\begin{tabular}{|c|ccccc|}
\hline
 $p$ & 11 & 13 & 17 & 19 & 29 \\
 \hline
 $c_{2p}$ & 3  & 5  & 17 & 27 & 565\\
\hline
\end{tabular}
\]
In each case, $\odd(c_{2p}) \ne 1$ and so $\odd(|\wt V_{2p}|) \ne 1$. 
Hence $|\{x - \overline{x} \mid x \in D(\Z C_m)\}| \ne 1$ for $m = 2p$ where $p \in \{11,13,17,19,29\}$.

In the case $m=pq$ for odd primes $p,q$, the bound $c_{pq}$ from \cref{prop:pq-case} is computed as in the following table. 
\[
\begin{tabular}{|c|ccccc|}
\hline
 $(p,q)$ & (3,5) &(3,7) & (3,11) & (5,7) & (3,13) \\
 \hline
 $c_{pq}$ & 2 & 4 & 44 & 90 & 104 \\
\hline
\end{tabular}
\]
In the cases $(p,q) \in \{(3,11), (5,7), (3,13) \}$, $\odd(c_{pq}) \ne 1$ and so $\odd(|\wt V_{pq}|) \ne 1$. Hence $|\{x - \overline{x} \mid x \in D(\Z C_m)\}| \ne 1$ for $m = pq$ or $2pq$ where $(p,q) \in \{(3,11), (5,7), (3,13) \}$.

We deal with the three remaining cases $m \in \{21, 30,42\}$ directly from the definition of $\wt V_m$:
\[ \wt V_m \cong  \coker\Big( \Psi_m^{+} \colon \Z[\zeta_m]^\times \to \bigoplus_{i=1}^n \smfrac{\F_{p_i}[\zeta_{m/p_i}]^\times}{\F_{p_i}[\lambda_{m/p_i}]^\times} \Big).  \]

First suppose that $m=30$. Then we have 
\[ \Psi_{30}^+ \colon \Z[\zeta_{15}]^\times \to \smfrac{\F_2[\zeta_{15}]^\times}{\F_2[\lambda_{15}]^\times} \oplus \smfrac{\F_3[\zeta_5]^\times}{\F_3[\lambda_5]^\times} \oplus \smfrac{\F_5[\zeta_3]^\times}{\F_5[\lambda_3]^\times}\] 
and $E = \langle \zeta_{15}, \Z[\lambda_{15}]^\times \rangle$ has index two in $\Z[\zeta_{15}]^\times$. By the same argument as in \cref{prop:pq-case}, we get that $|\wt V_{30}| = |\coker(\Psi_{30}^+)|$ is divisible by
\[ c_{30} = \smfrac{1}{30} \cdot \Big| \smfrac{\F_2[\zeta_{15}]^\times}{\F_2[\lambda_{15}]^\times} \Big| \cdot \Big| \smfrac{\F_3[\zeta_5]^\times}{\F_3[\lambda_5]^\times} \Big| \cdot \Big| \smfrac{\F_5[\zeta_3]^\times}{\F_5[\lambda_3]^\times} \Big| = c_{15} \cdot \Big| \smfrac{\F_2[\zeta_{15}]^\times}{\F_2[\lambda_{15}]^\times} \Big| = 2 \cdot (2^{\frac{f_2}{2}}+1)^{g_2} = 10 \] 
since $f_2 = \ord_{15}(2) = 4$ and $g_2 = 1$. Since $\odd(c_{30}) \ne 1$, this implies that $\odd(|\wt V_{30}|) \ne 1$. Hence $|\{x - \overline{x} \mid x \in D(\Z C_{15})\}| \ne 1$.

We deal with the remaining cases $m\in \{21,42\}$ by computing the involution on $\wt V_{21}$ explicitly. This turns out to be necessary since, by \cref{remark:D-computations}, we have $|D(\Z C_m)| = 4$ in each case and so $\odd(|\wt V_{21}|) = \odd(|\wt V_{42}|) = 1$. If $|\{x - \overline{x} \mid x \in D(\Z C_{m})\}|=1$ for $m \in \{21,42\}$ then, by \cref{prop:D-conjugation}, we would have $|2 \cdot \wt V_{21}| = 1$. Hence it suffices to prove that $|2 \cdot \wt V_{21}| \ne 1$.

We now claim that $\wt V_{21} = \coker(\Psi_{21}^+) \cong \Z/4$, which implies that $2 \cdot \wt V_{21} \cong \Z/2$. Firstly, by the proof of \cref{prop:pq-case}, we have that
\[ \coker(\Psi_{21}^+) \cong \coker\Big(\Z/21 \oplus \Z/2 \to \smfrac{\F_3[\zeta_7]^\times}{\F_3[\lambda_7]^\times} \oplus  \smfrac{\F_7[\zeta_3]^\times}{\F_7[\lambda_3]^\times}\Big)\] 
where $1 \in \Z/21$ maps to $\Psi_{21}^+(\zeta_{21})$ and $1 \in \Z/2$ maps to $\Psi_{21}^+(1-\zeta_{21})$.

Since $g_3=1$, $\F_3[\zeta_7] \cong \F_3(\zeta_7)$ is a field and $\F_3[\lambda_7] \cong \F_3(\lambda_7)$ is a subfield. 
This implies that $\frac{\F_3[\zeta_7]^\times}{\F_3[\lambda_7]^\times} \cong \frac{\Z/(3^6-1)}{\Z/(3^3-1)} \cong \Z/28$.
We have $\F_7[\zeta_3] = \Z[x]/\langle 7,1+x+x^3\rangle = \Z[x]/\langle 7,(x-2)(x-4)\rangle \cong \F_7 \times \F_7$ where $\zeta_3 \mapsto (2,4)$ and so $\F_7[\zeta_3]^\times \cong (\Z/6)^2$. Since $\F_7[\lambda_3]^\times = \F_7^\times \cong \Z/6$, this implies that $\frac{\F_7[\zeta_3]^\times}{\F_7[\lambda_3]^\times} \cong \Z/6$. Hence $D  :=  \frac{\F_3[\zeta_7]^\times}{\F_3[\lambda_7]^\times} \oplus  \frac{\F_7[\zeta_3]^\times}{\F_7[\lambda_3]^\times} \cong \Z/28 \oplus \Z/6$.

Now $\Psi_{21}^+(\zeta_{21}) = [(\zeta_7,\zeta_3)]$ where $[\zeta_7] \in \Z/28$ has order $7$ and $[\zeta_3] \in \Z/6$ has order $3$. This implies that $\coker(\Z/21 \to D) \cong \Z/4 \oplus \Z/2$. 
Note that $\Psi_{21}^+(1-\zeta_{21}) = [(1-\zeta_7,1-\zeta_3)]$. The isomorphism $\F_7[\zeta_3]^\times/\F_7[\lambda_3]^\times \to \Z/6$ sends $1-\zeta_3 \mapsto 1-2 = -1$ and so the image of $\Psi_{21}^+(1-\zeta_{21})$ in $\coker(\Z/21 \to D) \cong \Z/4 \oplus \Z/2$ has the form $(*,-1)$. Since it has order two, this implies that $\coker(\Psi_{21}^+) \cong \coker(\Z/21 \oplus \Z/2 \to D) \cong \Z/4$ as required.
\end{proof}

We now prove (A3). Our approach is to use \cref{lemma:very-useful}~\eqref{item:lemma-very-useful-iv}, and is analogous to our proof of (A2). We begin by classifying the $m \ge 2$ square-free for which $\prod_{d \mid m} |\{ x \in C(\Z[\zeta_d]) \mid x = - \ol{x}\}|=1$. We then determine the subset of these values for which $|\{ x \in D(\Z C_m) \mid x = - \ol{x}\}|=1$.

\begin{proposition} \label{prop:B(m)}
Let $m \ge 2$ be square-free. Then $\prod_{d \mid m} |\{ x \in C(\Z[\zeta_d]) \mid x = - \ol{x}\}|=1$ if and only if $h_m^{-}=1$.
\end{proposition}

\begin{proof}
If $h_m^{-}=1$, then $h_d^{-}=1$ for all $d \mid m$ and so $\prod_{d \mid m} |\{ x \in C(\Z[\zeta_d]) \mid x = - \ol{x}\}|=1$. Conversely, if $h_m^{-} \ne 1$ and $m \not \in \{29,39,58,78\}$, then \cref{prop:part1} implies that $|\{x - \overline{x} \mid x \in C(\Z[\zeta_m])\}| \ne 1$ and so $|\{ x \in C(\Z[\zeta_m]) \mid x = - \ol{x}\}|\ne 1$. It now suffices to show that, if $m \in \{29,39,58,78\}$, then $|\{ x \in C(\Z[\zeta_m]) \mid x = - \ol{x}\}|\ne 1$.

Suppose $m \in \{29,39,58,78\}$. By \cref{prop:part1}, we have that $C(\Z[\zeta_m]) = C(\Z[\zeta_m])^{-}$ and so
\[ \{ x \in C(\Z[\zeta_m]) \mid x = - \ol{x}\} = \{ x \in C(\Z[\zeta_m]) \mid 2x = 0 \}. \]
If $m=29$ or $58$, then $C(\Z[\zeta_m]) \cong (\Z/2)^3$ and so $\{ x \in C(\Z[\zeta_m]) \mid 2x = 0 \} \cong (\Z/2)^3$ \cite[p.~412]{Wa97}. If $m=39$ or $78$, then $C(\Z[\zeta_m]) \cong \Z/2$ and so $\{x \in C(\Z[\zeta_m]) \mid 2x=0\} \cong \Z/2$ \cite[p.~412]{Wa97}.
\end{proof}

We now prove the following. By \cref{lemma:very-useful}~\eqref{item:lemma-very-useful-iv}, this completes the proof of (A3).

\begin{proposition} \label{prop:A3}
The complete list of $m \ge 2$ square-free for which 
\[ \textstyle \prod_{d \mid m} |\{ x \in C(\Z[\zeta_d]) \mid x = - \ol{x}\}|=1 \quad \text{and} \quad |\{ x \in D(\Z C_m) \mid x = - \ol{x}\}|=1\] 
is as follows:
\vspace{-4.22mm}
\[ m =  
\begin{cases}
\text{$p$}, & \text{where $p \in \{2,3,5,7,11,13,17,19\}$} \\
\text{$2p$}, & 	\text{where $p \in \{3,5,7\}$}.
\end{cases}
\]
That is, $m \in \{2, 3, 5, 6, 7, 10, 11, 13, 14, 17, 19\}$.
\end{proposition}

\begin{proof}
If $|\{ x \in D(\Z C_m) \mid x = - \ol{x}\}|=1$, then $|\{ x-\ol{x} \mid x \in D(\Z C_m) \}|=1$. By \cref{prop:h_m=1,prop:part2,,prop:B(m)}, the $m \ge 2$ square-free for which $\prod_{d \mid m} |\{ x \in C(\Z[\zeta_d]) \mid x = - \ol{x}\}|=1$ and $|\{ x-\ol{x} \mid x \in D(\Z C_m)\}|=1$ are as follows:
\[ m =  
\begin{cases}
\text{$p$}, & \text{where $p \in \{2,3,5,7,11,13,17,19\}$} \\
\text{$2p$}, & 	\text{where $p \in \{3,5,7\}$}, \\
\text{$pq$}, & \text{where $(p,q) = (3,5)$}.
\end{cases}
\]
If $m = p$ for $p \le 17$ prime or $m=2p$ for $p \in \{3,5,7\}$, then \cref{lemma:D=0} implies that $D(\Z C_m)=0$ and so $|\{ x \in D(\Z C_m) \mid x = - \ol{x}\}|=1$. If $m =15$, then it is shown in \cite[p.~48]{CN72} that $|D(\Z C_{15})|=2$. This implies that $D(\Z C_{15}) \cong \Z/2$ has the trivial action and so $\{ x \in D(\Z C_m) \mid x = - \ol{x}\} \cong \Z/2$.
\end{proof}

Finally, we prove (A4). These results are implied by computations used in the proofs of \cref{prop:part2,prop:A3}, but we repeat them here for the convenience of the reader.

\begin{proposition} \label{prop:part3}
\mbox{}
\begin{clist}{(i)}
\item\label{item:prop-part-3-i}
$\{ x - \overline{x} \mid x \in \wt K_0(\Z C_{15})\} = 0$ and $\{x \in \wt K_0(\Z C_{15}) \mid x=-\overline{x}\} \cong \Z/2$.
\item\label{item:prop-part-3-ii}
$\{ x - \overline{x} \mid x \in \wt K_0(\Z C_{29})\} = 0$ and $\{x \in \wt K_0 (\Z C_{29}) \mid x=-\overline{x}\} \cong (\Z/2)^3$.
\end{clist}
\end{proposition}

\begin{proof}
\eqref{item:prop-part-3-i} By \cite[p.~412]{Wa97}, we have $h_{15}=1$ and so $\wt K_0(\Z C_{15}) \cong D(\Z C_{15})$.
It is shown in \cite[p.~48]{CN72} that $|D(\Z C_{15})|=2$ and so $D(\Z C_{15}) \cong \Z/2$ has the trivial involution. Hence we have $\{ x - \overline{x} \mid x \in \wt K_0(\Z C_{15})\} = 0$ and $\{x \in \wt K_0(\Z C_{15}) \mid x=-\overline{x}\} \cong \Z/2$.

\eqref{item:prop-part-3-ii} By \cref{lemma:D=0}, we have $D(\Z C_{29}) = 0$ and so $\wt K_0(\Z C_{29}) \cong C(\Z[\zeta_{29}])$. By \cite[p.~412]{Wa97}, we have that $C(\Z[\zeta_{29}]) \cong (\Z/2)^3$. We also have $C(\Z[\zeta_{29}])^+=1$ \cite[p.~412]{Wa97} and so, by \cref{lemma:C-}~\eqref{item:lemma-C-ii}, $C(\Z[\zeta_{29}]) = C(\Z[\zeta_{29}])^{-}$ and so has the trivial involution. Hence we have $\{ x - \overline{x} \mid x \in \wt K_0(\Z C_{29})\} = 0$ and $\{x \in \wt K_0 (\Z C_{29}) \mid x=-\overline{x}\} \cong (\Z/2)^3$.
\end{proof}

This completes the proofs of (A1)-(A4) and so, by the discussion at the start of this section, completes the proof of \cref{theorem:main12-red}.

\begin{remark} \label{remark:D-computations}
In order to minimise the possibility of errors, we computed $D(\Z C_m)$ for all relevant $m$ using the algorithm described in \cite{BB06} and implemented in Magma by Werner Bley.
We checked that $D(\Z C_m)=0$ for the $m$ listed in \cref{lemma:D=0}, we checked that all the bounds $c_m$ computed in the proof of \cref{prop:part2} divide $|D(\Z C_m)|$ and we computed $D(\Z C_{21}) \cong \Z/4$, $D(\Z C_{42}) \cong \Z/2$ and $D(\Z C_{15}) \cong \Z/2$. These computations are all consistent with the calculations above.
\end{remark}

\subsection{Proof of \cref{theorem:main3-red}} \label{ss:proofs-algebra-3}
For $m \ge 2$, let 
\[ A_m  := \frac{\{x \in \wt{K}_0(\Z C_m) \mid x = -\ol{x}\}}{\{x-\ol{x} \mid x \in \wt{K}_0(\Z C_m)\}} \cong \wh H^{1}(C_2; \wt{K}_0(\Z C_m))\] 
where the isomorphism comes from \cref{prop:tate-C2}.
Similarly to the proof of \cref{theorem:main12-red}, we begin by noting that it now suffices to prove the following two statements. 
\begin{enumerate}
\item[(B1)]
If $n \ge 1$, then $|A_{p^n}| = 1$ if $p \le 509$ is an odd prime such that $h_p$ is odd.
\item[(B2)]
If $n \ge 1$, then $|A_{2^n}| \cdot |A_{2^{n+1}}| \ge 2^{2^{n-2}-1}$.
\end{enumerate}

To see this, note that (B1) directly implies \cref{theorem:main3-red}~\eqref{item:theorem:main3-red-ii} since $h_3 = 1$ is odd (see, for example, \cref{prop:h_m=1}). Next, \cref{theorem:main3-red}~\eqref{item:theorem:main3-red-iii} follows from (B2) since it implies that
\begin{align*} 
\sup_{k \le m} |A_{2^k}| &\ge \max\{|A_{2^k}|,|A_{2^{k-1}}|\} \ge \sqrt{|A_{2^k}| \cdot |A_{2^{k-1}}|}   \\
&\ge 2^{2^{k-3}-1} \ge 2^{2^{m-3}-1}
\end{align*}
which tends to infinity exponentially in $m$.

We again make use of the following short exact sequence of $\Z C_2$-modules (see \cref{ss:induced-inv}):
\[ 0 \to D(\Z C_m) \to C(\Z C_m) \to \bigoplus_{d \mid m} C(\Z[\zeta_d])\to 0, \]
where $D(\Z C_m)$ has the induced involution, and each $C(\Z[\zeta_d])$ has the involution induced by conjugation. 
Each of statements (B1) and (B2) are proven via the following lemma.

\begin{lemma} \label{lemma:A_m}
Let $m \ge 2$. Then there is a $6$-periodic exact sequence of finite abelian groups
\[
\begin{tikzcd}[column sep=tiny]
\wh H^1(C_2;D(\Z C_m)_{2}) \ar[r] & A_m \ar[r] & \bigoplus_{d \mid m} \wh H^1(C_2; C(\Z[\zeta_d])_{2}) \ar[d,"\partial"] \\
\bigoplus_{d \mid m} \wh H^0(C_2; C(\Z[\zeta_d])_{2}) \ar[u,"\partial"] & \wh H^0(C_2;C(\Z C_m)_{2}) \ar[l] & \wh H^0(C_2;D(\Z C_m)_{2}). \ar[l] 
\end{tikzcd}
\]
Furthermore, we have that:
\begin{clist}{(i)}
\item\label{item:lemma-A-m-i}
If $h_m$ is odd, then $A_m \cong \wh H^1(C_2;D(\Z C_m)_{2})$
\item\label{item:lemma-A-m-ii}
If $|D(\Z C_m)|$ is odd, then $A_m \cong \bigoplus_{d \mid m} \wh H^1(C_2;C(\Z[\zeta_d])_{2})$.
\item\label{item:lemma-A-m-iii}
If $h_m$ and $|D(\Z C_m)|$ are both odd, then $A_m = 0$.
\end{clist}
\end{lemma}

\begin{proof}
By \cref{prop:tate-2-part}, we have that $A_m \cong \wh H^1(C_2;C(\Z C_m)_{2})$.
The short exact sequence stated above induces a short exact sequence on their $2$-primary submodules:
\[ 0 \to D(\Z C_m)_{2} \to C(\Z C_m)_{2} \to \bigoplus_{d \mid m} C(\Z[\zeta_d])_{2} \to 0. \] 
This follows since, for example, localisation is an exact functor.
The existence of the required 6-periodic exact sequence now follows from \cref{prop:hexagon} and the fact that 
\[\wh H^n(C_2;\bigoplus_{d \mid m} C(\Z[\zeta_d])_{2}) \cong \bigoplus_{d \mid m} \wh H^n(C_2; C(\Z[\zeta_d])_{2})\] 
for $n \in \Z$ by \cref{prop:tate-LES}~\eqref{item-prop:tate-LES-ii}. 

To prove \eqref{item:lemma-A-m-i}, suppose $h_m$ is odd. Then $h_d$ is odd for all $d \mid m$ since $h_d \mid h_m$ \cite[p.~205]{Wa97}. This implies that $C(\Z[\zeta_d])_{2} = 0$ for all $d \mid m$ and so $\bigoplus_{d \mid m} \wh H^n(C_2; C(\Z[\zeta_d])_{2})=0$ for all $n \in \Z$. The 6-periodic exact sequence then gives that $A_m \cong \wh H^1(C_2;D(\Z C_m)_{2})$.

To prove \eqref{item:lemma-A-m-ii}, suppose $|D(\Z C_m)|$ is odd. Then $D(\Z C_m)_{2}=0$,  $\wh H^n(C_2;D(\Z C_m)_{2}) = 0$ for all $n$, so the 6-periodic exact sequence gives that $A_m \cong \bigoplus_{d \mid m} \wh H^1(C_2; C(\Z[\zeta_d])_{2})$.
If $h_m$ and $|D(\Z C_m)|$ are odd, then $D(\Z C_m)_{2}=0$ and so $A_m \cong \wh H^1(C_2;D(\Z C_m)_{2}) = 0$, which proves~\eqref{item:lemma-A-m-iii}. 
\end{proof}

We now prove (B1).

\begin{proposition} \label{prop:A_p}
    Let $n \ge 1$ and let $p \le 509$ be an odd prime such that $h_p$ is odd. Then $|A_{p^n}|=1$.
\end{proposition}

\begin{remark}
By \cref{lemma:h_p} ~\eqref{item:lemma-h-p-i}, this condition holds precisely for the odd primes $p \le 509$ with
\[ p \not \in \{29, 113, 163, 197, 239, 277, 311, 337, 349, 373, 397, 421, 463, 491 \}. \]
\end{remark}

\begin{proof}
Since $p$ is an odd prime, \cref{prop:D-odd} implies that $|D(\Z C_{p^n})|$ is odd. Since $p \le 509$ and $h_p$ is odd, \cref{lemma:h_p}~\eqref{item:lemma-h-p-ii} implies that $h_{p^n}$ is odd. Hence $|A_{p^n}|=1$ by \cref{lemma:A_m} \eqref{item:lemma-A-m-iii}.
\end{proof}

We now prove (B2). By the discussion above, this completes the proof of \cref{theorem:main3-red}.

\begin{proposition} \label{prop:A_2}
If $n \ge 1$, then $|A_{2^n}| \cdot |A_{2^{n+1}}| \ge 2^{2^{n-2}-1}$.
\end{proposition}

\begin{proof}
By \cref{lemma:h_p} \eqref{item:lemma-h-p-ii}, $h_{2^n}$ is odd and so \cref{lemma:A_m}~\eqref{item:lemma-A-m-i} implies that
\[ A_{2^n} \cong \wh H^1(C_2;D(\Z C_{2^n})_{2}). \]
By \cref{prop:hexagon,prop:KM}, we have a $6$-periodic exact sequence of finite abelian groups:
    \[
\begin{tikzcd}
\wh H^1(C_2;V_{2^{n+1}}) \ar[r] & A_{2^{n+1}} \ar[r] & A_{2^n} \ar[d,"\partial"] \\
\wh H^0(C_2; D(\Z C_{2^n})) \ar[u,"\partial"] & \wh H^0(C_2;D(\Z C_{2^{n+1}})) \ar[l] & \wh H^0(C_2;V_{2^{n+1}}). \ar[l] 
\end{tikzcd}
\]
Furthermore, by \cref{prop:tate-herbrand}, we have that $\wh H^0(C_2; D(\Z C_{2^k})) \cong A_{2^k}$ as abelian groups for all $k \ge 1$, and $\wh H^1(C_2;V_{2^{n+1}}) \cong \wh H^0(C_2;V_{2^{n+1}})$ as abelian groups.

Since the involution on $V_{2^{n+1}}$ acts by negation, we have that 
\[ \wh H^1(C_2;V_{2^{n+1}}) \cong V_{2^{n+1}}/2V_{2^{n+1}} \cong \bigoplus_{i=1}^{n-2} (\Z/2)^{2^{n-i-2}} \cong (\Z/2)^{2^{n-2}-1}. \]
This implies that the 6-periodic exact sequence restricts to:
\[ A_{2^n} \xrightarrow[]{\alpha} (\Z/2)^{2^{n-2}-1} \xrightarrow[]{\beta} A_{2^{n+1}}. \]
By the first isomorphism theorem and exactness, we get that:
\[ 
2^{2^{n-2}-1} = |\ker(\beta)| \cdot |\im(\beta)| = |\im(\alpha)| \cdot |\im(\beta)| \le |A_{2^n}| \cdot |A_{2^{n+1}}| \]
which was the required bound.
\end{proof}

\addtocontents{toc}{\protect\addvspace{2em}}

\def\MR#1{}
\bibliography{biblio}
\end{document}